\RequirePackage{etoolbox,tikz,pgfplots}
\usepgfplotslibrary{fillbetween}
\patchcmd{\bibliographystyle}{#1}{apalike}{}{}

\documentclass[sn-mathphys]{sn-jnl}
\jyear{2023}%

\usepackage{thmtools}

\newtheorem{theorem}{Theorem}[section]

\newtheorem{lemma}[theorem]{Lemma}
\newtheorem{definition}[theorem]{Definition}
\newtheorem{assumption}[theorem]{Assumption}

\renewcommand{\phi}{\varphi}

\allowdisplaybreaks

\usepackage{mathtools,enumitem}
\usepackage[capitalise,nameinlink,noabbrev]{cleveref}
\usepackage[nolist,printonlyused,withpage]{acronym}
\usepackage[binary-units=true,exponent-product = \cdot, output-product = \cdot]{siunitx}
\usepackage{bm,comment}

\numberwithin{equation}{section}
\crefname{subsection}{Subsection}{Subsections}
\crefformat{equation}{(#2#1#3)}
\crefformat{align}{(#2#1#3)}
\crefformat{enumi}{(#2#1#3)}
\crefname{figure}{Figure}{Figures}

\crefname{theorem}{Theorem}{Theorems}
\crefname{definition}{Definition}{Definitions}
\crefname{lemma}{Lemma}{Lemmas}
\crefname{assumption}{Assumption}{Assumptions}

\newcommand{\pt}{\partial_t}
\newcommand{\eps}{\varepsilon}

\newcommand{\ds}{\,\textup{d}s}
\newcommand{\dx}{\,\textup{d}x}
\newcommand{\dt}{\,\textup{d}t}
\newcommand{\dW}{\,\textup{d}W}

\newcommand{\bbE}{\mathbb{E}}
\newcommand{\bbN}{\mathbb{N}}
\newcommand{\bbP}{\mathbb{P}}
\newcommand{\bbR}{\mathbb{R}}
\newcommand{\calE}{\mathcal{E}}
\newcommand{\calF}{\mathcal{F}}
\newcommand{\calL}{\mathcal{L}}
\newcommand{\calP}{\mathcal{P}}

\renewcommand{\d}{{\mathrm d}}
\renewcommand{\div}{\operatorname{div}}
\renewcommand{\epsilon}{\varepsilon}

\def\norm #1{\left\|#1\right\|}

\let\originalleft\left
\let\originalright\right
\renewcommand{\left}{\mathopen{}\mathclose\bgroup\originalleft}
\renewcommand{\right}{\aftergroup\egroup\originalright}

\newcommand{\phil}{\phi_\lambda}
\newcommand{\philn}{\phi_{\lambda,n}}
\newcommand{\mul}{\mu_\lambda}
\newcommand{\muln}{\mu_{\lambda,n}}
\newcommand{\sigmal}{\sigma_\lambda}
\newcommand{\sigmaln}{\sigma_{\lambda,n}}

\newcommand{\E}{\bbE}

\renewcommand{\rho}{\varrho}

\renewcommand{\P}{\bbP}
\renewcommand{\R}{\bbR}

\renewcommand{\tilde}{\widetilde}
\renewcommand{\hat}{\widehat}

\newcommand{|}{\vert}

\newcommand{\hphi}{\hat\phi}
\newcommand{\hmu}{\hat\mu}
\newcommand{\hsigma}{\hat\sigma}
\newcommand{\hW}{\hat W}

\newcommand{\Psil}{\Psi_{\!\lambda}}

\newcommand{\tphil}{\tilde\phil}
\newcommand{\tmul}{\tilde\mul}

\newcommand{\tsigmal}{\tilde\sigmal}
\newcommand{\hphil}{\hat\phil}

\newcommand{\hsigmal}{\hat\sigmal}
\newcommand{\tsigmaln}{\tilde\sigmaln}
\newcommand{\tW}{\tilde W}

\newcommand{\sups}{\sup_{\mathclap{s\in[0,t]}}\,}
\usepackage[misc]{ifsym}
\mathcode`\;="603B
\mathcode`\,="613B
\newcolumntype{M}[1]{>{\centering\arraybackslash}m{#1}}

\makeatletter  
\renewcommand\@biblabel[1]{#1.} 
\makeatother

\begin{document}
\title[A stochastic Cahn--Hilliard system modeling tumor growth]{Analysis and computations of a
stochastic Cahn--Hilliard model for tumor growth
with chemotaxis and variable mobility}
    \author{\fnm{Marvin} \sur{Fritz}$^1$}\email{marvin.fritz@ricam.oeaw.ac.at}
    \affil{$^1$\orgdiv{Computational Methods for PDEs}, \orgname{Johann Radon Institute for Computational and Applied Mathematics}, \orgaddress{\city{Linz}, \country{Austria}}}
        \author{\fnm{Luca} \sur{Scarpa}$^2$}\email{luca.scarpa@polimi.it}
    \affil{$^2$\orgdiv{Department of Mathematics}, \orgname{Politecnico di Milano}, \orgaddress{\city{Milano}, \country{Italy}}}
	
	\abstract{
		In this work, we present and analyze a system of PDEs, which models tumor growth by taking into account chemotaxis, active transport, and random effects. The stochasticity of the system is modelled by random initial data and Wiener noises that appear in the tumor and nutrient equations. The volume fraction of the tumor is governed by a stochastic phase-field equation of Cahn--Hilliard type, and the mass density of the nutrients is modelled by a stochastic reaction-diffusion equation. We allow a variable mobility function and non-increasing growth functions, such as logistic and Gompertzian growth. Via approximation and stochastic compactness arguments, we prove the existence of a probabilistic weak solution and, in the case of a constant mobilities, the well-posedness of the model in the strong probabilistic sense. Lastly, we propose a numerical approximation based on the Galerkin finite element method in space and the semi-implicit Euler--Maruyama scheme in time. We illustrate the effects of the stochastic forcings in the tumor growth in several numerical simulations.}


    \maketitle
	\newpage

\section{Introduction} 
\label{Sec:intro}
Cancer is among the main global causes of death.
According to \cite{sung2021global}, there were 19.3 million new cancer diagnoses and 9.96 million cancer-related deaths worldwide.
By 2040, the yearly number of new cancer cases is projected to reach 30.2 million, with 16.3 million fatalities attributable to cancer.
Each tumor is distinct and dependent on a variety of characteristics.
There is no guaranteed procedure for curing cancer, nor is its cause entirely known.
Utilizing mathematical models to precisely depict tumor progression is the primary objective of mathematical oncology. 

The mathematical research on tumor growth dynamics has been extensively developed in the last decades. One of the most frequent approaches to describe tumor dynamics is based 
on the so-called diffuse-interface theory, for which we refer, e.g., to 
\cite{CL10, GLNS, HZO, WLFC08}.

In this work, we consider a diffuse-interface model for tumor growth, which takes into account precise biological effects due to both chemotaxis and also possible random perturbations affecting the evolution. More precisely, 
we consider a colony of tumor cells in an open bounded domain $D \subset \mathbb{R}^3$, e.g., representing an organ, and
we focus on phenomenological characterizations to capture
mesoscale and macroscale events around the tumor's evolution in the time interval $[0,T]$ for some fixed final time $T>0$.
The difference of volume fractions between tumor and healthy cells is described by the field 
$\phi:[0,T]\times D\to[0,1]$,
with $\{\phi=1\}$ representing unmixed tumor tissue, and $\{\phi=0\}$ representing the surrounding healthy tissue. The diffuse-interface approach is based on the constitutive assumption that the 
tumor and healthy phases are separated by a 
narrow diffuse interface 
$\{0<\varphi<1\}$ in which the phase-variable may also assume intermediate values.
The further constitutive assumption is that the proliferation of tumor cells takes place by absorption of some nutrient (e.g.~glucose)
which the cells are provided with through capillarity.
The local nutrient concentration is represented by a field
$\sigma:[0,T]\times D\to[0,1]$. 

\subsection{Deterministic model}

Following \cite{lima2015analysis}, the constituents $\phi$ and $\sigma$ are governed by the extended mass balance laws
$$\begin{aligned}
\pt \phi  &= - \div J_\phi + \Gamma_\phi, \\
\pt \sigma  &= - \div J_\sigma + \Gamma_\sigma,
\end{aligned}$$
where $J_\phi$, $J_\sigma$ are mass fluxes and $\Gamma_\phi$, $\Gamma_\sigma$ are source functions. Typically, sink-and-source-type functions are assumed, i.e., it holds $\Gamma_\sigma=-\Gamma_\phi$ and consequently, the system is referred to as closed or conserved. However, 
in order to consider more general biological effects, 
such as apoptosis (i.e.~the natural cell death of tumor cells),
we proceed as in \cite{garcke2016cahn} and define the 
source functions as: 
\[
\Gamma_\phi(\phi,\sigma) =(\beta\sigma-\alpha)f(\phi), \qquad 
\Gamma_\sigma(\phi,\sigma) =-\delta \sigma f(\phi).
\]
Here, $\alpha$ denotes the apoptosis rate, $\beta$ the tumor proliferation rate, 
$\delta$ the nutrient consumption rate. In the particular case of $\beta=\delta$ and $\alpha=0$, one is in the scenario of sink and source, i.e., the total mass of tumor and nutrients is conserved.
The function $f$ is assumed to be non-negative and bounded on
the relevant domain $[0,1]$, and normalized in the sense $f(0)=f(1)=0$. This corresponds to the absence of nutrient uptake by the tumor mass if there is either no tumor $\{\phi=0\}$ anyway or if the tumor is already saturated in the sense $\{\phi=1\}$. A relevant choice is the logistic growth function $f(\phi)=\phi(1-\phi)$, see \cite{fritz2019local,fritz2019unsteady,fritz2021analysis,fritz2021modeling}, or the Gompertz function $f(\phi)=\phi\log(1/\phi)$ as in \cite{wagner2023phasefield}. 

The mass fluxes are related to the underlying energy of the system, and we choose the following extended Ginzburg--Landau energy:
$$\mathcal{E}(\phi,\sigma)=\int_\Omega \Big[ \frac{\eps^2}{2} |\nabla \phi|^2 + \Psi(\phi) + \frac{1}{2} |\sigma|^2 - \chi \phi \sigma \Big] \dx.$$
The parameters $\eps>0$ and $\chi \geq 0$  denote the interfacial width and the chemotaxis factor. Chemotaxis describes an adhesion force between the tumor cells and the nutrients. Moreover, $\Psi$ describes a double-well potential and often one considers the regular function $\Psi(\phi)=\frac14 \phi^2(1-\phi)^2$ as an approximation to the physically realistic Flory--Huggins potential, see \cite{cherfils2011cahn}.
We follow Gurtin's approach in \cite{gurtin1996generalized} and employ a scaled mass flux for the phase-field equation, i.e., we propose 
\[
J_\phi = -m_1(\phi) \nabla \delta_\phi \mathcal{E}(\phi,\sigma),
\qquad
J_\sigma = -m_2(\sigma)\nabla \delta_\sigma \mathcal{E}(\phi,\sigma),
\]
where $\delta_\phi \mathcal{E}$ denotes the first variation of $\mathcal{E}$ with respect to $\phi$. Here, $m_i$, $i \in \{1,2\}$, denote mobility functions that are assumed to be non-degenerative, such as $m_1(\phi)=\phi^2(1-\phi)^2+M$ for some arbitrarily small constant $M>0$.  Moreover, $m_2$ corresponds to the diffusivity of the nutrients.

By calculating the first variations of the Ginzburg--Landau energy, 
we obtain the system
\begin{equation} \label{Eq:Model:Det}
\begin{aligned}
  \pt \phi  ={}& \div(m_1(\phi)(\nabla \mu-\chi\nabla\sigma)) + \Gamma_\phi(\phi,\sigma),
  \\
  \mu ={}& \Psi'(\phi) -\eps^2 \Delta\phi , \\
  \pt \sigma   ={}&  \div(m_2(\sigma)(\nabla\sigma - \chi \nabla \phi))  + \Gamma_\sigma(\phi,\sigma) .
\end{aligned}
\end{equation}
This is usually complemented with no-flux boundary conditions for $\varphi$, $\mu$, and $\sigma$, as well as suitable initial data $\varphi_0$ and $\sigma_0$. 

The deterministic model was studied  in \cite{garcke2017well,garcke2017analysis} regarding the existence of weak solutions, in \cite{garcke2020long} regarding the long-time behavior, in \cite{kahle2020parameter,garcke2018optimal} regarding an optimal control problem, and in \cite{garcke2022numerical,hawkins2012numerical} regarding numerical approximations. Other biological phenomena were included in the works \cite{fritz2019local,fritz2021subdiffusive,fritz2021analysis,fritz2023tumor} such as ECM degradation, angiogenesis and subdiffusive behavior.

\subsection{Stochastic model}
The class of deterministic models for tumor growth is certainly 
widely employed and provides a more-than-adequate description of several biological mechanisms. However, as intuitive as the tumor-nutrient balance equations may seem, it is well-established that tumor growth may undergo erratic behaviors that one cannot predict by merely using deterministic models. It is the case, for example, of metastases, whose activation seems to depend on random biological signals (see \cite{TanChen98}), and creation of capillary networks and angiogenesis 
(see \cite{NDetAL05}). Moreover, random perturbations are evident in models accounting for therapeutic treatment in terms of therapy uncertainty or parameter identification problems. These considerations 
inevitably call for a switching to classes of models that are capable 
of capturing such randomness in the underlying process:
in the scientific literature, models of tumor growth that account for unpredictability have been considered in \cite{albano2006stochastic,badri2016optimal,speer1984stochastic,ferrante2000parameter,liotta1976stochastic}. 

Generally, stochastic partial differential equations (SPDEs) are the mathematical tools used to model physical systems subjected to the influence of internal, external or environmental noises. As noted in \cite{liu2013local,liu2010spde}, such stochastic systems can also be used to describe models that are too complex to be described deterministically. Therefore, one employs a statistical strategy to manage the complex free energy defining the biological features of the system in the phase-field equation. To accomplish this, a noise is added to the phase-field equation, thus considering any microscopical fluctuations that affect the evolution of the phase parameter. Moreover, the phase-field equation aims at metastable round shapes, but real-world tumor masses are not perfectly round. In a stochastic model, we introduce a type of uncertainty that results in an irregular shape for the tumor mass that is biologically more realistic.
Lastly, as noted in \cite{orrieri2020optimal}, one can mimic the consequences of angiogenesis by incorporating a multiplicative noise into the reaction-diffusion equation. This sort of stochastic forcing is related to the oxygen received by malignant cells; as a result, its contribution to the overall tumor growth process may be increased.

By taking these considerations into account, we further extend the deterministic model \eqref{Eq:Model:Det} by introducing stochastic components to the system. 
We define the independent cylindrical Wiener (Gaussian) processes
$W_1$, $W_2$ on the separable 
Hilbert spaces $U_1$, $U_2$, respectively, and end up with the stochastic system
\begin{equation} \label{Def:System}
\begin{aligned}
  \d \phi  ={}& \div(m_1(\phi) (\nabla \mu- \chi \nabla\sigma))\dt + (\beta \sigma-\alpha) f(\phi) \dt+ G_1(\phi)\dW_1, \\
  \mu ={}& \Psi'(\phi) -\eps^2 \Delta\phi,   \\
  \d\sigma   ={}& \div(m_2(\sigma)(\nabla \sigma - \chi \nabla \phi)) \dt -\delta \sigma f(\phi) \dt + G_2(\sigma)\dW_2, 
\end{aligned}
\end{equation}
where $G_1$ and $G_2$ are operators that will be specified below.
We equip this system with the classical 
Neumann boundary data $\partial_n \phi  = \partial_n\mu  = \partial_n\sigma  = 0$ on $\partial D$ and 
the initial data $\phi(0) = \phi_0, \sigma(0)=\sigma_0$.

The mathematical literature on the stochastic 
Cahn--Hilliard equation is quite developed, e.g., see the articles 
\cite{cui2022wellposedness,debussche2011stochastic,scarpa2018stochastic,scarpa2021degenerate,scarpa2021stochastic}. We also quote the works \cite{feireisl2019stability,feireisl2021diffuse,deugoue2021splitting,medjo2017existence}
for studies on stochastic diffuse interface models that are coupled to the Navier--Stokes equations.
The work \cite{orrieri2020optimal} studied a stochastic Cahn--Hilliard equation with a coupled reaction-diffusion system and besides its well-posedness an optimal control problem was investigated. However, the typical features of a tumor growth model such as chemotaxis, flow, active transport, and logistic growth were not considered.

\subsection{Goals, novelties, and structure of the paper}
The main goal of the present work is twofold: we focus on the 
stochastic system, both in terms of mathematical analysis and numerical simulations.

More precisely, we first show that the stochastic system admits a probabilistic-weak solution in a very general setting, including non-constant mobilities and positive chemotaxis coefficients.
Up to our knowledge, this is the first available result in the literature of stochastic models for tumor growth dealing with 
both random perturbations and non-constant mobilities.
The analysis is carried out in a wide generality, allowing the 
double-well potential to have polynomial growth at infinity, whose
order depends on the presence of chemotaxis and the fact that the mobility $m_1$ is constant or not. Let us also stress that the presence of non-constant mobilities and the coupling between 
the Cahn--Hilliard equation and the reaction-diffusion equation 
prevent from framing the analysis in any well-established variational 
theory based on monotonicity.
This calls for an ad-hoc study of the system, for which we employ 
a two-level approximation scheme: a Galerkin-type approximation on the functional spaces and a Yosida-type approximation on the nonlinearity.

Secondly, we prove that in the more classical case of constant mobilities the stochastic system is well-posed also in the strong probabilistic sense: this is based on a Yamada--Watanabe argument 
and comes down to showing pathwise uniqueness of solutions. Again, the
issue in non-trivial due to the singular (non-Lipschitz)
behavior of the 
double-well potential in the Cahn--Hilliard equation.

Eventually, we propose a numerical scheme to 
approximate solutions of the stochastic system, and we provide 
several numerical simulations. These explicitly show
the effect of noise on the tumor dynamics, according to
a span of noise intensities, and confirm a more accurate description 
of the cell aggregation phenomenon.

The work is structured as follows. In \cref{Sec:Prelim} 
we give some mathematical preliminaries such as required compactness results and important inequalities.  Moreover, we introduce the function spaces that are typically used in the variational analysis of stochastic PDEs. Well-known results in stochastic analysis such as Prokhorov's theorem and the Burkholder--Davis--Gundy inequality are shortly described.  In \cref{Sec:Analysis}, we prove the existence of probabilistic weak solutions to the stochastic system. Moreover, in the case of constant mobilities, we prove existence and pathwise uniqueness of strong solutions to the stochastic system. In \cref{Sec:Numerical},  we propose a fully discrete scheme for approximating the stochastic model.  We show some selected numerical simulations to highlight the influence of the stochastic components.

\section{Mathematical notation and preliminaries}
\label{Sec:Prelim}
In this section we introduce the main mathematical setting of the work, the functinal spaces that will be used, and we recall some useful 
preliminaries on infinite-dimensioanl stochastic analsysis.

In the following, let $D\subset\bbR^d$, $d \in \{2,3\}$, 
be a bounded domain with sufficiently smooth boundary $\partial D$. Furthermore, let $T>0$ be a finite time horizon and the space-time cylinders are denoted by $D_t=(0,t)\times D$, for every $t\in(0,T]$.
We simply write $(\cdot,\cdot)_D$ and 
$(\cdot,\cdot)_{D_t}$ to indicate the inner products in $L^2(D)$
and in $L^2(D_t)$, respectively, for every $t\in(0,T]$.
We also use the classical symbol $(\cdot)_D$ for the space average on $D$.

\subsection{Functional analysis}
Let $X$ be an arbitrary Banach space. We denote its dual by $X'$ and the duality pairing between $X$ and its dual is written as $\langle \cdot,\cdot \rangle_X$.  If $X$ is a Hilbert space, then the scalar product is denoted by $(\cdot,\cdot)_X$. In the case of a further Hilbert space $Y$, we denote the space of Hilbert--Schmidt operators from $X$ to $Y$ by $\calL^2(X,Y)$ consisting of operators $T\in \calL(X,Y)$ that satisfy $\|T\|_{\calL^2(X,Y)}^2=\sum_{k \in \bbN} \|Te_k\|^2_Y<\infty$ for an orthonormal basis $\{e_k\}_k$ of $X$.  
Analogously, we denote by $\mathcal L^1(X,Y)$ the space of trace-class operators from $X$ to $Y$.

We set the Hilbert spaces
\[
  H:=L^2(D), \qquad V:=H^1(D), \qquad
  V_2=\{u\in H^2(D):\partial_{n} u = 0 \text{ a.e.~on } \partial D\},
\]
endowed with their natural norms, and we identify $H$ with its dual 
through the Riesz isomorphism. We recall that this ensures that 
\[
  V_2\subset V \subset H \subset V' \subset V_2',
\]
where all inclusions are dense and compact.

\subsection{Stochastic calculus} \label{Subsec:Stochastic}
Let 
$(\Omega, \calF, (\calF_t)_{t}, \bbP)$ be a filtered probability space. We denote the progressive $\sigma$-algebra on $\Omega\times[0,T]$  by $\calP$, i.e., the collection of subsets of $\Omega\times[0,T]$ with the indicator function being a progressively measurable process, see \cite[p.314]{rogers2000diffusions}.

Let $X$ denote an arbitrary Banach space and $T>0$. The strongly measurable Bochner-integrable functions
	on $\Omega$ and $(0,T)$ of order $p,q\in [1,\infty]$ are denoted by 
	 $L^p(\Omega;X)$ and $L^q(0,T;X)$, respectively, with the usual 
  convention that $p=0$ and $q=0$ denote the space of strongly measurable functions on $\Omega$ and $(0,T)$, respectively.
	In the case of progressive measurability, we use the notation
	$L^p_\calP(\Omega;L^q(0,T;X))$. 
	If $X$ is separable and reflexive, we introduce an analogous notation for weakly-$*$ measurable functions
$$\begin{aligned}
	L^p_w(\Omega; L^\infty(0,T;X^*))&=
	\{ v\!:\!\Omega\!\to\! L^\infty(0,T; X^*) \text{ weakly-$*$ meas.}: 
	\mathbb E \|v\|_{L^\infty(0,T; X^*)}^p\!<\!\infty
	\} \\
 &=(L^{p/(p-1)}(\Omega; L^1(0,T; X)))^*,\end{aligned}$$ 
 where the second equality holds due to \cite[Thm.~8.20.3]{edwards2012functional}.

We introduce the independent cylindrical Wiener processes $W_i=(W_i(t))_{t\geq 0}$, $i \in \{1,2\}$, on the separable 
Hilbert spaces $U_i$. We further equip the spaces with the ONBs $\{u_n^i\}_n$ for $i \in \{1,2\}$ i.e. it holds $W_i(t)=\sum_{k=1}^\infty \beta^i_k(t) u^i_k$, $i \in \{1,2\}$, for $\{\beta^i_k\}_k$ being a family of i.i.d. $\calF_t$-Brownian motions.  
Let us recall that the series are formal, and they converge 
in every Hilbert-Schmidt extension $\tilde U_i$ of $U_i$, $i=1,2$:
in particular, $W_i$ is a well-defined process with continuous trajectories in $\tilde U_i$.
Moreover, if $K$ is a Hilbert space, 
for $i=1,2$ the stochastic integral
$\int_0^\cdot G(s)\dW_i(s)$
is a well-defined continuous $K$-valued process
for every $G\in L^2_\calP(\Omega; L^2(0,T; \mathcal L(U_i,K)))$.
We further recall that the Burkholder--Davis--Gundy inequality, see \cite[Proposition 2.3.8]{breit2018stochastically}, states
\begin{equation} \label{Lem:BurkholderDavisGundy}
\mathbb{E} \sup_{t \in[0,T]} \bigg\| \int_0^t G(s) \dW_i(s) \bigg\|_K^p \leq C_p \mathbb{E} \left( \int_0^T \left\| G(s) \right\|_{L^2(\mathcal{U}, K)}^2  \ds \right)^{p/2}
\end{equation}
for every $G\in L^2_\calP(\Omega; L^2(0,T; \mathcal L(U_i,K)))$.

It is well-known that if $X$ is a Polish space (i.e., a separable completely metrizable topological space), then every probability measure on $X$ is tight. Furthermore, by Prokhorov's theorem, see \cite[Theorem 2.6.1]{breit2018stochastically}, a collection of probability measures on $X$ is tight if and only if it is relatively weakly compact. Together with the Skorokhod theorem, this links the concept of weak convergence
of a probability measure with that of almost sure convergence of random variables.
Indeed, Skorokhod theorem states the existence of $X$-valued random variables $(U_n)_n$ on some probability space such that the law of $(U_n)_n$ is equal to a given weakly converging sequence of probability measures on $X$, and it holds $U_n(\omega) \to U_0(\omega)$ in $X$ almost surely.
This result can be generalized to the class of sub-Polish spaces, which are topological spaces that are not necessarily metrizable but retain several important properties of Polish spaces; then the result is referred to as Jakubowksi--Skorokhod theorem, see \cite[Theorem 2.7.1]{breit2018stochastically}.

\subsection{It\^o's formula}
Let $i\in\{1,2\}$ be fixed, let $K$ be a Hilbert space, let
$G\in L^2_\calP(\Omega; L^2(0,T; \mathcal L(U_i,K)))$, let
$f\in L^2_\calP(\Omega; L^1(0,T; K))$, and let $\phi_0$ be an $\mathcal{F}_0$-measurable $K$-valued random variable. Then the process
\begin{equation} \label{Eq:StochInt}
\phi(t) = \phi_0 + \int_0^t f(s) \ds + \int_0^t G(s) \dW_i(s)
\end{equation}
is a well-defined continuous $K$-valued process.
If the function $F \colon K \to \mathbb{R}$ is twice Fr\'echet-differentiable and its derivatives $DF$, $D^2F$ are uniformly continuous on bounded subsets of $K$, then the It\^o formula, see \cite[Theorem 2.4.1]{breit2018stochastically}, reads 
\begin{equation} \label{Eq:ItoFormula}
\begin{aligned}
F(\phi(t)) &= F(\phi_0) + \int_0^t \langle DF(\phi(s)), f(s) \rangle \ds + \int_0^t \langle DF(\phi(s)), G(s) \dW(s) \rangle \\ &\quad+ 2 \int_0^t \text{Tr}\big(G(s)^*D^2F(\phi(s))G(s)\big) \ds,
\end{aligned}
\end{equation}
for all $t \in [0,T]$, $\bbP$-almost surely, with the trace being defined as $\text{Tr}(A) = \sum_{k=1}^\infty \langle A u_k^i, u_k^i \rangle$ for every $A \in \calL^1(U,U)$.

\section{Mathematical analysis of the stochastic model} \label{Sec:Analysis}
The subsequent assumptions will be in order throughout the paper. 
\begin{assumption}\label{Assumption} Let the following assumptions hold: \vspace{.1cm}
	\begin{enumerate}[label=\textup{(A\arabic*)}, ref=A\arabic*, leftmargin=.9cm] \itemsep.2em
  \item   $\alpha, \beta, \delta, \chi\geq 0$ and $\eps>0$ are fixed;\label{Ass:Con}
  \item $f\in C^{0,1}(\bbR;[0,1])$; \label{Ass:Fun}
  \item \label{Ass:Mob} $m_1,m_2 \in C^0(\bbR;[m_0,m_\infty])$ for $0<m_0\leq m_\infty<\infty$;
  \item \label{Ass:Psi} $\Psi\in C^2(\bbR;\bbR_{\geq 0})$ with $\Psi'(0)=0$ satisfies the growth conditions 
  \[
  |\Psi'(x)|^{q}+|\Psi''(x)|
  \leq C_\Psi(1+|\Psi(x)|)
  \]
  and 
  \[
  \Psi''(x)\geq -C_\Psi
  \]
  for any $x \in \bbR$, for some $C_\Psi>0$,
  where 
  \[\begin{cases}
  q>1 \quad&\text{if either $\chi=0$ or $m_1$ is constant},\\
  q=2 \quad&\text{if both $\chi>0$ and $m_1$ is non-constant};
  \end{cases}
  \]
  \item \label{Ass:Noi} for $i\in\{1,2\}$, 
  $G_i:H \to \calL^2(U_i,H)$ is measurable and there is a sequence $(g_{i,k})_k \subset W^{1,\infty}(\R)$ such that
  $G_i(v)u_k=g_{i,k}(v)$ for any  $v\in H$ and $k \in \bbN$, with $C_{G_i}:=\sum_{k=0}^\infty \|g_{i,k}\|_{W^{1,\infty}(\R)}^2 < \infty$.
     \item \label{Ass:Ini} $\phi_0$, $\sigma_0$ satisfy
     $\phi_0 \in V$ with $\Psi(\phi_0)\in L^1(D)$ and
  $\sigma_0 \in H$.
\end{enumerate}
\end{assumption}

  Let us comment on the assumptions above. Foremost, we note that assumption \eqref{Ass:Fun} allows both the Gompertz and logistic growth functions, as discussed in \cref{Sec:intro}. Moreover,  
  \eqref{Ass:Psi} prescribes different growth conditions on $\Psi$, depending
  on $\chi$ and $m_1$. In  particular, we note that in the case where either 
  no chemotaxis is present or the first mobility is constant, the 
  assumption $q>1$ allows every polynomial double-well potential of any growth.
  In the pathological case with both chemotaxis and a variable mobility in the Cahn--Hilliard equation, the situation is much more delicate and the assumption on the potential is quite strict, prescribing double-well structure with 
  at most quadratic growth  at infinity. This restriction is due
  to the strong coupling in the system: indeed, 
  if $\chi>0$ one cannot rely on a maximum-principle argument for $\sigma$,
  and this results in the presence of a non-bounded proliferation term in the Cahn--Hilliard equation, which cannot be treated unless using an ad-hoc technique
  working in the constant mobility case.
  Assumption \eqref{Ass:Mob} allows both constant mobilities and positive mobilities, such as $m(\phi)=1_{[0,1]} \phi^2(1-\phi)^2+m_0$ for some small $m_0>0$. The assumption \cref{Ass:Noi} on the operators $G_i$ is widely used in literature and ensures the Lipschitz-continuity and the linear boundedness of the operators.
  The assumptions \eqref{Ass:Ini} on the initial data are restricted  to the nonrandom case. However, we note that it can be extended to the random case, as done for example in \cite{scarpa2021stochastic} for the stochastic Cahn--Hilliard equation.

\begin{definition}[Martingale solution]
\label{Def:Weak} We call $(\hat\Omega,\hat\calF,(\hat\calF_t)_t,
\hat\P,\hat W_1, \hat W_2,\hat\phi,\hat\mu,\hat\sigma)$ a martingale solution to \eqref{Def:System} if $(\hat\Omega,\hat\calF,(\hat\calF_t)_t,\hat\P)$ is a filtered probability space satisfying the usual conditions, 
$\hat W_1$ and $\hat W_2$ are cylindrical Wiener processes on $U_1$ and $U_2$, respectively, and
$$\begin{aligned}
    \hphi &\in L^0_{\mathcal P}(\hat\Omega;C^0([0,T];H)\cap L^2(0,T;V_2)), \\
    \hmu &\in L^0_{\mathcal P}(\hat\Omega;L^2(0,T;V)), \\
    \hsigma &\in L^0_{\mathcal P}(\hat\Omega;C^0([0,T];H)
    \cap L^2(0,T;V)),
    \end{aligned}$$
    are such that $\hmu=-\eps^2\Delta\hphi+\Psi'(\hphi)$ and
   \begin{equation} \label{Def:SystemVariational} \begin{aligned}
        (\hphi(t),v)_D+(m_1(\hphi)\nabla(\hmu-\chi\hsigma),\nabla v)_{D_t}&=(\phi_0,v)_D +(\beta\hsigma-\alpha,f(\hphi)v)_{D_t} \\&\quad+(\textstyle\int_0^t G_1(\hphi(s)) \d\hW_1(s),v)_D \\
        (\hsigma(t),v)_D
        +(m_2(\hsigma)\nabla(\hsigma-\chi\hphi),\nabla v)_{D_t}&=(\sigma_0,v)_D 
        - (\delta\hsigma,f(\hphi)v)_{D_t} \\&\quad + (\textstyle\int_0^t G_2(\hsigma(s)) \d\hW_2(s),v)_D
    \end{aligned}\end{equation}
  for every $v\in V$, for every $t\in[0,T]$, $\hat\bbP$-almost surely.
\end{definition}

\begin{definition}[Strong solution]
\label{Def:Strong} We call $(\phi,\mu,\sigma)$ a (probabilistically-) strong solution to \eqref{Def:System} if 
$$\begin{aligned}
    \phi &\in L^0_{\mathcal P}(\Omega;C^0([0,T];H)\cap L^2(0,T;V_2)), \\
    \mu &\in L^0_{\mathcal P}(\Omega;L^2(0,T;V)), \\
    \sigma &\in L^0_{\mathcal P}(\Omega;C^0([0,T];H)
    \cap L^2(0,T;V)),
    \end{aligned}$$
    are such that $\mu=-\eps^2\Delta\phi+\Psi'(\phi)$ and
   \begin{equation} \label{Def:SystemVariational2} \begin{aligned}
        (\phi(t),v)_D+(m_1(\hphi)\nabla(\mu-\chi\sigma),\nabla v)_{D_t}&=(\phi_0,v)_D +(\beta\sigma-\alpha,f(\phi)v)_{D_t} \\&\quad+(\textstyle\int_0^t G_1(\phi(s)) \d W_1(s),v)_D \\
        (\hsigma(t),v)_D
        +(m_2(\sigma)\nabla(\sigma-\chi\phi),\nabla v)_{D_t}&=(\sigma_0,v)_D 
        - (\delta\sigma,f(\phi)v)_{D_t} \\&\quad + (\textstyle\int_0^t G_2(\sigma(s)) \d W_2(s),v)_D
    \end{aligned}\end{equation}
  for every $v\in V$, for every $t\in[0,T]$, $\bbP$-almost surely.
\end{definition}

The existence of a martingale solution to the stochastic tumor growth system is stated in the next theorem. This is the main result of this work.

\begin{theorem}[Existence of martingale solutions]
  \label{Theorem:Main}
  Let \cref{Assumption} hold. Then there exists a martingale solution $(\hat\Omega,\hat\calF,(\hat\calF_t)_t,\hat\P,\hat W_1, \hat W_2,\hat\phi,\hat\mu,\hat\sigma)$
  to the stochastic tumor system in the sense of \cref{Def:Weak} with regularity
  $$\begin{aligned} \hphi &\in L^\ell(\hat\Omega;C^0([0,T];H)) \cap L^\ell_w(\hat\Omega;L^\infty(0,T;V)) \cap L^{\ell/2}(\hat\Omega;L^2(0,T;V_2)), \\
  \hmu &\in L^{\ell/2}(\hat\Omega;L^2(0,T;V)) \cap L^\ell(\hat\Omega;L^2(0,T;H)), \\
  \hsigma &\in L^\ell(\hat\Omega;C^0([0,T];H)) \cap L^{\ell/2}(\hat\Omega;L^2(0,T;V)),
  \end{aligned}$$
  for every $\ell\geq2$, 
  and there exists a constant $c$, depending only on the data
  $D,T,\alpha,\beta,\delta,\chi, \eps, C_\Psi, q,
  m_0, m_\infty, C_{G_1}, C_{G_2}$, such that 
  the following energy inequality holds:
 \begin{equation} \label{Eq:FinalEnergyEstimate}
\begin{aligned}
	&\hat\E\sup_{s\in[0,T]}\norm{\hphi(s)}_V^2
	+\hat\E\sup_{s\in[0,T]}\norm{\Psi(\hphi(s))}_{L^1(D)}
	+\hat\E
	\|\nabla\hmu\|_{L^2(0,T;H)}^2
 +\hat\E\|\hmu\|_{L^2(0,T; H)}
 \\&\quad +\hat\E\sup_{s\in[0,T]}\norm{\hsigma(s)}_H^2 + 
 \hat\E\|\hsigma\|_{L^2(0,T;V)}^2\\
	&\leq c \Big(\norm{\phi_0}_V^2+\norm{\Psi(\phi_0)}_{L^1(D)} +\norm{\sigma_0}_H^2 \Big).
\end{aligned}\end{equation}
\end{theorem}

In the special case where the mobilities are constant and the 
potential is the classical fourth-order polynomial potential, 
we show that pathwise uniqueness holds and that the system is well-posed
in the strong probabilistic sense on the original probability space 
$(\Omega,\mathcal F, \P)$ and with respect to the original Wiener
processes $W_1$ and $W_2$.

\begin{theorem}[Strong well-posedness]
  \label{Theorem:Main2}
  Let \cref{Assumption} hold, assume that the mobilities 
  $m_1, m_2$ are constant and that 
  \[
  |\Psi''(x)|^2\leq C_\Psi(1+|\Psi(x)|) \quad\forall\,x\in\mathbb R.
  \] 
  Then there exists a unique strong solution $(\phi,\mu,\sigma)$
  to the stochastic tumor system in the sense of \cref{Def:Strong}
  with regularity
  $$\begin{aligned} 
  \phi &\in L^\ell(\Omega;C^0([0,T];H)) \cap L^\ell_w(\Omega;L^\infty(0,T;V)) \cap L^{\ell/2}(\Omega;L^2(0,T;V_2)), \\
  \mu &\in L^{\ell/2}(\Omega;L^2(0,T;V)) \cap L^\ell(\Omega;L^2(0,T;H)), \\
  \sigma &\in L^\ell(\Omega;C^0([0,T];H)) \cap L^{\ell/2}(\Omega;L^2(0,T;V)).
  \end{aligned}$$
  for every $\ell\geq2$,
  and there exists a constant $c$, depending only on the data
  $D,T,\alpha,\beta,\delta,\chi, \eps, C_\Psi, q,
  m_0, m_\infty, C_{G_1}, C_{G_2}$, such that 
  the following energy inequality holds:
 \begin{equation} \label{Eq:FinalEnergyEstimate2}
\begin{aligned}
	&\E\sup_{s\in[0,T]}\norm{\phi(s)}_V^2
	+\E\sup_{s\in[0,T]}\norm{\Psi(\phi(s))}_{L^1(D)}
	+\E
	\|\nabla\mu\|_{L^2(0,T;H)}^2
 +\E\|\mu\|_{L^2(0,T; H)}
 \\&\quad +\E\sup_{s\in[0,T]}\norm{\sigma(s)}_H^2 + 
 \E\|\sigma\|_{L^2(0,T;V)}^2\\
	&\leq c \Big(\norm{\phi_0}_V^2+\norm{\Psi(\phi_0)}_{L^1(D)} +\norm{\sigma_0}_H^2 \Big).
\end{aligned}\end{equation}
  Moreover, for every strong solution $(\phi_1, \mu_1, \sigma_1),(\phi_2,\mu_2,\sigma_2)$ associated to some initial data 
  $(\phi_0^1, \sigma_0^1), (\phi_0^2, \sigma_2)$ satisfying \cref{Ass:Ini},
  there exist two sequences of constants $\{c_n\}_n$ and of stopping times
  $\{\tau_n\}_n$, such that $\tau_n\nearrow T$ almost surely and
  the following continuous dependence holds:
   \begin{equation} \label{Eq:FinalContinuous}
\begin{aligned}
	&\E\sup_{s\in[0,\tau_n]}\norm{(\phi_1-\phi_2)(s)}_H^2+
    \E\|\phi_1-\phi_2\|_{L^2(0,\tau_n; V_2)}^2\\
	&\qquad+\E\sup_{s\in[0,\tau_n]}\norm{(\sigma_1-\sigma_2)(s)}_H^2
   +\E\|\sigma_1-\sigma_2\|_{L^2(0,\tau_n;V)}^2\\
	&\leq c_n \Big(\norm{\phi_0^1-\phi_0^2}_H^2
 +\norm{\sigma_0^1-\sigma_0^2}_H^2 \Big) \quad\forall\,n\in\mathbb N.
\end{aligned}\end{equation}
\end{theorem}

The proof of \cref{Theorem:Main}
is structured in several steps based on intermediate auxiliary results. \smallskip

\noindent{\bf Step 1:} We introduce an approximated problem, where we consider a smooth Yosida-type regularization of the nonlinearity $\Psi'$, called $\Psil'$ for some arbitrary but fixed $\lambda>0$. Our main goal lies in showing the existence of a martingale solution $(\tphil,\tmul,\tsigmal)$ to this $\lambda$-approximated problem, deriving $\lambda$-uniform estimates and passing to the limit $\lambda \to 0$. However, to achieve the existence of a martingale solution, we have to first consider a Galerkin approximation, depending on some further parameter $n$, to the $\lambda$-approximation. The well-posedness of the Galerkin system is discussed in this step.  \smallskip

\noindent{\bf Step 2:} The goal in this step is deriving $n$-uniform estimates of the Galerkin approximation.  \smallskip  

\noindent{\bf Step 3:} Having derived $n$-uniform estimate, we can pass to the limit $n \to \infty$. This ensures the existence of a solution $(\tphil,\tmul,\tsigmal)$ to the $\lambda$-approximation. \smallskip

\noindent{\bf Step 4:} We derive $\lambda$-uniform estimates in suitable spaces.  \smallskip

\noindent{\bf Step 5:} Through the theorems of Prokhorov and Skorokhod, we show here the existence of martingale solutions. \medskip

The proof of \cref{Theorem:Main2} relies instead on proving pathwise
uniqueness for the system, which yields also strong-existence by a classical  
argument \`a la Yamada--Watanabe. The main point comes down then to proving the  continuous dependence \eqref{Eq:FinalContinuous}.\medskip

In the following we denote by $c$ a generic constant depending solely 
on the structure of the problem introduced above. If the constant 
depends on some approximation parameter such as $\lambda$ or $n$, 
we use the symbols $c_\lambda$ and $c_n$, respectively.

\subsection{Step 1: Yosida and Galerkin approximations}  \label{Sec:Step1:AppEx}
By \eqref{Ass:Psi}, it holds $\Psi''(r) \geq - C_\Psi$ for any $r \in \bbR$ and thus, the function $\gamma(r):=\Psi'(r) + C_\Psi r$
is nondecreasing and continuous. We can identify  $\gamma$  with a maximal monotone graph in $\bbR\times\bbR$ and it satisfies $\gamma(0)=0$.
For any $\lambda>0$, we define the Yosida approximation of $\gamma$, see \cite[p.99]{barbu2010nonlinear}, by
$\gamma_\lambda(r):=\frac{1}{\lambda} (r- (I+\lambda \gamma)^{-1}(r))$ for any $r \in \bbR$.
We recall that $\gamma_\lambda$ is $\frac{1}{\lambda}$-Lipschitz continuous, and it holds $\gamma_\lambda(r) \to \gamma(r)$ as $\lambda \downarrow 0$.
Then, we define the approximated double-well potential $\Psil:\bbR \to \bbR$ by
$
\Psil(r):=\Psi(0)- \frac{C_\Psi}{2} r^2+\int_0^r \gamma_\lambda(s) \ds.
$
In particular, it holds $\Psil'(r)=\gamma_\lambda(r)-C_\Psi r$ for any $r \in \bbR$ and we conclude that $\Psil'$ is $\max\{\frac{1}{\lambda},C_\Psi\}$-Lipschitz continuous. Moreover, it holds
$|\Psil(r)| \leq C_\lambda(1+|r|^2)$ for any $r \in \bbR$
for some constant $C_\lambda$ possibly depending on $\lambda$.

Consequently, we consider the following Yosida approximation of \cref{Def:System}:
  \begin{equation} \label{Def:SystemYosida}
\begin{aligned}  \d \phil &=    \div(m_1(\phil) (\nabla \mul- \chi\nabla \sigmal)) \dt + (\beta \sigmal-\alpha) f(\phil)  \dt + G_1(\phil)\dW_1, \\
  \mul &= \Psil'(\phil) -\eps^2 \Delta\phil ,\\
  \d\sigmal &=   \div(m_2(\sigmal) (\nabla\sigmal-\chi \nabla \phil)) \dt - \delta \sigmal f(\phil)\dt + G_2(\sigmal)\dW_2, 
\end{aligned}
\end{equation}
equipped with the initial values
$\phil(0) = \phi_0$, $\sigmal(0)=\sigma_0$ in $D$ and the boundary data
$0=\partial_n \phil = \partial_n \mul = \partial_n \sigmal$ on
$(0,T)\times \partial D$.
As described before, we want to show the existence of a solution to the $\lambda$-approximated system. This is stated in the following lemma.
The concept of martingale solution to the approximated problem 
\eqref{Def:SystemYosida} is analogous, {\em mutatis mutandis}, 
to the one given in Definition~\ref{Def:Weak} for the original system.

\begin{lemma} \label{Lem:ExistenceLambda}
	There exists a martingale solution $(\tilde\Omega,\tilde\calF,\tilde\calF_t,\tilde\P,\tW_1, \tW_2,\tphil,\tmul,\tsigmal)$ to \eqref{Def:SystemYosida} in the sense that
$$\begin{aligned}
    \tphil &\in L^2_{\mathcal P}(\hat\Omega;C^0([0,T];H)\cap L^2(0,T;V_2)), \\
    \tmul &\in L^2_{\mathcal P}(\hat\Omega;L^2(0,T;V)), \\
    \tsigmal &\in L^2_{\mathcal P}(\hat\Omega;C^0([0,T];H)\cap L^2(0,T;V)),
    \end{aligned}$$
    satisfy $\tmul=-\eps^2\Delta\tphil + \Psi'_\lambda(\tphil)$ and
   $$ \begin{aligned}
        (\tphil(t),v)_D+
        (m_1(\tphil)\nabla(\tmul-\chi\tsigmal),\nabla v)_{D_t}&=(\phi_0,v)_D 
        +(\beta\tsigmal-\alpha,f(\tphil)v)_{D_t} \\ &\quad +(\textstyle\int_0^t G_1(\tphil(s)) \d\tW_1(s),v)_D \\
        (\tsigmal(t),v)_D
        +(m_2(\tsigmal)\nabla(\tsigmal-\chi\tphil),\nabla v)_{D_t}&=(\sigma_0,v)_D 
        - (\delta\tsigmal,f(\tphil)v)_{D_t} \\&\quad + (\textstyle\int_0^t G_2(\tsigmal(s)) \d\tW_2(s),v)_D
    \end{aligned}$$
  for every $v\in V$, for every $t\in[0,T]$, $\tilde\bbP$-almost surely.
\end{lemma}

The proof of the lemma is shifted to the end of \Cref{Sec:Step3:LimN} as it requires a further approximation on the system, which we now introduce.
Let $(e_j)_{j\in\bbN}\subset V_2$ and $(\alpha_j)_{j\in\bbN} \subset \R$ 
be sequences of eigenfunctions and eigenvalues of the 
negative Laplace operator with homogeneous Neumann conditions, respectively, i.e.
 $ -\Delta e_j = \alpha_j e_j$ in $D$ with
  $\nabla e_j \cdot n_{\partial D} = 0$ on $\partial D$.
We normalize, so that $(e_j)_j$ is an orthonormal system of $H$
and an orthogonal system in $V$.
For every $n\in\bbN$, we define the finite
dimensional space $H_n:=\operatorname{span}\{e_1,\ldots,e_n\}\subset V_2$,
endowed with the $\norm{\cdot}_H$-norm. 

Let now $i\in\{1,2\}$. We define the approximated operator $G_{i,n}:H_n\to\calL^2(U_i,H_n)$  as
$G_{i,n}(v)u_k^i:=\sum_{j=1}^n(G_i(v)u_k^i, e_j)_H e_j$ for any 
$v\in H_n$ and
$k\in\bbN$.
One can check that $G_{i,n}$ is well-defined: indeed, for every $v\in H_n$ 
and every $n\in\bbN$, thanks to \eqref{Ass:Noi} we have $G_{i,n}(v)\in\calL^2(U,H_n)$ since
\begin{equation}\begin{aligned}\label{est_Gn}
  \norm{G_{i,n}(v)}_{\calL^2(U_i,H_n)}^2
  =\sum_{k=0}^\infty\norm{G_{i,n}(v)u_k^i}_H^2 &= 
  \sum_{k=0}^\infty\sum_{j=1}^n|(G_i(v)u_k^i, e_j)_H|^2 \\&\leq
  \sum_{k=0}^\infty\sum_{j=1}^\infty|(G_i(v)u_k^i, e_j)_H|^2\\ 
  &=\sum_{k=0}^\infty\norm{G_i(v)u_k^i}_H^2\\
  &= \norm{G_i(v)}^2_{\calL^2(U_i,H)}
  \leq C_{G_i}|D|.
\end{aligned}\end{equation}
A similar computation shows that $G_{i,n}$ is $C_{G_i}$-Lipschitz-continuous
from $H_n$ to $\calL^2(U_i,H_n)$. Indeed, for every 
for every $v_1,v_2\in H_n$ 
and every $n\in\bbN$ one has thanks to \eqref{Ass:Noi} that
\begin{equation}\begin{aligned}\label{est_Gn2}
  \norm{G_{i,n}(v_1)-G_{i,n}(v_2}_{\calL^2(U_i,H_n)}^2
  &=\sum_{k=0}^\infty\norm{G_{i,n}(v_1)u_k^i-G_{i,n}(v_2)u_k^i}_H^2 \\
  &= 
  \sum_{k=0}^\infty\sum_{j=1}^n
  |(G_i(v_1)u_k^i-G_i(v_2)u_k^i, e_j)_H|^2 \\
  &\leq
  \sum_{k=0}^\infty\sum_{j=1}^\infty
  |(G_i(v_1)u_k^i-G_i(v_2)u_k^i, e_j)_H|^2\\ 
  &=\sum_{k=0}^\infty\norm{G_i(v_1)u_k^i-G_i(v_2)u_k^i}_H^2\\
  &= \norm{G_i(v_1)-G_i(v_2)}^2_{\calL^2(U_i,H)}
  \leq C_{G_i}\|v_1-v_2\|_H^2.
\end{aligned}\end{equation}
Note that the estimates \eqref{est_Gn}-\eqref{est_Gn2} are independent of the parameter $n$.

Using typical approximations via mollifiers, 
for $i\in\{1,2\}$, the mobility function $m_i$ is approximated by $m_{i,n}$ satisfying $m_0\leq m_{i,n}(r)\leq m_\infty$ for any $r \in \bbR$ and
$m_{i,n}\to m_i$ in $C^0([a,b])$ as $n \to \infty$ for every $[a,b]\subset\mathbb R$.
As for the initial data, we define naturally
$\phi_0^n:=\sum_{j=1}^n(\phi_0, e_j)_He_j\in H_n$ and  $\sigma_0^n:=\sum_{j=1}^n(\sigma_0, e_j)_He_j\in H_n$.

Summarizing, we consider the Yosida system's Galerkin approximation
  \begin{equation} \label{Def:SystemGalerkin}
\begin{aligned}  \d \philn  &=  \div(m_{1,n}(\philn) (\nabla \muln- \chi\nabla  \sigmaln )) \dt + (\beta \sigmaln-\alpha) f(\philn)  \dt \\ &\quad+ G_{1,n}(\philn)\dW_1, \\
  \muln&=\Psil'(\philn) -\eps^2 \Delta\philn  ,\\
  \d\sigmaln &= \div(m_{2,n}(\sigmaln) (\nabla\sigmaln-\chi \nabla \philn)) \dt - \delta \sigmaln f(\philn) \dt \\ &\quad+ G_{2,n}(\sigmaln)\dW_2, 
\end{aligned}
\end{equation}
equipped with the initial values $\philn(0)=\phi_0^n$ and $\sigmaln(0)=\sigma_0^n$ in $D$, as well as no-flux boundary conditions.
The variational formulation of \eqref{Def:SystemGalerkin} reads
\begin{equation} \label{Def:SystemGalerkinVariational}\begin{aligned} 
	(\philn(t),h_n)_D
	&+ (m_{1,n}(\philn)\nabla(\muln-\chi\sigmaln),\nabla h_n)_{D_t}\\
 &\quad =(\phi_0^n,h_n)_D + (\beta \sigmaln - \alpha,f(\philn) h_n)_{D_t} \\
	 &\qquad +
	(\textstyle\int_0^tG_{1,n}(\philn(s))\dW_1(s),h_n)_D \\
	(\sigmaln(t),h_n)_D
		&+(m_{2,n}(\sigmaln)\nabla(\sigmaln-\chi\philn),\nabla h_n)_{D_t}\\
  &\quad=(\sigma_0^n,h_n)_D
		- \delta (\sigmaln,f(\philn) h_n)_{D_t} \\
		&\qquad   +
		(\textstyle\int_0^tG_{2,n}(\philn(s))\dW_2(s),h_n)_D,
\end{aligned}
\end{equation}
for any $h_n \in H_n$,
for every $t\in[0,T]$, $\bbP$-almost surely.

It is a standard procedure to show the existence of a solution $(\philn,\muln,\sigmaln)$ 
to the Galerkin system \eqref{Def:SystemGalerkin}: this is done in detail e.g. in \cite[p.3813--3814]{scarpa2021degenerate}
for the stochastic Cahn--Hilliard equation, so we omit here the technical details.
The idea is to represent $\philn$, $\muln$ and $\sigmaln$ as the sum of the basis functions in $H_n$ and reduce the problem to an abstract evolution equation in a finite-dimensional space. Then, by standard theory, it follows that the Galerkin system \cref{Def:SystemGalerkin} admits a unique (probabilistically strong) solution
\[
  \philn, \muln, \sigmaln \in L^\ell\big(\Omega; C^0([0,T]; H_n)\big)
  \quad\forall\,\ell\in\mathopen[2,\infty\mathclose).
\]

\subsection{Step 2: uniform estimates in $n$, with $\lambda$ fixed} \label{Sec:Step2:UniformN}
In this step, we show that the approximated solution $(\philn, \muln,\sigmaln)$
satisfies an energy estimate, independently of $n$,
with $\lambda>0$ being fixed. Ultimately, this step ensures the existence of a solution to the $\lambda$-approximated system when passing to the limit $n \to \infty$ in the next step in \Cref{Sec:Step3:LimN}.

\begin{lemma}
There holds the estimate:
    \begin{equation} \label{Est:lambdaFinal}\begin{aligned}  &\E\sups\norm{\philn(s)}_V^{\ell} +
	\E\sups\norm{\Psil(\philn(s))}_{L^1(D)}^{\ell/2}
	+\E\norm{\nabla\muln}_{L^2(0,t; H)}^{\ell} \\&\quad +\E\sups|(\muln(s))_D|^{\ell/2} + \bbE\sup_{s \in [0,T]}\norm{\sigmaln(s)}_H^\ell + \bbE\|\nabla \sigmaln\|_{L^2(0,t;H)}^\ell
  \\ &\leq c_\lambda \Big(1+ 
	\norm{\phi_0}_V^{\ell}  + \norm{\sigma_0}_H^\ell\Big).
\end{aligned}
\end{equation}
\end{lemma}
\begin{proof}
 We split the proof into four parts. Firstly, we derive a $\philn$-dependent estimate for $\sigmaln$ by a suitable application of It\^o's formula. Similarly, we derive a $\sigmaln$-dependent estimate for $\philn$ and its space average. Lastly, we add the three estimates and by an application of Gronwall's inequality, we obtain a $n$-uniform estimate. \smallskip

\noindent{\bf (i)~} We write It\^o's formula \cref{Eq:ItoFormula} for $K\norm{\sigmaln}_H^2$ for $K>0$ being a constant that is determined later on. This procedure yields the equation
\begin{equation} \label{Eq:TestedGalerkinSigma}\begin{aligned}
  &K\norm{\sigmaln(t)}_H^2+ 2K(m_{2,n}(\sigmaln),|\nabla \sigmaln|^2)_{D_t}  
  +2K\delta  (f(\philn),|\sigmaln|^2)_{D_t}
  \\ &=K\norm{\sigma_0^n}_H^2
+2K\chi\big(m_{2,n}(\sigmaln(s)) \nabla \sigmaln(s),\nabla \philn(s)\big)_{D_t} 
  \\[-.2cm]&\quad +\!K \norm{G_{2,n}(\sigmaln(s))}^2_{L^2(0,t;\calL^2(U_2,H))}\!+\!2K\!\!\int_0^t\!\left(\sigmaln(s),G_{2,n}(\sigmaln(s))\!\dW_2(s)\right)_D
  \end{aligned}
  \end{equation}
  for every $t\in[0,T]$, $\mathbb P$-almost surely.
  On the left-hand side of the equality, we may use $m_{2,n}(\sigmaln) \geq m_0$ and $f(\phil) \geq 0$, see  the assumptions \cref{Ass:Mob} and \cref{Ass:Fun}.
On the right-hand side, we exploit the definition of the approximate initial value $\sigma_0^n$ to obtain $\|\sigma_0^n\|_H \leq \|\sigma_0\|_H$. Concerning the second term on the right-hand side, we have by the Young inequality and the bound of the mobility function, see \cref{Ass:Mob},
$$\begin{aligned}&2K\chi\big(m_{2,n}(\sigmaln(s)) \nabla \sigmaln(s),\nabla \philn(s)\big)_{D_t} \\&\leq Km_0 \|\nabla \sigmaln\|_{L^2(0,t;H)}^2 + \frac{Km_\infty^2\chi^2}{m_0}  \|\nabla \philn\|_{L^2(0,t;H)}^2.
\end{aligned}$$
We insert the estimates back into \cref{Eq:TestedGalerkinSigma} to obtain
\begin{equation} \label{Eq:TestedGalerkinSigma1}\begin{aligned}
  &K\norm{\sigmaln(t)}_H^2+ Km_0\|\nabla \sigmaln\|_{L^2(0,t;H)}^2  
  \\ &\leq K\norm{\sigma_0}_H^2+
\frac{Km_\infty^2\chi^2}{m_0}  \|\nabla \philn\|_{L^2(0,t;H)}^2 +K \norm{G_{2,n}(\sigmaln(s))}^2_{L^2(0,t;\calL^2(U_2,H))}
  \\[-.2cm]&\quad +2K\int_0^t\left(\sigmaln(s),G_{2,n}(\sigmaln(s))\dW_2(s)\right)_D.
  \end{aligned}
  \end{equation}
At this point, by \eqref{est_Gn} one has that
\[
  \norm{G_{2,n}(\sigmaln)}^2_{L^2(0,T;\calL^2(U_2,H))}\leq c,
\]
which implies by the Burkholder--Davis--Gundy inequality, see \cref{Lem:BurkholderDavisGundy},
that for every $\ell\geq2$,
    $$\begin{aligned} &\E\sup_{s\in[0,T]}\left|\int_0^s
    \left(\sigmaln(r),G_{2,n}(\sigmaln(r))\dW_2(r)\right)_D
    \right|^{\ell/2}\\
    &\qquad\leq c\E\left(\int_0^t
    \|\sigmaln(s)\|_H^2\|G_{2,n}(\sigmaln(s))\|^2_{\mathcal L^2(U_1,H)}\ds 
    \right)^{\ell/4}\\
    &\qquad\leq c
    \E\|\sigmaln\|_{L^2(0,t; H)}^{\ell/2}.
    \end{aligned}
    $$
Consequently, we take the power of $\ell/2$, the supremum in time, and then the expectation $\bbE$ in inequality \eqref{Eq:TestedGalerkinSigma1}: we deduce that
\begin{equation} \label{Eq:TestedGalerkinSigma2}\begin{aligned}
  &K^{\ell/2}\bbE\sup_{s \in [0,T]}\norm{\sigmaln(s)}_H^\ell+ (Km_0)^{\ell/2} \bbE\|\nabla \sigmaln\|_{L^2(0,t;H)}^\ell
  \\ &\leq c_{K,\ell} \Big(1+\norm{\sigma_0}_H^\ell + \bbE\|\nabla \philn\|_{L^2(0,t;H)}^\ell+
  \bbE\norm{\sigmaln}_{L^2(0,t;H)}^{\ell/2}\Big),
  \end{aligned}
  \end{equation} 
where the implicit constant $c_{K,\ell}$ is independent of
$\lambda$ and $n$.

\noindent{\bf (ii)~} Next, we derive a bound for $\bbE\sup_{s \in [0,T]}\norm{\nabla\philn(s)}_H^\ell$ that allows us to absorb the $\nabla\philn$-dependency on the right-hand side of \cref{Eq:TestedGalerkinSigma2} by a Gronwall argument.
We consider the energy functional
\begin{equation}\label{energy}
\mathcal E_\lambda(v):=\frac{\eps^2}{2} \|\nabla v\|_H^2 + \|\Psil(v)\|_{L^1(D)} , \quad v\in H_n.
\end{equation}
As proved in \cite[p.3829]{scarpa2021degenerate}, $\mathcal E_\lambda:H_n \to [0,\infty)$ is twice Fr\'echet differentiable with 
\[
\begin{aligned}
D\mathcal E_\lambda(v)[h]&=\eps^2 (\nabla v,\nabla h)_D+(\Psil'(v),h)_D
&&\forall h \in H_n, \\
D^2\mathcal E_\lambda(v)[h,k]&=\eps^2 (\nabla h,\nabla k)_D+ (\Psil''(v)h,k)_D &&\forall h,k\in H_n.
\end{aligned}
\]
We  apply It\^o's formula \cref{Eq:ItoFormula} to $\philn \mapsto \mathcal E_\lambda(\philn)$.
To this end, note that by \eqref{Def:SystemGalerkin} we have 
$D\mathcal E_\lambda(\philn)=\muln$, and thus, we obtain 
\begin{equation}\label{Eq:TestedGalerkinPhi0}\begin{aligned}
	&\mathcal E_\lambda(\philn(t))
	+(m_{1,n}(\philn),|\nabla\muln|^2)_{D_t}  \\
	&=\mathcal E_\lambda(\phi_0^{n}) 
	+(\muln f(\philn),\beta\sigmaln-\alpha)_{D_t}+\chi(m_{1,n}(\philn) \nabla \sigmaln,\nabla\muln)_{D_t}
 \\ &\quad +\int_0^t\left(\muln(s), G_{1,n}(\philn(s))\,\d W_1(s)\right)_D\\
 &\quad+\frac12\int_0^t\sum_{k=0}^\infty\Big[\|\nabla G_{1,n}(\philn(s))u_k^1\|_H^2
+\big(\Psil''(\philn(s)),|G_{1,n}(\philn(s))u_k^1|^2\big)_D\Big]\ds.
\end{aligned}\end{equation}

 First, we note that it holds $\calE_\lambda(\phi_0^n)=\frac{\eps^2}{2} \|\nabla \phi_0^n\|_H^2 + \|\Psi_\lambda(\phi_0^n)\|_{L^1(D)}$. From the definition of the approximate initial value
$\phi_0^n$, we have 
$\|\nabla\phi_0^{n}\|_H\leq\norm{\nabla\phi_0}_H$. Since $\Psil$ is bounded by a quadratic function due to the properties of the Yosida approximation
and $(\phi_0^n)_n$ is bounded in $H$ thanks to the properties of
the orthogonal projection on $H_n$, 
we have that 
\[
\norm{\Psil(\phi_0^{n})}_{L^1(D)}\leq 
c_\lambda\big(1+\norm{\phi_0^n}_H^2\big)\leq 
c_\lambda\big(1+\norm{\phi_0}_H^2\big).
\]
Now, note that 
taking the space average of the equation governing $\muln$, see \cref{Def:SystemGalerkin}$_2$, and using the 
Lipschitz-continuity of $\Psi'_\lambda$, we get 
\[
(\muln)_D = (\Psil'(\philn))_D\leq c_\lambda\|\philn\|_H.
\] 
Hence, we may treat the second and third terms on the right-hand side by the Young and Poincar\'e inequalities as follows:
$$\begin{aligned}
    (\muln f(\philn),\beta\sigmaln-\alpha)_{D_t} &\leq 
    \frac{m_0}{4} \|\nabla \muln\|_{L^2(0,t;H)}^2 \\
    &\quad+ 
    c_\lambda \Big( 1+ \|\philn\|_{L^2(0,t;H)}^2 + \|\sigmaln\|_{L^2(0,t;H)}^2 \Big), \\
    \chi(m_{1,n}(\philn) \nabla \sigmaln,\nabla\muln)_{D_t} &\leq
      \frac{m_0}{4} \|\nabla \muln\|_{L^2(0,t;H)}^2 + \frac{m_\infty^2\chi^2}{2m_0} \|\nabla \sigmaln\|_{L^2(0,t;H)}^2,
\end{aligned}$$
where we used the bounds for $f$ and the mobility $m_{1,n}$ as assumed in \cref{Ass:Fun} and \cref{Ass:Mob}. We insert the estimates back into \cref{Eq:TestedGalerkinPhi0} to get
\begin{equation}\label{Eq:TestedGalerkinPhi1}\begin{aligned}
	&\frac{\eps^2}{2} \|\nabla \philn(t)\|_H^2 + \|\Psil(\philn(t))\|_{L^1(D)}
	+\frac{m_0}{2}\|\nabla\muln\|_{L^2(0,t;H)}^2  \\
	&\leq c_\lambda \Big(1+\|\phi_0\|_{V}^2+
 \|\philn\|_{L^2(0,t;H)}^2 + \|\sigmaln\|_{L^2(0,t;H)}^2\Big) 
  \\ &\quad+ \frac{m_\infty^2\chi^2}{2m_0} \|\nabla \sigmaln\|_{L^2(0,t;H)}^2
  +\int_0^t\left(\muln(s), G_{1,n}(\philn(s))\,\d W_1(s)\right)_D\\
 &\quad+\frac12\int_0^t\sum_{k=0}^\infty\Big[\|\nabla G_{1,n}(\philn(s))u_k^1\|_H^2
 +\big(\Psil''(\philn(s)),|G_{1,n}(\philn(s))u_k^1|^2\big)_D\Big]\ds.
\end{aligned}\end{equation}

We take the power $\ell/2$ on both sides of this inequality, the supremum in time and then the expectation. Let us focus on the stochastic integral first: we apply the Burkholder--Davis--Gundy
inequality 
and argue as before \eqref{Eq:TestedGalerkinSigma2}
by exploiting \cref{Ass:Noi} to obtain
\begin{align*}
	\E\sups\,\left|\int_0^s\left(\muln(r), 
	G_{1,n}(\philn(r))\,\d W_1(r)\right)_D\right|^{\ell/2}
 \leq c
	\E\norm{\muln}_{L^2(0,t; H)}^{\ell/2}.
\end{align*}
We sum and subtract $(\muln)_D$ on the right-hand side,
apply the Poincar\'e-Wirtinger and Young inequalities, which yields
\begin{align*}
	&\E\sups\,\left|\int_0^s\left(\muln(r), 
	G_{1,n}(\philn(r))\,\d W_1(r)\right)_D\right|^{\ell/2}\\
	&\leq \frac{m_0^{\ell/2}}{2\cdot 2^{\ell/2}} \E\norm{\nabla\muln}_{L^2(0,t; H)}^{\ell} 
	+ c_\ell\Big( 1+\E\norm{(\muln)_D}_{L^2(0,t)}^{\ell/2} \Big).
\end{align*}
Taking the space average of the equation governing $\muln$, see \cref{Def:SystemGalerkin}$_2$, and using the assumption on the potential $\Psi$, see \cref{Ass:Psi}, we get 
\[
(\muln)_D = (\Psil'(\philn))_D\leq
\norm{\Psil'(\philn)}_{L^1(D)}\leq c\big(1 + \norm{\Psil(\philn)}_{L^1(D)}\big),
\] 
so that, putting everything together, 
\begin{align*}
	&\E\sups\,\left|\int_0^s\left(\muln(r), 
	G_{1,n}(\philn(r))\,\d W_1(r)\right)_H\right|^{\ell/2}\\
	&\leq \frac{m_0^{\ell/2}}{2\cdot 2^{\ell/2}} \E\norm{\nabla\muln}_{L^2(0,t; H)}^{\ell} 
	+ c_\ell\Big( 
 1+\E\int_0^t\norm{\Psi_\lambda(\philn(s))}_{L^1(D)}^{\ell/2}\ds
 \Big).
\end{align*}

Let us focus now on the trace terms in It\^o's formula \cref{Eq:TestedGalerkinPhi1}.
Since $G_1(\philn)$ takes its values in $\calL^2(U_1,V)$, thanks to 
\cref{Ass:Noi} we have
\begin{align*}
	\int_0^t\sum_{k=0}^\infty
 \|\nabla G_{1,n}(\philn(s))u_k^1\|_H^2\ds
 &=\int_0^t\sum_{k=0}^\infty
 \|g_{1,k}'\|_{L^\infty(\mathbb R)}^2
 \|\nabla \philn(s)\|_H^2\ds\\
 &\leq C_{G_1}\|\nabla\philn\|^2_{L^2(0,t; H)}
\end{align*}

 Lastly, we note that  it holds
 $|\Psi''_\lambda|\leq c_\lambda$ for a certain $c_\lambda>0$ and thus,  by the same computations as above, we obtain
\[
\sum_{k=0}^\infty
\big(\Psil''(\philn),|G_{1,n}(\philn)u_k^1|^2\big)_D
\leq c_\lambda\norm{G_1(\philn)}_{\calL^2(U_1,H)}^2
\leq c_\lambda |D| C_{G_1}.
\]
Putting everything together, we finally obtain the estimate \begin{equation} \label{Eq:TestedPhi} \begin{aligned}
	&\frac{\eps^\ell}{2^{\ell/2}}\E\sups\norm{\nabla\philn(s)}_H^{\ell} +
	\E\sups\norm{\Psil(\philn(s))}_{L^1(D)}^{\ell/2}
	+\frac{m_0^{\ell/2}}{2^{1+\ell/2}}\E\norm{\nabla\muln}_{L^2(0,t; H)}^{\ell}\\&\quad+\E\sups|(\muln(s))_D|^{\ell/2}\\
	&\leq c_\lambda \Bigg(1+ 
	\norm{\phi_0}_V^{\ell} 
	+\E\int_0^t\norm{\Psi_\lambda(\philn(s))}_{L^1(D)}^{\ell/2}\ds
 \Bigg)  \\ 
 &\quad+ c_\lambda \Big(\bbE\|\sigmaln\|_{L^2(0,t;H)}^\ell + \bbE\|\philn\|_{L^2(0,t;V)}^\ell
 \Big)  
 +\frac{m_\infty^\ell\chi^\ell}{(2m_0)^{\ell/2}} \bbE \|\nabla \sigmaln\|_{L^2(0,t;H)}^\ell.
\end{aligned}\end{equation}

\noindent{\bf (iii)~}  
We focus here  on the space average of $\philn$.
We choose $h_n=1$ in \cref{Def:SystemGalerkinVariational}$_1$, integrate in time, and apply It\^o's formula, see \cref{Eq:ItoFormula}, with the function $x\mapsto |x|^2$,
    which yields
    $$\begin{aligned} |(\philn(t))_D|^2&=|(\phi_0^n)_D|^2+ 
    2\int_{D_t}(\beta\sigmaln-\alpha)f(\philn)\\
    &\quad+
    \norm{(G_{1,n}(\philn))_D}^2_{L^2(0,t;\calL^2(U_1,\bbR))} \\&\quad+2\int_0^t (\philn(s))_D\, (G_{1,n}(\philn(s)))_D\dW_1(s). \end{aligned}$$ 
Now, it holds by the H\"older inequality that $|(\phi_0^n)_D| \leq c \norm{\phi_0^n}_H \leq c \norm{\phi_0}_H$ where we used the definition of the initial $\phi_0^n$ in the last step. Moreover, 
thanks to \eqref{Ass:Fun} one has that 
\[
  \int_{D_t}(\beta\sigmaln-\alpha)f(\philn)
  \leq c\left(1+\|\sigmaln\|_{L^2(0,t; H}^2\right).
\]
Furthermore, by \eqref{est_Gn} we readily have 
\[
  \norm{(G_{1,n}(\philn))_D}^2_{L^2(0,t;\calL^2(U_1,\bbR))}
  \leq c\norm{G_{1,n}(\philn)}^2_{L^2(0,T;\calL^2(U_1,H))}\leq c.
\]
This implies by the Burkholder--Davis--Gundy and inequality, see \cref{Lem:BurkholderDavisGundy},
and the Poincar\'e-Wirtinger inequality, 
that for every $\ell\geq2$
    $$\begin{aligned} &\E\sup_{s\in[0,T]}\left|
    \int_0^s (\philn(r))_D\, (G_{1,n}(\philn(r)))_D\dW_1(r)
    \right|^{\ell/2}\\
    &\qquad\leq c\E\left(\int_0^t
    \|\philn(s)\|_H^2\|G_{1,n}(\philn(s))\|^2_{\mathcal L^2(U_1,H)}\ds 
    \right)^{\ell/4}\\
    &\qquad\leq c
    \E\|\philn\|_{L^2(0,t; H)}^{\ell/2}.
    \end{aligned}
    $$
Putting this information together, taking supremum in time, power $\ell/2$ and expectations, yield that
\begin{equation}\label{est:luca_mean}
\E\sups|(\philn(s))_D|^\ell \leq
c\left(1 + \|\sigmaln\|_{L^2(0,t;H)}^{\ell}+
\E\|\philn\|_{L^2(0,t; H)}^{\ell/2}\right).
\end{equation}

\noindent{\bf (iv)~} Next, we add the three estimates to obtain a new estimate, absorb the terms and apply Gronwall's lemma to get a $n$-uniform estimate.
Indeed, we add \cref{Eq:TestedGalerkinSigma2},
\cref{Eq:TestedPhi}, and \cref{est:luca_mean} to obtain
\begin{align*}  
&\E\sups|(\philn(s))_D|^\ell+
\frac{\eps^\ell}{2^{\ell/2}}\E\sups\norm{\nabla\philn(s)}_H^{\ell} +
	\E\sups\norm{\Psil(\philn(s))}_{L^1(D)}^{\ell/2}
	\\&\quad
 +\frac{m_0^{\ell/2}}{2^{1+\ell/2}}
  \E\norm{\nabla\muln}_{L^2(0,t; H)}^{\ell}
  +\E\sups|(\muln(s))_D|^{\ell/2} \\&\quad
  + K^{\ell/2}\bbE\sup_{s \in [0,T]}\norm{\sigmaln(s)}_H^\ell 
  + \Big((Km_0)^{\ell/2}-\frac{m_\infty^\ell\chi^\ell}{(2m_0)^{\ell/2}}\Big) \bbE\|\nabla \sigmaln\|_{L^2(0,t;H)}^\ell
  \\ &\leq c_\lambda (1+ 
	\norm{\phi_0}_V^{\ell}  + \norm{\sigma_0}_H^\ell)\\
 &\quad+ c_\lambda
	\E\int_0^t
 \Big(\norm{\Psi_\lambda(\philn(s))}_{L^1(D)}^{\ell/2}
 +\|\sigmaln(s)\|_H^\ell+\|\philn(s)\|_V^\ell
 \Big)\ds.
\end{align*}
At this point, we select $K$ large enough to ensure that the prefactor of $\bbE\|\nabla \sigmaln\|_{L^2(0,t;H)}^\ell$ is positive. 
Hence, the Gronwall lemma and the Poincar\'e inequality 
yield \eqref{Est:GalerkinFinal}.
\end{proof}

Now, by \eqref{Est:GalerkinFinal} 
we infer
the $n$-uniform bounds
\begin{align}
	\label{Est:Phi1}
	\norm{\philn}_{L^\ell(\Omega; C^0([0,T]; V))}&\leq c_\lambda,\\
	\label{Est:Mu}
	\norm{\muln}_{L^{\ell/2}(\Omega; L^2(0,T; V))} + 
	\norm{\nabla\muln}_{L^\ell(\Omega; L^2(0,T; H))} &\leq c_\lambda, \\
 	\label{Est:Sigma1}
	\norm{\sigmaln}_{L^\ell(\Omega; C^0([0,T];H)\cap L^2(0,T;V))}&\leq c_\lambda.
 \end{align} 
Moreover, the Lipschitz continuity of $G_{i,n}$ and \cref{Ass:Noi} yield 
$$\begin{aligned}
    \norm{G_{1,n}(\philn)}_{L^\infty(\Omega\times(0,T); \calL^2(U_1,H))
	\cap L^\ell(\Omega; L^\infty(0,T; \calL^2(U_1,V)))} &\leq c_\lambda, \\
 \norm{G_{2,n}(\sigmaln)}_{L^\infty(\Omega\times(0,T); \calL^2(U_2,H))
	\cap L^\ell(\Omega; L^\infty(0,T; \calL^2(U_2,V)))} &\leq c_\lambda. 
\end{aligned}$$
This in turn implies, by \cite[Lemma 2.1]{flandoli1995martingale}, that 
for every $s \in (0,1/2)$
$$\begin{aligned}
    \norm{\int_0^\cdot G_{1,n}(\philn(s))\,\d W_1(s)}_{L^\ell(\Omega; W^{s,\ell}(0,T; V))} &\leq c_{\lambda,s}, \\
 \norm{\int_0^\cdot G_{2,n}(\sigmaln(s))\,\d W_2(s)}_{L^\ell(\Omega; W^{s,\ell}(0,T; V))} &\leq c_{\lambda,s}.
\end{aligned}$$
Let us fix now $\bar s\in(1/\ell,1/2)$: 
by comparison with the model equations, we directly obtain then
\begin{align}
 \label{Est:Phi2}
\norm{\philn}_{L^\ell(\Omega; L^2(0,T; V_2)\cap W^{\bar s,\ell}(0,T; V^*))} &\leq c_\lambda, \\
 \label{Est:Sigma2}
\norm{\sigmaln}_{L^\ell(\Omega; W^{\bar s,\ell}(0,T; V^*))} &\leq c_\lambda.
\end{align}

\subsection{Step 3: passage to the limit as $n\to\infty$, with $\lambda$ fixed}
\label{Sec:Step3:LimN}
We perform here the passage to the limit as $n\to\infty$, 
keeping $\lambda>0$ fixed.

\begin{lemma}
    The sequence of laws of 
    $(\philn, G_{1,n}(\philn)\cdot W_1,
    W_1)_{n\in\bbN}$
is tight on the product space
\[
\left(C^0([0,T]; H)\cap L^2(0,T; V)\right) \times 
C^0([0,T]; H) \times 
C^0([0,T]; \tilde U_1),
\]
and the sequence of laws of 
 $(\sigmaln, G_{2,n}(\sigmaln)\cdot W_2, W_2)_{n\in\bbN}$
is tight on the product space
\[
\left(C^0([0,T]; V^*)\cap L^2(0,T; H)\right) \times
C^0([0,T]; H) \times 
C^0([0,T]; \tilde U_2).
\]
\end{lemma}
\begin{proof}
Since it holds $\bar s>1/\ell$,
the Aubin--Lions compactness lemma, see \cite[Cor.~5, p.~86]{simon1986compact},
implies the compact inclusions
\begin{align*}
L^\infty(0,T; V)\cap W^{\bar s,\ell}(0,T; V^*)&\hookrightarrow\hookrightarrow C^0([0,T]; H),\\
L^2(0,T; V_2)\cap W^{\bar s,\ell}(0,T; V^*)&\hookrightarrow\hookrightarrow L^2(0,T; V).
\end{align*}
Hence, for every $R>0$ the closed ball $B_R$ 
in $L^2(0,T; V_2)\cap 
L^\infty(0,T; V)\cap W^{\bar s,\ell}(0,T; V^*)$ of radius $R$
is compact in $C^0([0,T]; H)\cap L^2(0,T; V)$. Moreover, 
thanks to the Markov inequality and the estimates \eqref{Est:Phi1} and \eqref{Est:Phi2} 
we have
\begin{align*}
	\P\{\philn\in B_R^c\}&=
	\P\{\norm{\philn}_{L^2(0,T; V_2)\cap L^\infty(0,T; V)\cap W^{\bar s,\ell}(0,T; V^*)}> R\}\\
	&\leq
	\frac1{R^{\ell}}\E\norm{\philn}_{L^2(0,T; V_2)\cap L^\infty(0,T; V)\cap W^{\bar s,\ell}(0,T; V^*)}^\ell\leq
	\frac{c_\lambda^\ell}{R^\ell},
\end{align*}
which yields $\lim_{R\to\infty}\sup_{n\in\bbN}\P\{\philn\in B_R^c\}=0$. Hence, the family of laws of $(\philn)_n$ on $C^0([0,T]; H)\cap L^2(0,T; V)$ is tight.
Using a similar argument, 
since $W^{\bar s, \ell}(0,T; V)$ is compactly embedded in
$C^0([0,T]; H)$, one can also show that the family of laws of $
G_{1,n}(\philn)\cdot W_1:=\int_0^\cdot G_{1,n}(\philn(s))\,\d W_1(s)$
is tight on $C^0([0,T]; H)$.
Eventually, we identify $W_1$ with
a constant sequence of random variables with values in
$C^0([0,T]; \tilde U_1)$, which is tight.  This proves the first assertion of the lemma. The second assertion  follows analogously 
by the compact inclusions
\begin{align*}
L^\infty(0,T; H)\cap W^{\bar s,\ell}(0,T; V^*)&\hookrightarrow\hookrightarrow C^0([0,T]; V^*),\\
L^2(0,T; V)\cap W^{\bar s,\ell}(0,T; V^*)&\hookrightarrow\hookrightarrow L^2(0,T; H),
\end{align*}
and this concludes the proof.
\end{proof}

Finally, we are in the position to prove \cref{Lem:ExistenceLambda}, i.e., we prove that the Yosida approximated system admits a martingale solution.
\begin{proof}[Proof of \cref{Lem:ExistenceLambda}]
By the Prokhorov and Skorokhod theorems
and their weaker version on sub-Polish spaces, see the discussion at the end of  \Cref{Subsec:Stochastic},
recalling the estimates \eqref{Est:Phi1}--\eqref{Est:Phi2}, we conclude that
there exists a probability space $(\tilde\Omega,\tilde\calF, \tilde\P)$ and measurable 
maps $\Lambda_n:(\tilde\Omega,\tilde\calF)\to(\Omega,\calF)$ such that 
$\tilde\P\circ \Lambda_n^{-1}=\P$ for every $n\in\bbN$ and as $n \to \infty$ it holds
\begin{align*}
	\tilde\philn:=\philn\circ\Lambda_n &\to \tilde\phil
	&&\text{in } L^p(\tilde\Omega;C^0([0,T]; H)\cap L^2(0,T; V)) \quad\,\forall\,p<\ell,\\
	\tilde\philn &\stackrel{*}{\rightharpoonup} \tilde\phil 
	&&\text{in } L^\ell_w(\tilde\Omega; L^\infty(0,T; V)),\\
   \tilde\philn &\rightharpoonup \tilde\phil 
	&&\text{in } L^\ell(\tilde\Omega; L^2(0,T; V_2))
 \cap L^\ell(\tilde\Omega; W^{\bar s,\ell}(0,T; V^*)),\\
	\tilde\muln:=\muln\circ\Lambda_n &\rightharpoonup\tilde\mul
	&&\text{in } L^{\ell/2}(\tilde\Omega; L^2(0,T; V)),\\
	\nabla\tilde\muln &\rightharpoonup\nabla\tilde\mul
	&&\text{in } L^{\ell}(\tilde\Omega; L^2(0,T; H)),\\
 \tsigmaln:=\sigmaln\circ\Lambda_n &\to \tsigmal
	&&\text{in } L^p(\tilde\Omega;C^0([0,T]; V^*)\cap L^2(0,T; H)) \quad\,\forall\,p<\ell,\\
 \tilde\sigmaln &\stackrel{*}{\rightharpoonup} \tilde\sigmal 
	&&\text{in } L^\ell_w(\tilde\Omega; L^\infty(0,T; H)),\\
	\tsigmaln &\rightharpoonup \tsigmal 
	&&\text{in } L^\ell(\tilde\Omega; L^2(0,T; V))
 \cap L^\ell(\tilde\Omega; W^{\bar s,\ell}(0,T; V^*)),
\end{align*}
for some measurable processes
\begin{align*}
    	\tilde \phil &\in L^\ell(\tilde\Omega; 
 W^{\bar s,\ell}(0,T; V^*) \cap 
 C^0([0,T]; H)\cap L^2(0,T; V_2))\cap 
	L^\ell_w(\tilde\Omega; L^\infty(0,T; V)),\\
	\tilde\mul &\in L^{\ell/2}(\tilde\Omega; L^2(0,T; V)), 
	\quad\nabla\tilde\mul\in L^\ell(\tilde\Omega; L^2(0,T; H)),\\
 \tsigmal &\in L^\ell(\tilde\Omega; 
 W^{\bar s,\ell}(0,T; V^*)\cap 
 C^0([0,T]; V^*)\cap  L^2(0,T; V))
 \cap L^\ell_w(\tilde\Omega; L^\infty(0,T; H)),
\end{align*}
as well as
\begin{align*}
	\tilde I_{\lambda,n}^1:=(G_{1,n}(\philn)\cdot W_1)\circ\Lambda_n &\to \tilde I_{\lambda}^1
	&&\text{in } L^p(\tilde\Omega; C^0([0,T]; H)) \quad\forall\,p<\ell,\\
 \tilde I_{\lambda,n}^2:=(G_{2,n}(\sigmaln)\cdot W_2)\circ\Lambda_n &\to \tilde I_{\lambda}^2
	&&\text{in } L^p(\tilde\Omega; C^0([0,T]; H)) \quad\forall\,p<\ell,\\
	\tilde W_{1,n}:=W_1\circ\Lambda_n &\to \tilde W_1 
	&&\text{in } L^p(\tilde\Omega; C^0([0,T]; \tilde U_1)) \quad\forall\,p<\ell, \\
 \tilde W_{2,n}:=W_2\circ\Lambda_n &\to \tilde W_2 
	&&\text{in } L^p(\tilde\Omega; C^0([0,T]; \tilde U_2)) \quad\forall\,p<\ell,
\end{align*}
for some measurable processes 
\begin{align*}
	\tilde I^1_\lambda, \tilde I^2_\lambda &\in L^\ell(\tilde\Omega; C^0([0,T]; H)),\\
	 \tilde W_1 &\in L^\ell(\tilde \Omega; C^0([0,T]; \tilde U_1)),\\
  \tilde W_2 &\in L^\ell(\tilde \Omega; C^0([0,T]; \tilde U_2)).
\end{align*}
Note that possibly enlarging the new probability space, it is not restrictive to 
suppose that $(\tilde\Omega,\tilde\calF,\tilde\P)$ is independent of $\lambda$.
Now, since $\Psil'$ and $G_i:H \to \calL^2(U_i,H)$ are Lipschitz-continuous, we readily have
\[\begin{aligned}
\Psil'(\tilde\philn)&\to \Psil'(\tilde\phil) &&\text{in } L^\ell(\tilde\Omega; L^2(0,T; H)), \\
G_{1,n}(\tilde\philn) &\to G_1(\tilde\phil) 
&&\text{in } L^\ell(\tilde\Omega; L^2(0,T; \calL^2(U_1,H)))
\quad\forall\,p<\ell, \\
G_{2,n}(\tilde\sigmaln) &\to G_2(\tilde\sigmal) 
&&\text{in } L^\ell(\tilde\Omega; L^2(0,T; \calL^2(U_2,H)))
\quad\forall\,p<\ell,
\end{aligned}
\]
Moreover, since $\phi_0\in V$ and $\sigma_0 \in H$, it implies
$\phi_0^n\to\phi_0$ in $V$ and $\sigma_0^n \to \sigma_0$ in $H$. Similarly to \cite{scarpa2021degenerate} (i.e., following the arguments in \cite{flandoli1995martingale}, \cite[Section 8.4]{da2014stochastic}, \cite[Section~4.5]{sapountzoglou2019doubly}, \cite[Section~4.3]{vallet2019well}), a standard procedure allows identifying the 
limit terms $\tilde I_\lambda^1$ and $\tilde I_\lambda^2$ as
the $H$-valued martingales given by 
\begin{equation}
\label{id_stoc}
\tilde I_\lambda^1=\int_0^\cdot G_1(\tilde\phil(s))\,\d\tilde W_1(s), \quad \tilde I_\lambda^2=\int_0^\cdot G_2(\tilde\sigmal(s))\,\d\tilde W_2(s).
\end{equation}

We test now each equation in \eqref{Def:SystemGalerkin} by an 
arbitrary element $v\in V$ and apply $\Lambda_n$.
By exploiting the convergences above and the fact that, for $i=1,2$,
$m_{i,n}(\tilde\philn)\to m_i(\tilde\phil)$
almost everywhere thanks to the estimate
$|m_{i,n}|\leq m^*$
and the dominated convergence theorem,
we can let $n\to\infty$ and conclude the proof of 
\cref{Lem:ExistenceLambda}.
\end{proof}

\subsection{Step 4: uniform estimates in $\lambda$} \label{Sec:Step4:UniformLambda}
We prove here uniform estimates independently of $\lambda$. To this end, 
we start with  the following lemma.

\begin{lemma}
There holds the estimate:
    \begin{equation} \label{Est:GalerkinFinal}\begin{aligned}  &\tilde\E\sups\norm{\tilde\phil(s)}_V^{\ell} +
	\tilde\E\sups\norm{\Psil(\tilde\phil(s))}_{L^1(D)}^{\ell/2}
	+\tilde\E\norm{\nabla\mul}_{L^2(0,t; H)}^{\ell} \\&\quad 
   +\tilde\E\sups|(\tilde\mul(s))_D|^{\ell/2} + 
   \tilde\E\sup_{s \in [0,T]}\norm{\tilde\sigmal(s)}_H^\ell 
   + \tilde\E\|\nabla \tilde\sigmal\|_{L^2(0,t;H)}^\ell
  \\ &\leq c \Big(1+ 
	\norm{\phi_0}_V^{\ell} + \|\Psi(\phi_0)\|_{L^1(D)}^{\ell/2}+
 \norm{\sigma_0}_H^\ell\Big).
\end{aligned}
\end{equation}
\end{lemma}
\begin{proof}
Let us go back to \cref{Eq:TestedGalerkinPhi0}.
Since $(\muln)_D\leq\|\Psi_\lambda'(\philn)\|_{L^1(D)}$,
on the right-hand side one has, by the Poincar\'e and H\"older inequalities,
the inclusion $V\hookrightarrow L^6(D)$, and \cref{Ass:Psi},
\begin{align*}
    &(\muln f(\philn),\beta\sigmaln-\alpha)_{D_t} \\
    &\leq \frac{m_0}{4} \|\nabla\muln\|^2_{L^2(0,t; H)}
    +c(1+\|\sigmaln\|^2_{L^2(0,t; H)})\\
    &\qquad+
    c\int_0^t
    \|\Psi_\lambda'(\philn(s))\|_{L^{1}(D)}
    (1+\|\sigmaln(s)\|_{L^1(D)})\ds,
\end{align*}
as well as
\begin{align*}
    \chi(m_{1,n}(\philn) \nabla \sigmaln,\nabla\muln)_{D_t} &\leq
      \frac{m_0}{4} \|\nabla \muln\|_{L^2(0,t;H)}^2 + \frac{m_\infty^2\chi^2}{2m_0} \|\nabla \sigmaln\|_{L^2(0,t;H)}^2.
\end{align*}
Analogously, we have thanks to
\cref{Ass:Noi} and \cref{Ass:Psi} that
\begin{align*}
	\int_0^t\sum_{k=0}^\infty
 \|\nabla G_{1,n}(\philn(s))u_k^1\|_H^2\ds
 &=\int_0^t\sum_{k=0}^\infty
 \|g_{1,k}'\|_{L^\infty(\mathbb R)}^2
 \|\nabla \philn(s)\|_H^2\ds\\
 &\leq C_{G_1}\|\nabla\philn\|^2_{L^2(0,t; H)}
\end{align*}
and
\[
\int_0^t\sum_{k=0}^\infty
\big(\Psil''(\philn(s)),|G_{1,n}(\philn(s))u_k^1|^2\big)_D\ds
\leq c\int_0^t(1+\|\Psi_\lambda(\philn(s))\|_{L^1(D)})\ds.
\]
Taking this information into account in \cref{Eq:TestedGalerkinSigma2}
we deduce that 
\begin{align*}
    &\frac{\eps^2}2\sups\|\nabla\philn(s)\|_H^2
    +\sups\|\Psi_\lambda(\philn(t))\|_{L^1(D)}
    +\frac{m_0}2\|\nabla\muln\|^2_{L^2(0,t; H)}\\
    &\leq \frac{\eps^2}2\|\nabla\phi_0\|_H^2
    +\|\Psi_\lambda(\phi_0^n)\|_{L^1(D)}
    + \frac{m_\infty^2\chi^2}{2m_0}
    \|\nabla\sigmaln\|^2_{L^2(0,t; H)}
    +c\|\sigmaln\|^2_{L^2(0,t;H)}\\
    &\quad+
    c\|\nabla\philn\|^2_{L^2(0,t; H)}
    +c\int_0^t\|\Psi_\lambda'(\philn(s))\|_{L^{1}(D)}
    (1+\|\sigmaln(s)\|_H)\ds\\
    &\quad
    +c\int_0^t(1+\|\Psi_\lambda(\philn(s))\|_{L^1(D)})\ds
    +\sups\int_0^s\left(\muln(r), G_{1,n}(\philn(r))\dW_1(r)\right)_D.
\end{align*}
Furthermore, from \cref{Eq:TestedGalerkinSigma1}
one also has
\begin{equation}\label{est:luca2}
\begin{aligned}
  &K\sups\norm{\sigmaln(s)}_H^2+ 
  Km_0\|\nabla \sigmaln\|_{L^2(0,t;H)}^2  
  \\ &\leq K\norm{\sigma_0}_H^2+ KC_{G_2}|D|T+
\frac{Km_\infty^2\chi^2}{m_0}  \|\nabla \philn\|_{L^2(0,t;H)}^2 
  \\&\quad +2K\sups
  \int_0^s\left(\sigmaln(r),G_{2,n}(\sigmaln(r))\dW_2(r)\right)_D.
\end{aligned}
\end{equation}
Summing up the two inequalities and rearranging the terms
we infer then 
\begin{align*}
    &\frac{\eps^2}2\sups\|\nabla\philn(s)\|_H^2
    +\sups\|\Psi_\lambda(\philn(t))\|_{L^1(D)}
    +\frac{m_0}2\|\nabla\muln\|^2_{L^2(0,t; H)}\\
    &\qquad+K\sups\norm{\sigmaln(s)}_H^2+ 
  \left(Km_0-\frac{m_\infty^2\chi^2}{2m_0}\right)
  \|\nabla \sigmaln\|_{L^2(0,t;H)}^2 
    \\
    &\leq c+ \frac{\eps^2}2\|\nabla\phi_0\|_H^2
    +\|\Psi_\lambda(\phi_0^n)\|_{L^1(D)}+K\norm{\sigma_0}_H^2
    +c\|\sigmaln\|^2_{L^2(0,t;H)}\\
    &\qquad+
    c\|\nabla\philn\|^2_{L^2(0,t; H)}
    +c\int_0^t\|\Psi_\lambda'(\philn(s))\|_{L^{1}(D)}
    (1+\|\sigmaln(s)\|_H)\ds\\
    &\qquad+c\int_0^t\|\Psi_\lambda(\philn(s))\|_{L^1(D)}\ds
    +\sups\int_0^s\left(\muln(r), G_{1,n}(\philn(r))\dW_1(r)\right)_D\\
    &\qquad+2K\sups
  \int_0^s\left(\sigmaln(r),G_{2,n}(\sigmaln(r))\dW_2(r)\right)_D,
\end{align*}
where $K>0$ is fixed so that the prefactor on the left-hand side is positive.
Now, by applying $\Lambda_n$ and letting $n\to\infty$, 
by the convergences as $n\to\infty$ proved above, the strong-weak convergence, 
and the weak lower semicontinuity of the norms we deduce,
by possibly renominating the constants, 
\begin{align*}
    &\sups\|\nabla\tilde\phil(s)\|_H^2
    +\sups\|\Psi_\lambda(\tilde\phil(t))\|_{L^1(D)}
    +\sups|(\tilde\mul(s))|
    +\|\nabla\tilde\mul\|^2_{L^2(0,t; H)}\\
    &\quad+\sups\norm{\tilde\sigmal(s)}_H^2+ 
  \|\nabla\tilde\sigmal\|_{L^2(0,t;H)}^2 
    \\
    &\leq c\left(1+\|\nabla\phi_0\|_H^2
    +\|\Psi(\phi_0)\|_{L^1(D)}+\norm{\sigma_0}_H^2\right)
    +c\|\tilde\sigmal\|^2_{L^2(0,t;H)}
    +c\|\nabla\tilde\phil\|^2_{L^2(0,t; H)}\\
    &\quad
    +c\int_0^t\left(1+\|\Psi_\lambda(\tilde\phil(s))\|_{L^1(D)}+
    \|\Psi_\lambda'(\tilde\phil(s))\|_{L^{1}(D)}
    \|\tilde\sigmal(s)\|_H\right)\ds\\
    &\quad+c\sups\left[\int_0^s
    \left(\tilde\mul(r), G_1(\tilde\phil(r))\d\tilde W_1(r)\right)_D
    +
  \int_0^s\left(\tilde\sigmal(r),G_{2}(\tilde\sigmal(r))\d\tilde W_2(r)\right)_D\right].
\end{align*}
The Burkholder--Davis--Gundy and Poincar\'e inequalities  
imply that 
\begin{align*}
	&\tilde\E\sups\,\left|\int_0^s\left(\tilde\mul(r), 
	G_{1}(\tilde\phil(r))\,\d \tilde W_1(r)\right)_H\right|^{\ell/2}\\
	&\leq \frac{m_0^{\ell/2}}{2\cdot 2^{\ell/2}}
 \tilde\E\norm{\nabla\tilde\mul}_{L^2(0,t; H)}^{\ell} 
	+ c_\ell\Big( 
 1+\tilde\E\int_0^t\norm{\Psi_\lambda(\tilde\phil(s))}_{L^1(D)}^{\ell/2}\ds
 \Big)
\end{align*}
and
\begin{align*}
&\tilde\E\sup_{s\in[0,T]}\left|\int_0^s
    \left(\tilde\sigmal(r),G_{2}(\tilde\sigmal(r))\dW_2(r)\right)_D
    \right|^{\ell/2}\\
    &\qquad\leq c\tilde\E\left(\int_0^t
    \|\tilde\sigmal(s)\|_H^2\|G_{2}(\tilde\sigmal(s))\|^2_{\mathcal L^2(U_1,H)}\ds 
    \right)^{\ell/4}
    \leq c
    \tilde\E\|\tilde\sigmal\|_{L^2(0,t; H)}^{\ell/2}.
\end{align*}
Consequently, taking power $\ell/2$ and expectations we obtain, 
recalling also \cref{est:luca_mean},
\begin{align*}
    &\tilde\E\sups\|\tilde\phil(s)\|_V^\ell
    +\tilde\E\sups\|\Psi_\lambda(\tilde\phil(t))\|_{L^1(D)}^{\ell/2}
    +\tilde\E\sups|(\tilde\mul(s))|^{\ell/2}
    +\tilde\E\|\nabla\tilde\mul\|^\ell_{L^2(0,t; H)}\\
    &\quad+\tilde\E\sups\norm{\tilde\sigmal(s)}_H^\ell+ 
  \tilde\E\|\nabla\tilde\sigmal\|_{L^2(0,t;H)}^\ell 
    \\
    &\leq c\left(1+\|\nabla\phi_0\|_H^2
    +\|\Psi(\phi_0)\|_{L^1(D)}+\norm{\sigma_0}_H^2\right)
    +c\tilde\E\|\tilde\sigmal\|^\ell_{L^2(0,t;H)}
    +c\tilde\E\|\nabla\tilde\phil\|^\ell_{L^2(0,t; H)}\\
    &\quad
    +c\tilde\E\int_0^t\left(1+
    \|\Psi_\lambda(\tilde\phil(s))\|_{L^1(D)}^{\ell/2}+
    \|\Psi_\lambda'(\tilde\phil(s))\|_{L^{1}(D)}^{\ell/2}
    \|\tilde\sigmal(s)\|_H^{\ell/2}\right)\ds.
\end{align*}
We show how to control the last term on the right-hand side in the three cases
$\chi=0$, $\chi>0$, and $m_1=\bar m_1$ constant.
First, if $\chi>0$, by assumption \cref{Ass:Psi} one has $q=2$: hence, 
one can use the Young inequality on the last term on the right-hand side as
\[
\tilde\E\int_0^t
    \|\Psi_\lambda'(\tilde\phil(s))\|_{L^{1}(D)}^{\ell/2}
    \|\tilde\sigmal(s)\|_H^{\ell/2}\ds
    \leq 
c\tilde\E\int_0^t\left(
    \|\Psi_\lambda(\tilde\phil(s))\|_{L^{1}(D)}^{\ell/2} + 
    \|\tilde\sigmal(s)\|_H^{\ell}\right)\ds
\]
and 
conclude thanks to the Gronwall lemma.
If it holds $\chi=0$,  we are in the setting  $q>1$ by \cref{Ass:Psi}:
hence, one readily has from the arbitrariness of $\ell$ in
\cref{est:luca2} that 
\[
  \tilde\E\sups\|\tilde\sigmal(s)\|^{\frac{\ell q}{2(q-1)}}_H\leq 
  c_\ell
\]
so that the last term on the right-hand side can be handled as
\begin{align*}
&\tilde\E\int_0^t
    \|\Psi_\lambda'(\tilde\phil(s))\|_{L^{1}(D)}^{\ell/2}
    \|\tilde\sigmal(s)\|_H^{\ell/2}\ds\\
    &\leq 
c\tilde\E\int_0^t\left(1+
    \|\Psi_\lambda(\tilde\phil(s))\|_{L^{1}(D)}^{\frac\ell{2q}}\right)
    \|\tilde\sigmal(s)\|_H^{\frac\ell2}\ds\\
    &\leq c
    \tilde\E\int_0^t\left(1+
    \|\Psi_\lambda(\tilde\phil(s))\|_{L^{1}(D)}^{\ell/2}\right)\ds
    +c\tilde\E\int_0^t
    \|\tilde\sigmal(s)\|_H^{\frac{\ell q}{2(q-1)}}\ds\\
    &\leq c + c
    \tilde\E\int_0^t\left(1+
    \|\Psi_\lambda(\tilde\phil(s))\|_{L^{1}(D)}^{\ell/2}\right)\ds
\end{align*}
and one can still conclude by the Gronwall lemma. 
If $m_1=\bar m_1$ is constant (hence again $q>1$),  
one can perform 
It\^o formula for the $H$-norm of $\tilde\phil$, which yields, 
together with It\^o formula for the $H$-norm
of $\tilde\sigmal$, that
\begin{align*}
    &\sups\|\tilde\phil(s)\|_H^2+2\eps^2\bar m_1
    \|\Delta\tilde\phil\|_{L^2(0,t; H)}^2
    +\sups\|\tilde\sigmal(s)\|_H^2
    +2m_0\|\nabla\tilde\sigmal\|_{L^2(0,t; H)}^2\\
    &\leq\|\phi_0\|_H^2 + \|\sigma_0\|_H^2
    +2\chi(\bar m_1 + m_\infty)\int_0^t\|\nabla\tilde\phil(s)\|_H
    \|\nabla\tilde\sigmal(s)\|_H\ds\\
    &\qquad-2\int_{D_t}\Psi_\lambda''(\tilde\phil)|\nabla\tilde\phil|^2
    +2\int_{D_t}(\beta\tilde\sigmal-\alpha)f(\tilde\phil)\tilde\phil\\
    &\qquad+\|G_1(\tilde\phil)\|^2_{L^2(0,t;\mathcal L^2(U_1,H))}
    +2\int_0^t\left(\tilde\phil(s), 
    G_1(\tilde\phil(s))\d \tilde W_1(s)\right)_D\\
    &\qquad+\|G_2(\tilde\sigmal)\|^2_{L^2(0,t;\mathcal L^2(U_2,H))}
    +2\int_0^t\left(\tilde\sigmal(s), 
    G_2(\tilde\sigmal(s))\d \tilde W_2(s)\right)_D.
\end{align*}
Noting that 
\begin{align*}
    &2\chi(\bar m_1 + m_\infty)\int_0^t\|\nabla\tilde\phil(s)\|_H
    \|\nabla\tilde\sigmal(s)\|_H\ds
    -2\int_{D_t}\Psi_\lambda''(\tilde\phil)|\nabla\tilde\phil|^2\\
    &\leq m_0\|\nabla\tilde\sigmal\|^2_{L^2(0,t; H)}
    +\left(\frac{\chi^2(\bar m_1+m_\infty)^2}{m_0}+2C_\Psi\right)\|\nabla\tilde\phil\|^2_{L^2(0,t; H)}\\
    &\leq m_0\|\nabla\tilde\sigmal\|^2_{L^2(0,t; H)}
    +\eps^2\bar m_1\|\Delta\tilde\phil\|^2_{L^2(0,t; H)}
    +c\|\tilde\phil\|^2_{L^2(0,t; H)}
\end{align*}
and that all the remaining terms can be handled using the arguments
already performed in the previous section, we still deduce in particular
that 
\[
\tilde\E\sups\|\tilde\sigmal(s)\|^{\frac{\ell q}{2(q-1)}}_H\leq 
  c_\ell
\]
and this allows us to conclude as in the previous case.
\end{proof}

We have obtained then the estimates
\begin{align}
	\label{est1_lam}
	\norm{\tilde\phil}_{L^\ell(\tilde\Omega; L^\infty(0,T; V))}&\leq c,\\
	\label{est2_lam}
	\norm{\tilde\mul}_{L^{\ell/2}(\tilde\Omega; L^2(0,T; V))} + 
	\norm{\nabla\tilde\mul}_{L^\ell(\tilde\Omega; L^2(0,T; H))} &\leq c, \\	\label{est3_lam}
	\norm{\tilde\sigmal}_{L^\ell(\tilde\Omega; L^\infty(0,T; H))\cap
 L^\ell(\tilde\Omega; L^2(0,T; V))}&\leq c.
\end{align}
In the same way as before, we bound $G_1(\tilde\phil)$ and $G_2(\tsigmal)$
thanks to \cref{Ass:Noi} as 
$$\begin{aligned}
    \norm{G_{1}(\tilde\phil)}_{L^\infty(\tilde\Omega\times(0,T); \calL^2(U_1,H))
	\cap L^\ell(\tilde\Omega; L^\infty(0,T; \calL^2(U_1,V)))} &\leq c, \\
 \norm{G_{2}(\tilde\sigmal)}_{L^\infty(\tilde\Omega\times(0,T); \calL^2(U_2,H))
	\cap L^\ell(\tilde\Omega; L^\infty(0,T; \calL^2(U_2,V)))} &\leq c. 
\end{aligned}$$
As usual, these imply a bound of the stochastic integrals by
by \cite[Lem.~2.1]{flandoli1995martingale} in form
$$\begin{aligned}
    \norm{\int_0^\cdot G_{1}(\tilde\phil(s))\,\d \tilde W_1(s)}_{L^\ell(\tilde\Omega; W^{s,\ell}(0,T; V))} &\leq c_{s}, \\
 \norm{\int_0^\cdot G_{2}(\tilde\sigmal(s))\,\d \tilde W_2(s)}_{L^\ell(\tilde\Omega; W^{s,\ell}(0,T; V))} &\leq c_{s}.
\end{aligned}$$
for every $s \in (0,1/2)$, and
by comparison in the equations
we also obtain the estimates
\begin{equation}\label{est4_lam}
\begin{aligned}
\norm{\tilde \phil}_{L^\ell(\tilde\Omega; W^{\bar s,\ell}(0,T; V^*))}\leq c, \\
\norm{\tilde \sigmal}_{L^\ell(\tilde\Omega; W^{\bar s,\ell}(0,T; V^*))}\leq c,
\end{aligned}
\end{equation}
where 
$\bar s\in(1/\ell,1/2)$ is fixed. 

Finally, since $\tilde\mul=-\eps^2\Delta\tilde\phil+\Psi'_\lambda(\tilde\phil)$, 
testing 
by $\gamma_\lambda(\tilde\phil)=
\Psil'(\tilde\phil)+C_\Psi \tilde\phil$
and rearranging the terms
we have
\begin{align*}
	(\gamma_\lambda'(\tilde\phil),|\nabla\tilde\phil|^2)_D
	+\|\gamma_\lambda(\tilde\phil)\|^2_H&=(\tilde\mul+C_\Psi\tilde\phil,\gamma_\lambda(\tilde\phil))_D\\
 &\leq\frac12\|\gamma_\lambda(\tilde\phil)\|^2_H
 +\|\tilde\mul\|^2_H + C_\Psi^2\|\tilde\phil\|_H^2.
\end{align*}
Since the first term on the left-hand side is nonnegative by the monotonicity of $\gamma_\lambda$,
the Young inequality yields, after integrating in time,
\[
\norm{\Psil'(\tilde\phil)}_{L^2(0,T; H)}^2\leq
4\norm{\tilde\mul}_{L^2(0,T; H)}^2 
+6C_\Psi^2\norm{\tilde\phil}_{L^2(0,T; H)}^2,
\]
so that by \eqref{est1_lam}--\eqref{est2_lam} we have
\begin{equation}
\label{est5_lam}
\norm{\Psil'(\tilde\phil)}_{L^{\ell/2}(\tilde\Omega; L^2(0,T; H))}\leq c.
\end{equation}
By comparison we deduce that 
$\Delta\tilde\phil\in L^{\ell/2}(\tilde\Omega; L^2(0,T; H))$ and by elliptic regularity
\begin{equation}\label{est6_lam}
\norm{\tilde\phil}_{L^{\ell/2}(\tilde\Omega; L^2(0,T; V_2))}\leq c.
\end{equation}

\subsection{Step 5: passage to the limit $\lambda \to 0$}
\label{Sec:Step5:LimLambda}
In this step, we pass to the limit as $\lambda\to 0$ in the Yosida approximation \cref{Def:SystemYosida}
and recover a martingale solution to the original problem \eqref{Def:System}.
Since the arguments are very similar to the ones in Subsection~\ref{Sec:Step3:LimN} when we passed to the limit $n \to \infty$,
we shall omit the exact details.

\begin{proof}[Proof of \cref{Theorem:Main}]
By the Aubin--Lions compactness lemma one has the compact inclusions
\begin{align*}
    L^2(0,T; V_2)\cap 
    L^\infty(0,T; V)\cap W^{\bar s,\ell}(0,T; V^*)
    &\hookrightarrow\hookrightarrow C^0([0,T]; H)\cap L^2(0,T; V),\\
    L^2(0,T; V)\cap 
    L^\infty(0,T; H)\cap W^{\bar s,\ell}(0,T; V^*)
    &\hookrightarrow\hookrightarrow C^0([0,T]; V^*)\cap L^2(0,T; H).
\end{align*}
Hence, using the estimates \eqref{est1_lam} and \eqref{est4_lam}
and arguing exactly as in Subsection~\ref{Sec:Step3:LimN},
one readily deduces that the family of laws 
$(\tilde\phil, G_1(\tilde\phil)\cdot \tilde W_1, \tilde W_1)_\lambda$
is tight on the product space
\[
\left(C^0([0,T]; H)\cap L^2(0,T; V)\right)
\times C^0([0,T]; H) \times C^0([0,T]; \tilde U_1),
\]
and the family of laws 
$(\tilde\sigmal, G_2(\tilde\sigmal)\cdot \tilde W_2, \tilde W_2)_\lambda$
is tight on the product space
\[
\left(C^0([0,T]; V^*)\cap L^2(0,T; H)\right) \times C^0([0,T]; H) \times C^0([0,T]; \tilde U_2),
\]
The Prokhorov and Jakubowski--Skorokhod theorems
(see the discussion in \cref{Subsec:Stochastic})
ensure the existence of a further probability space $(\hat\Omega,\hat\calF,\hat\P)$
and measurable 
maps $\Xi_\lambda:(\hat\Omega,\hat\calF)\to(\tilde\Omega,\tilde\calF)$ such that 
$\hat\P\circ \Xi_\lambda^{-1}=\tilde\P$ for every $\lambda>0$ and 
\begin{align*}
	\hat\phil:=\tilde\phil\circ\Xi_\lambda\to \hat\phi
	\quad&\text{in } L^\ell(\hat\Omega;C^0([0,T]; H)
   \cap L^2(0,T; V)) \quad\,\forall\,p<\ell,\\
	\hat\phil\stackrel{*}{\rightharpoonup} \hat\phi 
	\quad&\text{in } L^\ell_w(\hat\Omega; L^\infty(0,T; V)),\\
	\hat\phil\rightharpoonup \hat\phi \quad&\text{in } L^{\ell/2}(\hat\Omega; L^2(0,T; V_2))
	\cap L^\ell(\hat\Omega; W^{\bar s,\ell}(0,T; V^*)),\\
	\hat\mul:=\tilde\mul\circ\Xi_\lambda \rightharpoonup\hat\mu
	\quad&\text{in } L^{\ell/2}(\hat\Omega; L^2(0,T; V)),\\
	\nabla\hat\mul \rightharpoonup\nabla\hat\mu
	\quad&\text{in } L^{\ell}(\hat\Omega; L^2(0,T; H)),\\
 \hat\sigmal:=\tilde\sigmal\circ\Xi_\lambda\to \hat\sigma
	\quad&\text{in } L^\ell(\hat\Omega;C^0([0,T]; V^*)
 \cap L^2(0,T; H)) \quad\,\forall\,p<\ell,\\
	\hat\sigmal\stackrel{*}{\rightharpoonup} \hat\sigma 
	\quad&\text{in } L^\ell_w(\hat\Omega; L^\infty(0,T; H)),\\
	\hat\sigmal\rightharpoonup \hat\sigma \quad&\text{in } 
 L^{\ell}(\hat\Omega; L^2(0,T; V))
	\cap L^\ell(\hat\Omega; W^{\bar s,\ell}(0,T; V^*)),\\
	\Psil'(\hat\phil)\rightharpoonup \hat\xi
	\quad&\text{in } L^{\ell/2}(\hat\Omega; L^2(0,T; H)),
\end{align*}
for some measurable processes
\begin{align*}
	\hat\phi&\in 
    L^\ell(\hat\Omega; W^{\bar s,\ell}(0,T; V^*)
    \cap C^0([0,T]; H))
    \cap L^\ell_w(\hat \Omega; L^\infty(0,T; V))\\
    &\qquad\cap L^{\ell/2}(\hat\Omega; L^2(0,T; V_2)),\\
	\hat\mu &\in L^{\ell/2}(\hat\Omega; L^2(0,T; V)), 
	\quad\nabla\hat\mu\in L^\ell(\hat\Omega; L^2(0,T; H)),\\
    \hat\sigma&\in 
    L^\ell(\hat\Omega; W^{\bar s,\ell}(0,T; V^*)
    \cap C^0([0,T]; V^*)\cap L^2(0,T; V))
    \cap L^\ell_w(\hat\Omega; L^\infty(0,T; H)),\\
	\hat\xi &\in  L^{\ell/2}(\hat\Omega; L^2(0,T; H)),
\end{align*}
as well as
\begin{align*}
	\hat I_{1,\lambda}:=(G_1(\tilde\phil)\cdot \tilde W_1)\circ\Xi_\lambda \to \hat I_1
	\quad&\text{in } L^\ell(\hat\Omega; C^0([0,T]; H)) \quad\forall\,p<\ell,\\
 \hat I_{2,\lambda}:=(G_2(\tilde\sigmal)\cdot \tilde W_2)\circ\Xi_\lambda \to \hat I_2
	\quad&\text{in } L^\ell(\hat\Omega; C^0([0,T]; H)) \quad\forall\,p<\ell,\\
	\hat W_{1,\lambda}:=\tilde W_1\circ\Xi_\lambda \to \hat W_1 
	\quad&\text{in } L^\ell(\hat \Omega; C^0([0,T]; \tilde U_1)) \quad\forall\,p<\ell, \\
 \hat W_{2,\lambda}:=\tilde W_2\circ\Xi_\lambda \to \hat W_2 
	\quad&\text{in } L^\ell(\hat \Omega; C^0([0,T]; \tilde U_2)) \quad\forall\,p<\ell,
\end{align*}
for some measurable processes 
\begin{align*}
	\hat I_1, \hat I_2 &\in L^\ell(\hat \Omega; C^0([0,T]; H)),\\
\hat W_1 &\in L^\ell(\hat\Omega; C^0([0,T]; \tilde U_1)),\\
   \hat W_2 &\in L^\ell(\hat\Omega; C^0([0,T]; \tilde U_2)).
\end{align*}
Since for $i=1,2$ 
$G_i:H\to\calL^2(U_i,H)$ is Lipschitz-continuous, we also have 
\[
\begin{aligned}
G_1(\hat\phil) \to G_1(\hat\phi) \quad\text{in } L^\ell(\hat\Omega; L^2(0,T; \calL^2(U_1,H)))
\quad\forall\,p<\ell, \\
G_2(\hat\sigmal) \to G_2(\hat\sigma) \quad\text{in } L^\ell(\hat\Omega; L^2(0,T; \calL^2(U_2,H)))
\quad\forall\,p<\ell.
\end{aligned}
\]
Moreover, since it holds $\Psi'=\gamma - C_\Psi I$ and $\gamma$ is a maximal monotone function,
by the strong-weak closure of maximal monotone operators (see \cite[Lem.~2.3]{barbu2010nonlinear})
and the strong convergence of $(\hat\phil)_\lambda$,
we have that 
\[
\xi = \Psi'(\hat\phi)  \quad\text{a.e.~in } \hat\Omega\times(0,T)\times D.
\]
Following the same argument as in \cref{Sec:Step3:LimN},
thanks to
the strong convergences of $\hat\phil\to\hat\phi$ and 
$G_1(\hat\phil)\to G_1(\hat\phi)$ proved above,
it is
a classical argument to show that $\hat I_{i,\lambda}$, $i=1,2$, are
the $H$-valued martingales given by
\[
\hat I_{1,\lambda}=\int_0^\cdot G_1(\hat\phil(s))\,
\d \hat W_{1,\lambda}(s), \quad \hat I_{2,\lambda}
=\int_0^\cdot G_2(\hat\sigmal(s))\,
\d \hat W_{2,\lambda}(s),
\]
and 
\[
\hat I_1=\int_0^\cdot G_1(\hat\phi(s))\,\d \hat W_1(s), \quad
\hat I_2=\int_0^\cdot G_2(\hat\sigma(s))\,\d \hat W_2(s).
\]

Now, since $\hat\P\circ\Xi_\lambda=\tilde \P$ for every $\lambda>0$,
it follows that
\begin{align*}
	(\hat\phil(t),v)_D  
	&=(\phi_0,v)_D -(m(\hat\phil)\nabla(\hat\mul-\chi\hat\sigmal),\nabla v)_{D_t}
 \\&\quad+ (\beta\hsigmal-\alpha,f(\hphil)v)_{D_t} +
	\left(\int_0^tG_1(\hat\phil(s))\,\d \hat W_{1,\lambda}(s),v\right)D \\
(\hsigmal(t),v)_D
	&=(\sigma_0,v)_D- (m_2(\hsigmal)
	 \nabla(\hsigmal-\chi\hat\phil),\nabla v)_{D_t}
  \\&\quad- (\delta\hsigmal,f(\hphil)v)_{D_t} + 
	\left(\int_0^tG_2(\hsigmal(s))\,\d \tilde W_2(s),v\right)_D
\end{align*}
for every $t\in[0,T]$, $\hat\P$-almost surely, where
$\hat\mul=-\eps^2\Delta\hat\phil+\Psi_\lambda'(\hat\phil)$.
Hence, 
using the convergences proved above, 
the continuity and boundedness of $m_1$ and $m_2$
together with the 
dominated convergence theorem, we 
can let $\lambda\to 0$ in the variational formulations and
obtain exactly \eqref{Def:SystemVariational}, and $\hat\mu=-\eps^2\Delta\hat\phi+\Psi'(\hat\phi)$.
Finally, the energy inequality \eqref{Eq:FinalEnergyEstimate} as stated in \cref{Theorem:Main} follows from 
applying 
$\Xi_\lambda$ to \eqref{Est:GalerkinFinal} and letting $\lambda\searrow0$ by the weak-lower semicontinuity of the norms.
\end{proof}

\subsection{Strong well-posedness}
We focus here on the proof of the second main result:
let us suppose that the two mobilities are constant, i.e.~$m_1=\bar m_1$
and $m_2=\bar m_2$ some $\bar m_1,\bar m_2>0$.

\begin{proof}[Proof of \cref{Theorem:Main2}]
Let $(\phi_1, \mu_1, \sigma_1),(\phi_2,\mu_2,\sigma_2)$ 
be two solutions defined on the same probability space, 
associated to some initial data 
$(\phi_0^1, \sigma_0^1), (\phi_0^2, \sigma_2)$ satisfying \cref{Ass:Ini}.
The It\^o formulas for the square of the $H$ norms of $\phi_1-\phi_2$
and $\sigma_1-\sigma_2$ yield
\begin{align*}
    &\frac12\|(\phi_1-\phi_2)(t)\|_H^2 +
    \frac12\|(\sigma_1-\sigma_2)(t)\|_H^2
    +\delta\int_{D_t}f(\phi_1)|\sigma_1-\sigma_2|^2\\
    &\qquad+\bar m_1\eps^2\int_0^t\|\Delta(\phi_1-\phi_2)(s)\|_H^2\ds
    +\bar m_2\int_0^t\|\nabla(\sigma_1-\sigma_2)(s)\|_H^2\ds\\
    &=\frac12\|\phi_0^1-\phi_0^2\|_H^2 +
    \frac12\|\sigma_0^1-\sigma_0^2\|_H^2
    +\chi(\bar m_1+\bar m_2)\int_{D_t}\nabla(\sigma_1-\sigma_2)\cdot\nabla(\phi_1-\phi_2)\\
    &\qquad+\beta\int_{D_t}f(\phi_1)(\sigma_1-\sigma_2)(\phi_1-\phi_2)
    -\alpha\int_{D_t}(f(\phi_1)-f(\phi_2))(\phi_1-\phi_2)\\
    &\qquad+\int_{D_t}\sigma_2(f(\phi_1)-f(\phi_2))
    \left[\beta(\phi_1-\phi_2)-\delta(\sigma_1-\sigma_2)\right]\\
    &\qquad-\int_{D_t}(\Psi'(\phi_1)-\Psi'(\phi_2))\Delta(\phi_1-\phi_2)\\
    &\qquad+\frac12\int_0^t\left[
    \|G_1(\phi_1)-G_1(\phi_2)\|_{\mathcal L^2(U_1,H)}^2+
    \|G_2(\sigma_1)-G_2(\sigma_2)\|_{\mathcal L^2(U_2,H)}^2
    \right]\ds\\
    &\qquad+\int_0^t\left((\phi_1-\phi_2)(s),
    (G_1(\phi_1(s))-G_1(\phi_2(s)))\d W_1(s)\right)_D\\
    &\qquad+\int_0^t\left((\sigma_1-\sigma_2)(s),
    (G_2(\sigma_1(s))-G_2(\sigma_2(s)))\d W_2(s)\right)_D.
\end{align*}
Let us estimate the terms on the right-hand side separately.
First,  since $V_2$ is compact in $V$, we have 
\begin{align*}
    \chi(\bar m_1+\bar m_2)\int_{D_t}\nabla(\sigma_1-\sigma_2)\cdot\nabla(\phi_1-\phi_2)&\leq
    \frac{\bar m_2}2\int_0^t\|\nabla(\sigma_1-\sigma_2)(s)\|_H^2\ds\\
    &+\frac{\bar m_1\eps^2}4\int_0^t\|\Delta(\phi_1-\phi_2)(s)\|_H^2\ds\\
    &+c\int_0^t\|(\phi_1-\phi_2)(s)\|_H^2\ds.
\end{align*}
Moreover, since $f$ is Lipschitz and bounded, the fourth and fifth
integrals on the right-hand side are controlled, thanks to the Young inequality, by 
\begin{align*}
  &\beta\int_{D_t}f(\phi_1)(\sigma_1-\sigma_2)(\phi_1-\phi_2)
    -\alpha\int_{D_t}(f(\phi_1)-f(\phi_2))(\phi_1-\phi_2)\\
  &\leq c\int_0^t\left[\|(\sigma_1-\sigma_2)(s)\|_H^2
  +\|(\phi_1-\phi_2)(s)\|_H^2\right]\ds.
\end{align*}
Furthermore, as for the sixth term, we use the H\"older inequality,
the Lipschitz-continuity of $f$,
the inclusion $H^{\frac74}(D)\hookrightarrow L^\infty(D)$, 
the compact inclusion $V_2\hookrightarrow H^{\frac74}(D)$, and obtain
\begin{align*}
    &\int_{D_t}\sigma_2(f(\phi_1)-f(\phi_2))
    \left[\beta(\phi_1-\phi_2)-\delta(\sigma_1-\sigma_2)\right]\\
    &\leq c\int_0^t\|\sigma_2(s)\|_{H}
    \|(\phi_1-\phi_2)\|_{H^{\frac74}(D)}
    \left[\|(\phi_1-\phi_2)\|_{H}+
    \|(\sigma_1-\sigma_2)\|_{H}\right]\ds\\
    &\leq \int_0^t\|(\phi_1-\phi_2)\|_{H^{\frac74}(D)}^2\ds+
    c\int_0^t\|\sigma_2(s)\|_{H}^2
    \left[\|(\phi_1-\phi_2)\|_{H}^2+
    \|(\sigma_1-\sigma_2)\|_{H}^2\right]\ds\\
    &\leq \frac{\bar m_1\eps^2}4\int_0^t\|\Delta(\phi_1-\phi_2)(s)\|_H^2\ds\\
    &\qquad+
    c\int_0^t(1+\|\sigma_2(s)\|_{H}^2)
    \left[\|(\phi_1-\phi_2)\|_{H}^2+
    \|(\sigma_1-\sigma_2)\|_{H}^2\right]\ds.
\end{align*}
The seventh term can be handled by using the mean-value theorem,
the H\"older inequality, 
the inclusion $H^{\frac74}(D)\hookrightarrow L^\infty(D)$,
and the fact that $|\Psi''|\leq C_\Psi(1+|\Psi|)$, as
\begin{align*}
  &-\int_{D_t}(\Psi'(\phi_1)-\Psi'(\phi_2))\Delta(\phi_1-\phi_2)\\
  &\leq\frac{\bar m_1\eps^2}4\int_0^t\|\Delta(\phi_1-\phi_2)(s)\|_H^2\ds\\
  &\qquad+
  c\int_0^t\left(1+\|\Psi(\phi_1(s))\|_{L^1(D)}^2
  +\|\Psi(\phi_2(s))\|_{L^1(D)}^2\right)
  \|(\phi_1-\phi_2)(s)\|_{H^{\frac74}(D)}^2\ds.
\end{align*}
Eventually, by the Lipschitz continuity of $G_1$ and $G_2$ one has
\begin{align*}
    &\frac12\int_0^t\left[
    \|G_1(\phi_1)-G_1(\phi_2)\|_{\mathcal L^2(U_1,H)}^2+
    \|G_2(\sigma_1)-G_2(\sigma_2)\|_{\mathcal L^2(U_2,H)}^2
    \right]\ds\\
    &\leq c\int_0^t\left(
    \|(\phi_1-\phi_2)(s)\|_{H}^2+
    \|(\sigma_1-\sigma_2)(s)\|_H^2
    \right)\ds.
\end{align*}
Similarly, for every stopping  time $\tau\in[0,T]$,
the Burkholder--Davis--Gundy and Young inequalities yield
\begin{align*}
    &\E\sup_{s\in[0,\tau]}\int_0^s\left((\phi_1-\phi_2)(r),
    (G_1(\phi_1(r))-G_1(\phi_2(r)))\d W_1(r)\right)_D\\
    &\leq\frac14\E\sup_{s\in[0,\tau]}\|(\phi_1-\phi_2)(s)\|_{H}^2
    +c\E\int_0^\tau\|(\phi_1-\phi_2)(s)\|_{H}^2\ds
\end{align*}
and
\begin{align*}
    &\E\sup_{s\in[0,\tau]}\int_0^s\left((\sigma_1-\sigma_2)(r),
    (G_2(\sigma_1(r))-G_2(\sigma_2(r)))\d W_2(r)\right)_D\\
    &\leq\frac14\E\sup_{s\in[0,\tau]}\|(\sigma_1-\sigma_2)(s)\|_{H}^2
    +c\E\int_0^\tau\|(\sigma_1-\sigma_2)(s)\|_{H}^2\ds.
\end{align*}
We define now the sequence of stopping times
\[
  \tau_n:=\inf\left\{s\in[0,T]: \|\sigma_2(s)\|_H^2
  +\|\Psi(\phi_1(s))\|_{L^1(D)} + 
  \|\Psi(\phi_2(s))\|_{L^1(D)} \geq n\right\}, \quad n\in\mathbb N,
\]
which is well-defined and satisfies $\tau_n\nearrow T$ almost surely
as $n\to\infty$ since $\sigma_2 \in L^\infty(0,T; H)$
and $\Psi(\phi_1),\Psi(\phi_2)\in L^\infty(0,T; L^1(D))$ almost surely.

Putting everything together, by taking supremum in time $t\in[0,\tau_n]$
and then expectations,
we get, after rearranging the terms,
\begin{align*}
    &\frac14\E\sup_{s\in[0,\tau_n]}\|(\phi_1-\phi_2)(s)\|_{H}^2
    +\frac14\E\sup_{s\in[0,\tau_n]}\|(\sigma_1-\sigma_2)(s)\|_{H}^2\\
    &\qquad+\frac{\bar m_1\eps^2}4\E
    \int_0^{\tau_n}\|\Delta(\phi_1-\phi_2)(s)\|_H^2\ds
    +\frac{\bar m_2}2\E
    \int_0^{\tau_n}\|\nabla(\sigma_1-\sigma_2)(s)\|_H^2\ds\\
    &\leq\frac12\|\phi_0^1-\phi_0^2\|_H^2 +
    \frac12\|\sigma_0^1-\sigma_0^2\|_H^2\\
    &\qquad+c\E\int_0^{\tau_n}\left(1+\|\sigma_2(s)\|_H^2\right)
    \left(\|(\phi_1-\phi_2)(s)\|_H^2+ 
    \|(\sigma_1-\sigma_2)(s)\|_H^2\right)\ds\\
    &\qquad+
    c\E\int_0^{\tau_n}\left(1+\|\Psi(\phi_1(s))\|_{L^1(D)}^2
  +\|\Psi(\phi_2(s))\|_{L^1(D)}^2\right)
  \|(\phi_1-\phi_2)(s)\|_{H^{\frac74}(D)}^2\ds.
\end{align*}
By definition of $\tau_n$ and again the compact inclusion 
$V_2\hookrightarrow H^{\frac74}(D)$,
this implies that for every $n\in\mathbb N$
there exists a constant $c_n>0$ such that 
\begin{align*}
    &\frac14\E\sup_{s\in[0,\tau_n]}\|(\phi_1-\phi_2)(s)\|_{H}^2
    +\frac14\E\sup_{s\in[0,\tau_n]}\|(\sigma_1-\sigma_2)(s)\|_{H}^2\\
    &\qquad+\frac{\bar m_1\eps^2}4\E
    \int_0^{\tau_n}\|\Delta(\phi_1-\phi_2)(s)\|_H^2\ds
    +\frac{\bar m_2}2\E
    \int_0^{\tau_n}\|\nabla(\sigma_1-\sigma_2)(s)\|_H^2\ds\\
    &\leq\frac12\|\phi_0^1-\phi_0^2\|_H^2 +
    \frac12\|\sigma_0^1-\sigma_0^2\|_H^2\\
    &\qquad+c_n\E\int_0^{\tau_n}
    \left(\|(\phi_1-\phi_2)(s)\|_{H^{\frac74}(D)}^2+ 
    \|(\sigma_1-\sigma_2)(s)\|_H^2\right)\ds\\
    &\leq\frac12\|\phi_0^1-\phi_0^2\|_H^2 +
    \frac12\|\sigma_0^1-\sigma_0^2\|_H^2
    +\frac{\bar m_1\eps^2}8\E
    \int_0^{\tau_n}\|\Delta(\phi_1-\phi_2)(s)\|_H^2\ds\\
    &\qquad+c_n\E\int_0^{\tau_n}
    \left(\|(\phi_1-\phi_2)(s)\|_{H}^2+ 
    \|(\sigma_1-\sigma_2)(s)\|_H^2\right)\ds.
\end{align*}
Rearranging the terms and using the Gronwall lemma yields the conclusion.
\end{proof}
  \section{Numerical scheme and simulations}
\label{Sec:Numerical}

In this section, we delve into the details of our numerical discretization scheme (\cref{Subsec:Discrete}) and the resulting simulations on tumor growth (\cref{Subsec:Simul}). 

\subsection{Discretization scheme} \label{Subsec:Discrete}
For an overview of numerical methods for SPDEs, we refer to the books \cite{lord2014introduction,zhang2017numerical,jentzen2011taylor}. Regarding the spatial approximation of the stochastic tumor model, we introduce a finite-dimensional space $V_N \subset V$ and denote by $\Pi_N$ the orthogonal projection $\Pi_N:H \to V_N$. For each variable, we select the bilinear rectangular finite element space $Q_1$. We seek an approximation $(\phi,\mu,\sigma) \in V_N \times V_N \times V_N$, defined by
\begin{equation*} 
\begin{aligned}
  (\d\phi,v)_D  ={}& (m_1(\phi) (\nabla \mu- \chi \nabla\sigma),\nabla v)_{D} + ((\beta \sigma-\alpha) f(\phi),v)_{D} + (  G_1(\phi)\dW_1,v )_D, \\
  (\mu,v)_{D} ={}& (\Psi'(\phi),v)_{D} +\eps^2 (\nabla\phi,\nabla v)_{D}   \\
  (\d\sigma,v)_D  ={}& (m_2(\sigma) (\nabla \sigma- \chi \nabla\phi),\nabla v)_{D} - \delta (\sigma f(\phi),v)_{D} + ( G_2(\sigma\dW_2,v)_D,
\end{aligned}
\end{equation*}
for any $v \in V_N$ and $t \in [0,T]$, equipped with the initial conditions $\phi(0)=\Pi_N  \phi_0$ and $\sigma(0)=\Pi_N  \sigma_0$. 
We write the solution in terms of the finite element basis $(w_j)_j$ as follows
$$\phi(t,x)=\sum_{j=1}^N \phi_j(t) w_j(x), \quad \mu(t,x)=\sum_{j=1}^N \mu_j(t) w_j(x), \quad \sigma(t,x)=\sum_{j=1}^N \sigma_j(t) w_j(x),$$
and we write the coefficient functions in the vector $\bm{\phi}=[\phi_1,\dots,\phi_N]^\top$ and for $\bm{\sigma}$ and $\bm{\mu}$ analogously. Then inserting the ansatz into the system, it gives
\begin{equation}
\begin{aligned}
  M \d \bm{\phi}  ={}& \big[K_{m_1(\phi)} \bm{\mu} - \chi K_{m_1} \bm{\sigma} + \beta M_{f(\phi)} \bm{\sigma} - \alpha \bm{f}(\bm{\phi}) \big] \dt+ \bm{G_1}(\bm{\phi}) \dW_1(t)  \\
  M\bm{\mu} ={}& \bm{\Psi}'(\bm{\phi}) + \eps^2 K \bm{\phi},   \\
  M \d \bm{\sigma}  ={}& \big[K_{m_2(\sigma)} \bm{\sigma} - \chi K_{m_1} \bm{\phi} - \delta M_{f(\phi)} \bm{\sigma}  \big] \dt+ \bm{G_2}(\bm{\sigma}) \dW_2(t)  \\
\end{aligned}
\end{equation}
where $M=[(w_i,w_j)_D]_{i,j=1,\dots,N}$ is the mass matrix, $M_{f(\bm{\phi})}=[(f(\bm{\phi})w_i,w_j)_D]_{i,j=1,\dots,N}$ the weighted mass matrix, $K=[(\nabla w_i,\nabla w_j)_D]_{i,j=1,\dots,N}$ the stiffness matrix, $K_{m_1(\bm{\phi})}$ the weighted stiffness matrix, $\bm{f}(\bm{\phi})=[(f(\phi),w_j)_D]_{j=1,\dots,N}$, and analogously for $\bm{\Psi}$. Finally, $\bm{G_1}:\bbR^N\to \calL(U_1,\bbR^N)$ is defined by $\bm{G_1}(\bm{\phi}) u = [(G_1(\phi)u,w_j)_D]_{j=1,\dots,N}$ for $u \in U_1$ and in the same way for $G_2(\sigma)$.

For the discretization in time, we approximate $\bm{\phi}(t)$ at $t=t_n=n\Delta t$ by $\bm{\phi}_n$. Because $V_N$ is finite-dimensional, the system consists of SODEs, and we can apply the semi-implicit Euler--Maruyama method with a time step $\Delta t > 0$ to define $\bm{\phi}_n$. That is, given $(\bm{\phi}_n,\bm{\mu}_n,\bm{\sigma}_n)$, we find $(\bm{\phi}_{n+1},\bm{\mu}_{n+1},\bm{\sigma}_{n+1})$ by iterating:
\begin{equation}
\begin{aligned}
  M \bm{\phi}_{n+1}  ={}& M \bm{\phi}_{n}+K_{m_1^n} \bm{\mu}_{n+1} - \chi K_{m_1^n} \bm{\sigma}_{n+1} + \beta M_{f^n} \bm{\sigma}_n - \alpha \bm{f}(\bm{\phi}_n) + \bm{G_1}(\bm{\phi}_n) \Delta \bm{W}_{1,n}  \\
  M\bm{\mu}_{n+1} ={}& \bm{\Psi_c}'(\bm{\phi}_{n+1})+\bm{\Psi_e}'(\bm{\phi}_n) + \eps^2 K \bm{\phi}_{n+1},   \\
  M \bm{\sigma}_{n+1}  ={}& M \bm{\sigma}_{n}+K_{m_2^n} \bm{\sigma}_{n+1} - \chi K_{m_1^n} \bm{\phi}_{n+1} - \delta M_{f^n} \bm{\sigma}_n + \bm{G_2}(\bm{\phi}_n) \Delta\bm{W}_{2,n}  \\
\end{aligned}
\end{equation}
for initial data $\bm{\phi}_0 = \bm{\phi}(0)$ and $\bm{\sigma}_0 = \bm{\sigma}(0)$. Here, we introduced $m_1^n=m_1(\bm{\phi}_{n})$, $\bm{G_1}(\bm{\phi}_n) \in \R^{N\times N}$ with $\bm{G_1}(\bm{\phi}_n)=[(G(\bm{\phi}_n)u_k,w_j)_D]_{j,k=1,\dots,N}$ and $\Delta \bm{W}_{1,n} \in \R^N$ with $\Delta \bm{W}_{1,n}=[(W_1(t_{n+1})-W_1(t_n),u_k)_D]_{k=1,\dots,N}$. and in the same way for $G_2$ and $W_2$. Having the Wiener process $W_1(t)=\sum_{j=1}^\infty \sqrt{q_j^1} e_j^1 \beta_j^1(t)$ where $(e_j^1)_j$ is an ONB for $U_1$, we compute $\bm{G_1}(\bm{\phi}_n) \Delta \bm{W_{1,n}}$ by multiplying the matrix   $\bm{G_1}(\bm{\phi}_n) \in \R^{N\times N}$ by the vector
$$\big[\sqrt{q_1}(\beta_1(t_{n+1}) - \beta_1(t_n)), \ldots, \sqrt{q_{N}}(\beta_{N}(t_{n+1}) - \beta_{N}(t_n))\big]^\top
$$
In fact, we assume that $\bm{G_1}(\bm{\phi}_n)$ is a diagonal matrix with entries $g_1(\bm{\phi}_n^k)$ where $\bm{\phi}_n^k$ is the $k$-th entry of $\bm{\phi}_n$. The same is assumed for $G_2$. We use the classical convex-concave splitting of the nonlinear potential $\Psi=\Psi_e+\Psi_c$ into its expansive $\Psi_e$ and contractive part $\Psi_c$. In fact, in the case of $\Psi(\phi)=\frac14\phi^2(1-\phi^2)$ we set
$\Psi_e(\phi)=\phi^3-\frac32\phi^2-\frac14\phi$ and $\Psi_c(\phi)=\frac34\phi$.
We treat the expansive part explicitly and the contractive part implicitly. 
Like this, we obtain an unconditionally stable scheme in the deterministic case, see \cite{eyre1998unconditionally}, which is additionally linear.

The nonlinear equations are solved by a Newton method in each iterative decoupling-iteration. The system is implemented in the finite element library FEniCS \cite{alnaes2015fenics}.

\subsection{Simulations} \label{Subsec:Simul}

We set the model parameters to $\eps=0.01$, $\chi=5$, $\alpha=1$, $\beta=15$, $\delta=100 $. Moreover, we choose the nonlinear potential $\Psi$ as above and the mobility functions as $m_1(\phi)=10^{-16}+\phi^2(1-\phi)^2$ and $m_2(\sigma)=10$.  We choose the spatial domain $D=[0,1]^2$ with $\dx=0.01$ and the time domain $[0,1]$ with $\dt=0.01$. The initial for the tumor volume fraction is chosen as $\phi_0(x)=\exp(1-1/(1-16|x-\frac12|^2))$. The nutrient's initial is set to zero but we choose the boundary condition $\sigma=1$ on $\partial D$ in contrast to our mathematical analysis. However, the analysis can be carried out in the same way, e.g., see \cite{fritz2019unsteady} in the case of the deterministic setting.

Regarding the stochasticity in the system, we consider $G_1(\phi)=\nu \phi_+(1-\phi)_+$ for different noise levels $\nu$ in the tumor equation. Here, the subscript $\cdot_+$ indicates a cut-off in the sense of $\phi_+= \max\{0,\min\{1,\phi\}\}$ that ensures the boundedness of $\phi$ between $0$ and $1$. We motivate this choice since we only want to have noise in the interface since should be a fundamental part of the tumor's growth process and should not affect fully malignant or healthy cells. Moreover, we consider an additive noise $G_2(\sigma)=1$ in the nutrients. 
 \medskip

\noindent{\bf Test 1 (Tumor and nutrient mass over time).}
To comprehensively capture the influence of the noise effects, we conducted 50 simulations, each yielding a unique result. In \cref{Fig:TumorVolume}, we present the mean (average) and standard deviation (a measure of the variability) of these 50 samples for two distinct noise levels. For the low noise level simulation, we observe that the mean tumor volume increases over time, which aligns with our baseline expectations from the deterministic setting, e.g., see \cite{fritz2023tumor}. However, the introduction of noise has a notable impact. In the high noise level simulation, the mean tumor volume exhibits a more pronounced increase. This is because the stochasticity introduced by the noise acts as a proliferation term in the equation, essentially furthering tumor growth by introducing additional uncertainty into the system. The standard deviation, which quantifies the variability in tumor volume across the 50 samples, reflects the impact of noise. In the case of the low noise level, the standard deviation is smaller and is in close proximity to the mean. This indicates a relatively consistent and predictable tumor growth pattern with less uncertainty. In contrast, the high noise level simulation yields a significantly larger standard deviation, indicating a wider range of possible outcomes, reflecting the increased variability and uncertainty introduced by the noise.

\begin{figure}[H]
	\begin{tikzpicture}
	\definecolor{color0}{HTML}{4E79A7}
	\definecolor{color1}{HTML}{F28E2B}
	\definecolor{color2}{HTML}{E15759}
	\definecolor{color3}{HTML}{76B7B2}
	\definecolor{color4}{HTML}{59A14F}
	\definecolor{color5}{HTML}{EDC948}
	\definecolor{color6}{HTML}{B07AA1}
	\definecolor{color7}{HTML}{FF9DA7}
	\definecolor{color8}{HTML}{9C755F}
	\definecolor{color9}{HTML}{BAB0AC}
	\begin{axis}[
		height=.38\textheight,
        width=.57\textwidth,
		legend pos=north west,
		legend cell align={left},
        ymin=0.079, ymax=0.35,
		xmin=0, xmax=1,
		xtick={0,0.2,0.4,0.6,0.8,1.0},
		]
  		\addplot[ultra thick,color0,smooth,each nth point=2] table [x expr=\coordindex/100, y index=0]{Figures/mean.txt};
  				\addlegendentry{mean}
  		\addplot[name path=us_top,ultra thick,color1,smooth,each nth point=2] table [x expr=\coordindex/100, y index=0]{Figures/meanplustd.txt};
  		\addplot[name path=us_down,ultra thick,color1,smooth,each nth point=2] table [x expr=\coordindex/100, y index=0]{Figures/meanminstd.txt};
		    \addplot[color1!50,fill opacity=0.5] fill between[of=us_top and us_down];
		    \addlegendentry{standard deviation}
	\end{axis}
\end{tikzpicture}	\begin{tikzpicture}
	\definecolor{color0}{HTML}{4E79A7}
	\definecolor{color1}{HTML}{F28E2B}
	\definecolor{color2}{HTML}{E15759}
	\definecolor{color3}{HTML}{76B7B2}
	\definecolor{color4}{HTML}{59A14F}
	\definecolor{color5}{HTML}{EDC948}
	\definecolor{color6}{HTML}{B07AA1}
	\definecolor{color7}{HTML}{FF9DA7}
	\definecolor{color8}{HTML}{9C755F}
	\definecolor{color9}{HTML}{BAB0AC}
	\begin{axis}[
		height=.38\textheight,
        width=.56\textwidth,
		legend pos=north west,
		legend cell align={left},
        ymin=0.079, ymax=0.35,
		xmin=0, xmax=1,
		xtick={0,0.2,0.4,0.6,0.8,1.0},
		yticklabels={},
		]
  		\addplot[ultra thick,color0,smooth,each nth point=2] table [x expr=\coordindex/100, y index=0]{Figures/mean2.txt};
  				\addlegendentry{mean}
  		\addplot[name path=us_top,ultra thick,color1,smooth,each nth point=2] table [x expr=\coordindex/100, y index=0]{Figures/meanplustd2.txt};
  		\addplot[name path=us_down,ultra thick,color1,smooth,each nth point=2] table [x expr=\coordindex/100, y index=0]{Figures/meanminstd2.txt};
		    \addplot[color1!50,fill opacity=0.5] fill between[of=us_top and us_down];
		    \addlegendentry{standard deviation}
	\end{axis}
\end{tikzpicture}
	\caption{Considering 50 samples, the mean and standard deviation  of the tumor volume $t \mapsto \int_D \phi(t,x) \dx$ is shown for the noise intensities $\nu=2.5$ (left) and $\nu=0.5$ (right). \label{Fig:TumorVolume}
	}
\end{figure}

In \cref{Fig:NutrientVolume}, we delve into the dynamic evolution of nutrient volume over time. As in \cref{Fig:TumorVolume}, we present both the mean and standard deviation of 50 simulations for each noise level. Notably, the standard deviation reveals a striking difference between the low and high noise level cases. For the low noise level, the standard deviation is relatively smaller, indicating a more consistent and predictable behavior of nutrient volume. In contrast, the high noise level simulation yields a significantly larger standard deviation, reflecting the increased variability and uncertainty introduced by the noise. While the means of nutrient volume in both low and high noise level cases initially show similar behavior, they exhibit distinctions over time. In both scenarios, nutrient volume increases initially due to the ongoing diffusion process, where nutrients move from the boundary and spread throughout the entire domain. However, the key divergence occurs as tumor growth progresses. The nutrient volume's mean appears more damped after the initial increase, primarily because of the influence of tumor growth, which results in a decrease of the nutrient level. The damped effect is more pronounced in the high noise level simulation, where tumor growth is more substantial. 

\begin{figure}[H]
	\begin{tikzpicture}
	\definecolor{color0}{HTML}{4E79A7}
	\definecolor{color1}{HTML}{F28E2B}
	\definecolor{color2}{HTML}{E15759}
	\definecolor{color3}{HTML}{76B7B2}
	\definecolor{color4}{HTML}{59A14F}
	\definecolor{color5}{HTML}{EDC948}
	\definecolor{color6}{HTML}{B07AA1}
	\definecolor{color7}{HTML}{FF9DA7}
	\definecolor{color8}{HTML}{9C755F}
	\definecolor{color9}{HTML}{BAB0AC}
	\begin{axis}[
		height=.38\textheight,
        width=.57\textwidth,
		legend pos=north west,
		legend cell align={left},
        ymin=0.757, ymax=0.767,
        ytick ={0.755, 0.76, 0.765, 0.77},
		xmin=0.01, xmax=1,
		xtick={0.2,0.4,0.6,0.8,1.0},
		]
  		\addplot[ultra thick,color0,smooth,each nth point=2] table [x expr=0.01+\coordindex/99, y index=0]{Figures/meannut.txt};
  				\addlegendentry{mean}
  		\addplot[name path=us_top,ultra thick,color1,smooth,each nth point=2] table [x expr=0.01+\coordindex/99, y index=0]{Figures/meanplustdnut.txt};
  		\addplot[name path=us_down,ultra thick,color1,smooth,each nth point=2] table [x expr=0.01+\coordindex/99, y index=0]{Figures/meanminstdnut.txt};
		    \addplot[color1!50,fill opacity=0.5] fill between[of=us_top and us_down];
		    \addlegendentry{standard deviation}
	\end{axis}
\end{tikzpicture}	\begin{tikzpicture}
	\definecolor{color0}{HTML}{4E79A7}
	\definecolor{color1}{HTML}{F28E2B}
	\definecolor{color2}{HTML}{E15759}
	\definecolor{color3}{HTML}{76B7B2}
	\definecolor{color4}{HTML}{59A14F}
	\definecolor{color5}{HTML}{EDC948}
	\definecolor{color6}{HTML}{B07AA1}
	\definecolor{color7}{HTML}{FF9DA7}
	\definecolor{color8}{HTML}{9C755F}
	\definecolor{color9}{HTML}{BAB0AC}
	\begin{axis}[
		height=.38\textheight,
        width=.56\textwidth,
		legend pos=north west,
		legend cell align={left},
        ymin=0.757, ymax=0.767,
	ytick ={0.755, 0.76, 0.765, 0.77},
		xmin=0.01, xmax=1,
		xtick={0.2,0.4,0.6,0.8,1.0},
		yticklabels={},
		]
  		\addplot[ultra thick,color0,smooth,each nth point=2] table [x expr=0.01+\coordindex/99, y index=0]{Figures/meannut2.txt};
  				\addlegendentry{mean}
  		\addplot[name path=us_top,ultra thick,color1,smooth,each nth point=2] table [x expr=0.01+\coordindex/99, y index=0]{Figures/meanplustdnut2.txt};
  		\addplot[name path=us_down,ultra thick,color1,smooth,each nth point=2] table [x expr=0.01+\coordindex/99, y index=0]{Figures/meanminstdnut2.txt};
		    \addplot[color1!50,fill opacity=0.5] fill between[of=us_top and us_down];
		    \addlegendentry{standard deviation}
	\end{axis}
\end{tikzpicture}
	\caption{Considering 50 samples, the mean and standard deviation  of the nutrient volume $t \mapsto \int_D \sigma(t,x) \dx$ is shown for the noise intensities $\nu=2.5$ (left) and $\nu=0.5$ (right). \label{Fig:NutrientVolume}}
\end{figure}
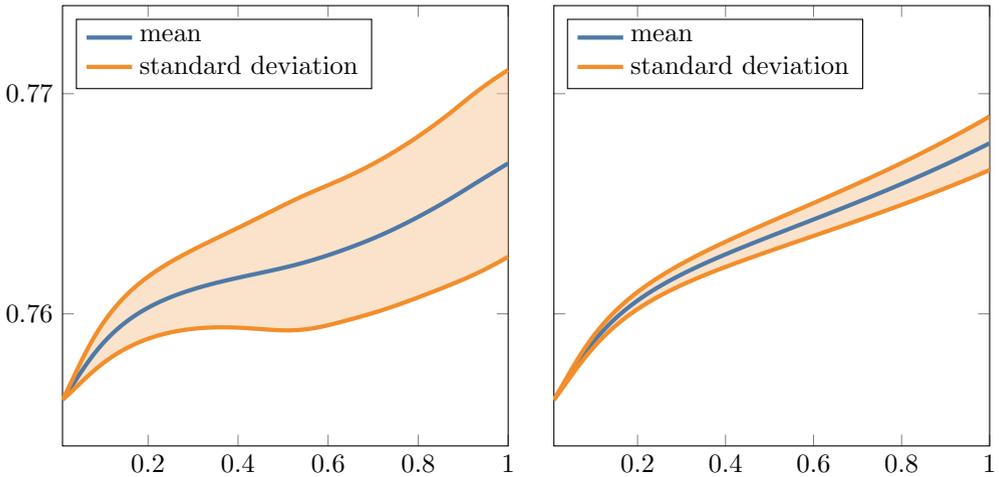

\noindent{\textbf{Test 2 (Tumor and nutrient evolution for different noise levels but same seed).}}
\cref{Fig:TumorSameSeed} presents a series of visual representations showcasing the dynamic evolution of a tumor over several time steps. What makes this analysis distinct is the introduction of four different noise levels $\nu \in \{0,0.5,1,2.5\}$. Crucially, to maintain predictability and consistency in the randomness, the same seed for generating random effects was employed in all cases. This ensures that the tumor's movement remains similar across different noise levels. In the most-left row of \cref{Fig:TumorSameSeed}, we provide a reference point by simulating the tumor evolution in a deterministic setting, where no noise is introduced. In this case, the tumor maintains a perfectly round, circular shape throughout all time steps. This behavior aligns with existing knowledge and expectations, e.g., see \cite{fritz2023tumor}. The other rows of \cref{Fig:TumorSameSeed} represent the introduction of stochasticity through low, moderate, and high noise levels. As the noise levels increase, the tumor's shape becomes increasingly wobbly. The level of irregularity is more pronounced in the high noise level case, to the extent that it no longer resembles its initial circular shape. The primary observation here is the significant impact of stochasticity on tumor shape evolution. In the presence of noise, the tumor's behavior departs from the deterministic circular shape and takes on an irregular form. This is particularly evident at higher noise levels, where the tumor's movement becomes less predictable and more influenced by random effects.

\begin{figure}[H] \begin{center}
		\begin{tabular}{cM{.19\textwidth}M{.19\textwidth}M{.19\textwidth}M{.19\textwidth}}
			&
			\quad $\nu=0.0$&\quad $\nu=0.5$&\quad $\nu=1.0$&\quad $\nu=2.5$ \\
				\!\!\!\!\!\!$t=0.4$\!\!\!\!\!&
			\includegraphics[width=.22\textwidth]{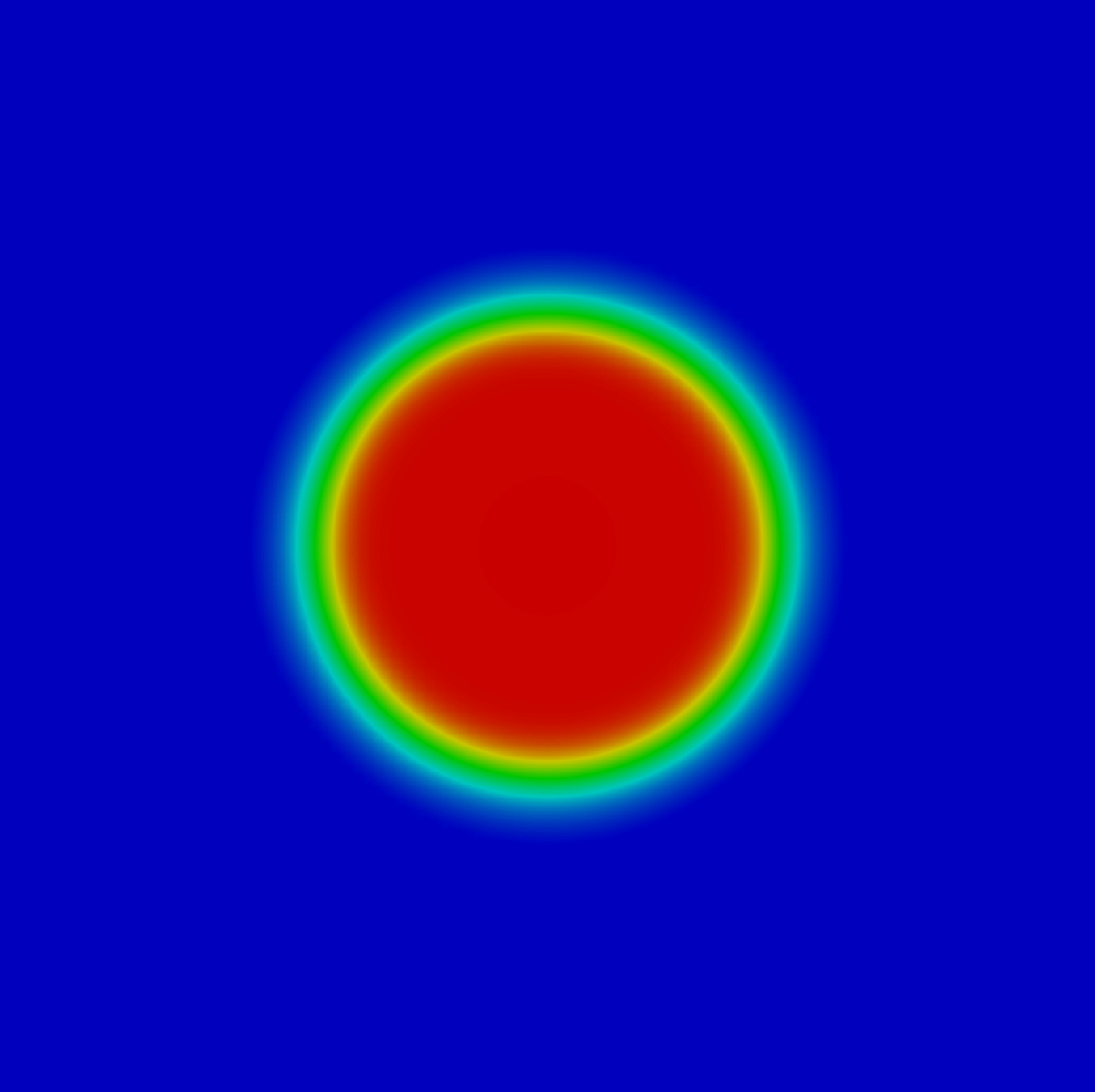}&
			\includegraphics[width=.22\textwidth]{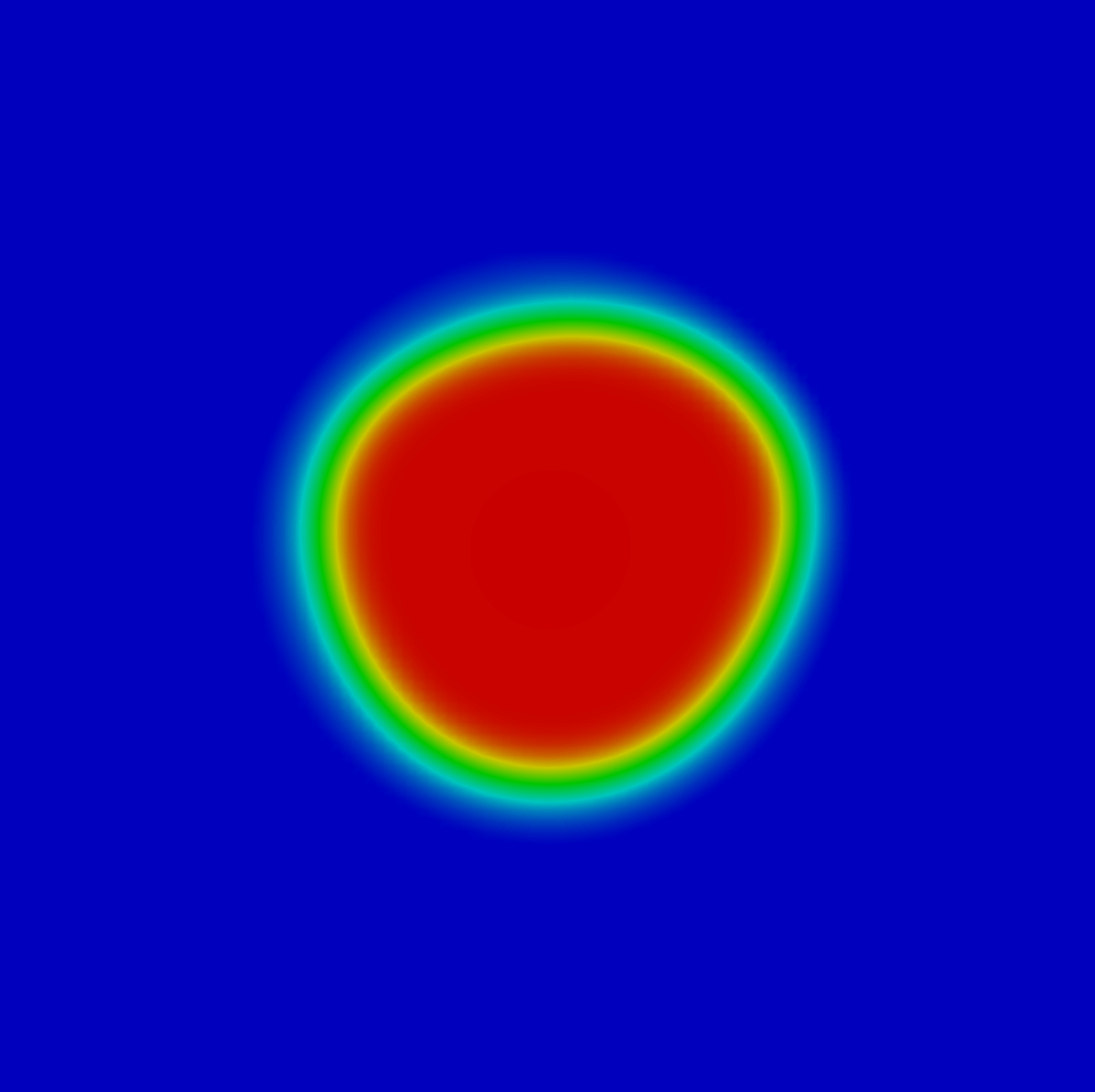}&
			\includegraphics[width=.22\textwidth]{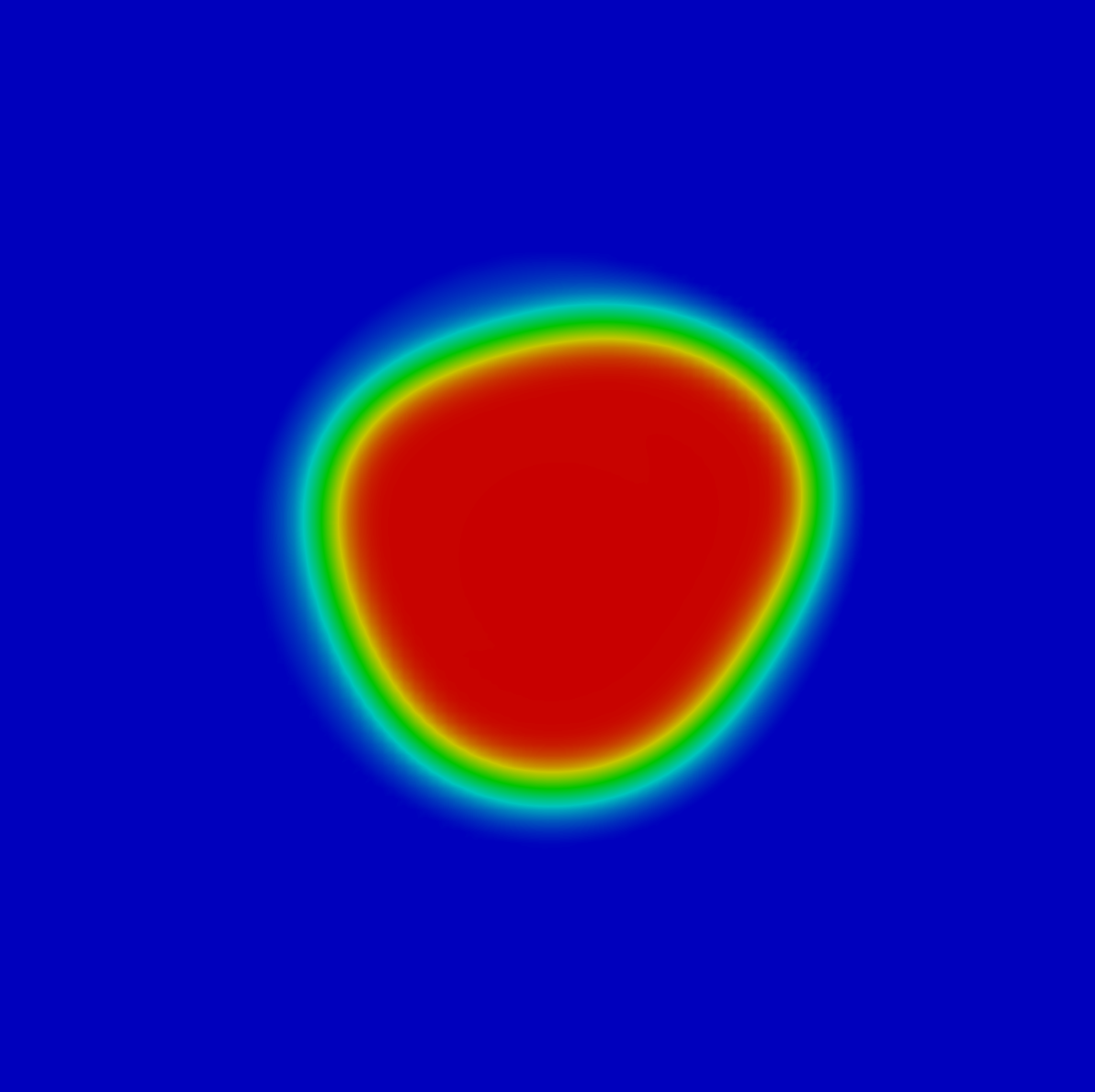}&
			\includegraphics[width=.22\textwidth]{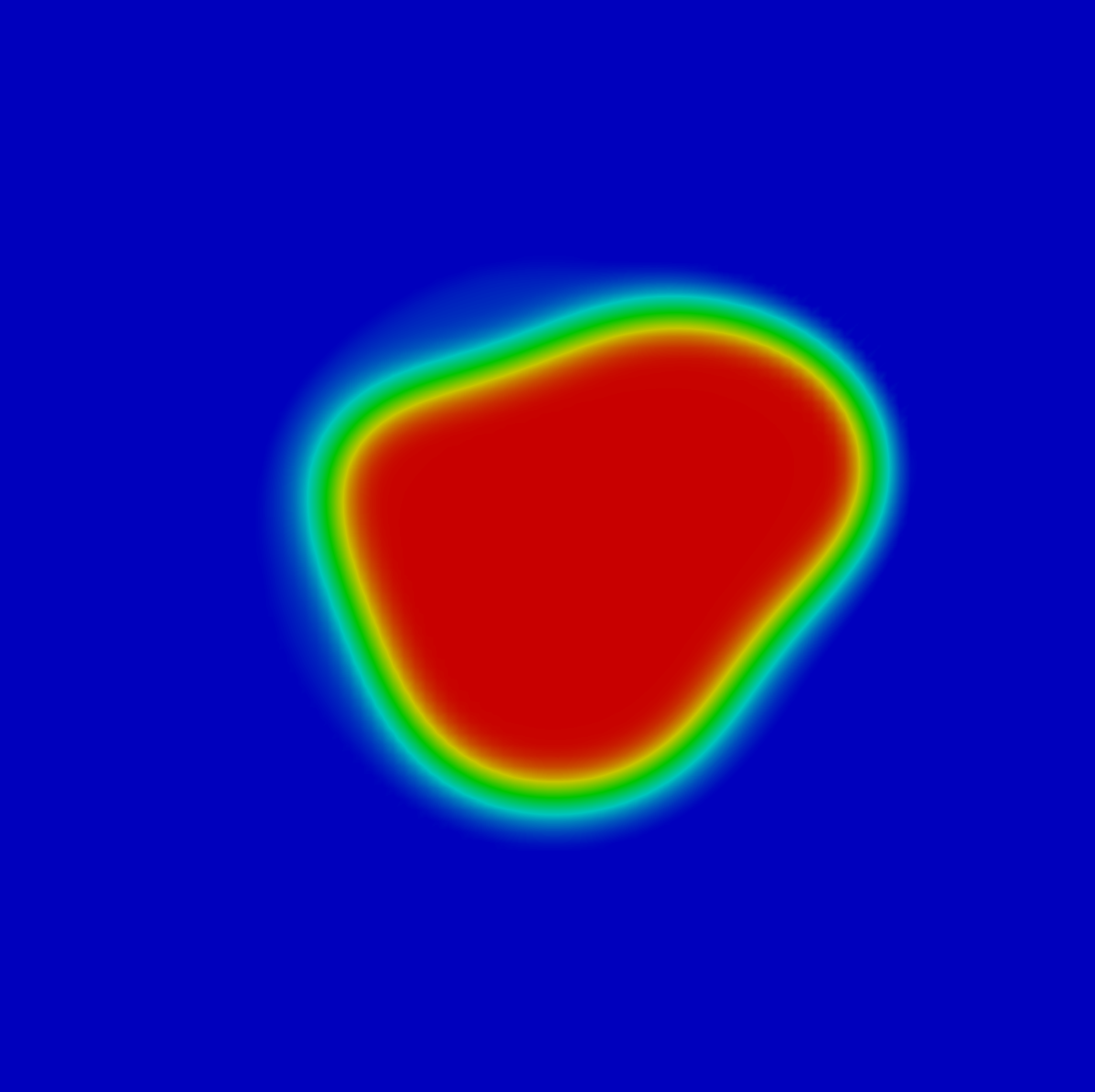}\\[-.1cm]
				\!\!\!\!\!\!$t=1.0$\!\!\!\!\!&
			\includegraphics[width=.22\textwidth]{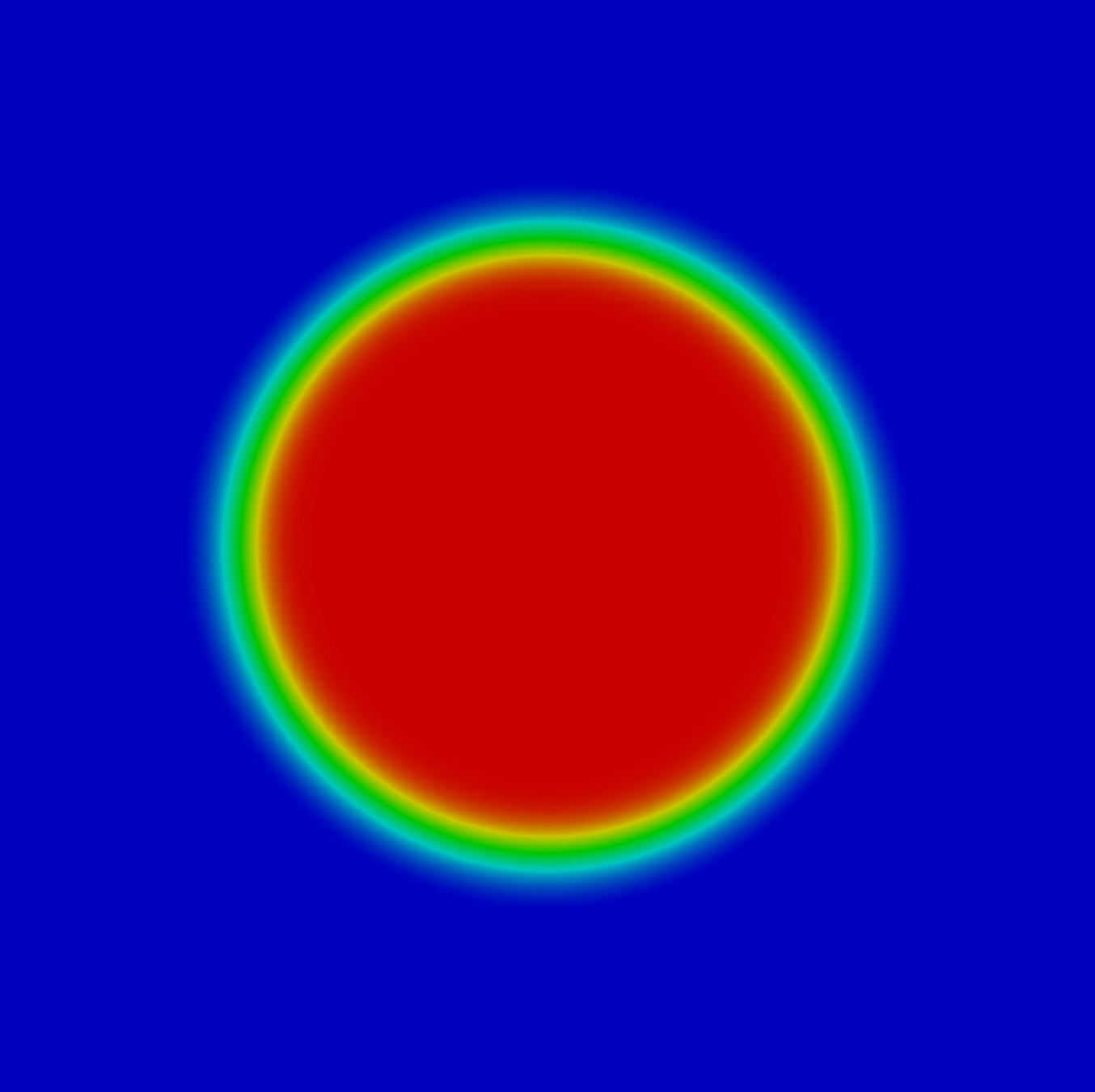}&
			\includegraphics[width=.22\textwidth]{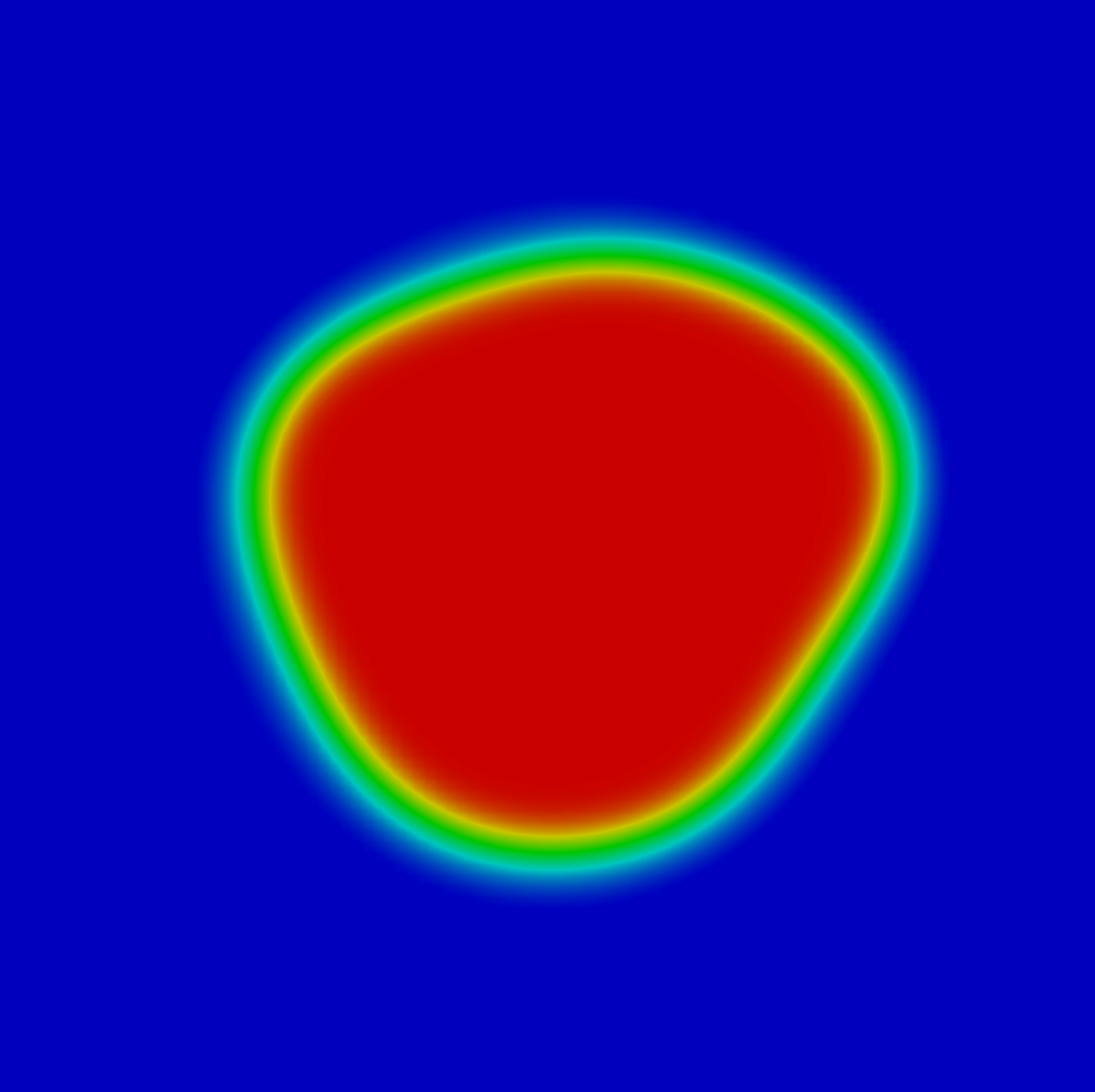}&
			\includegraphics[width=.22\textwidth]{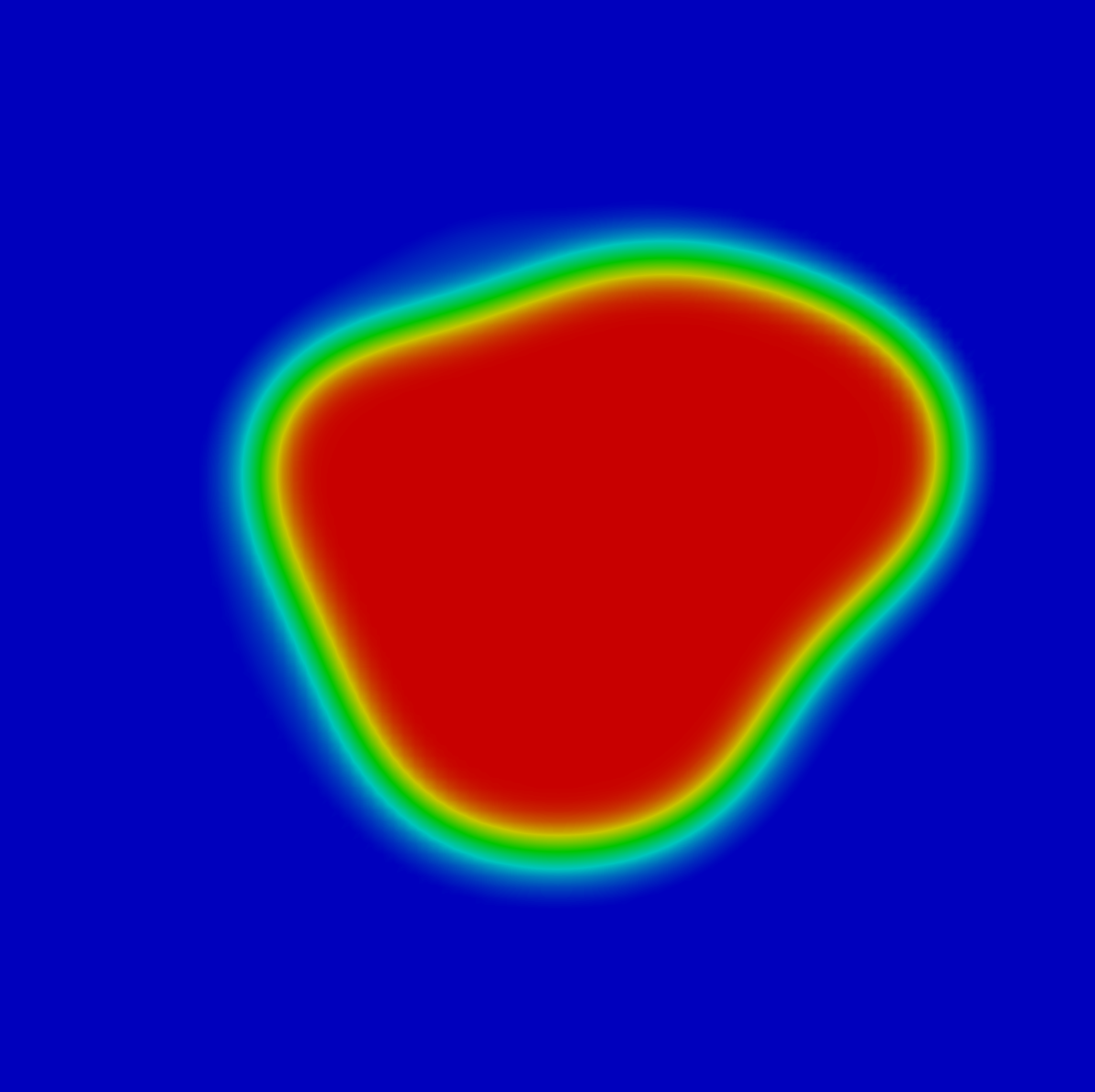}&
			\includegraphics[width=.22\textwidth]{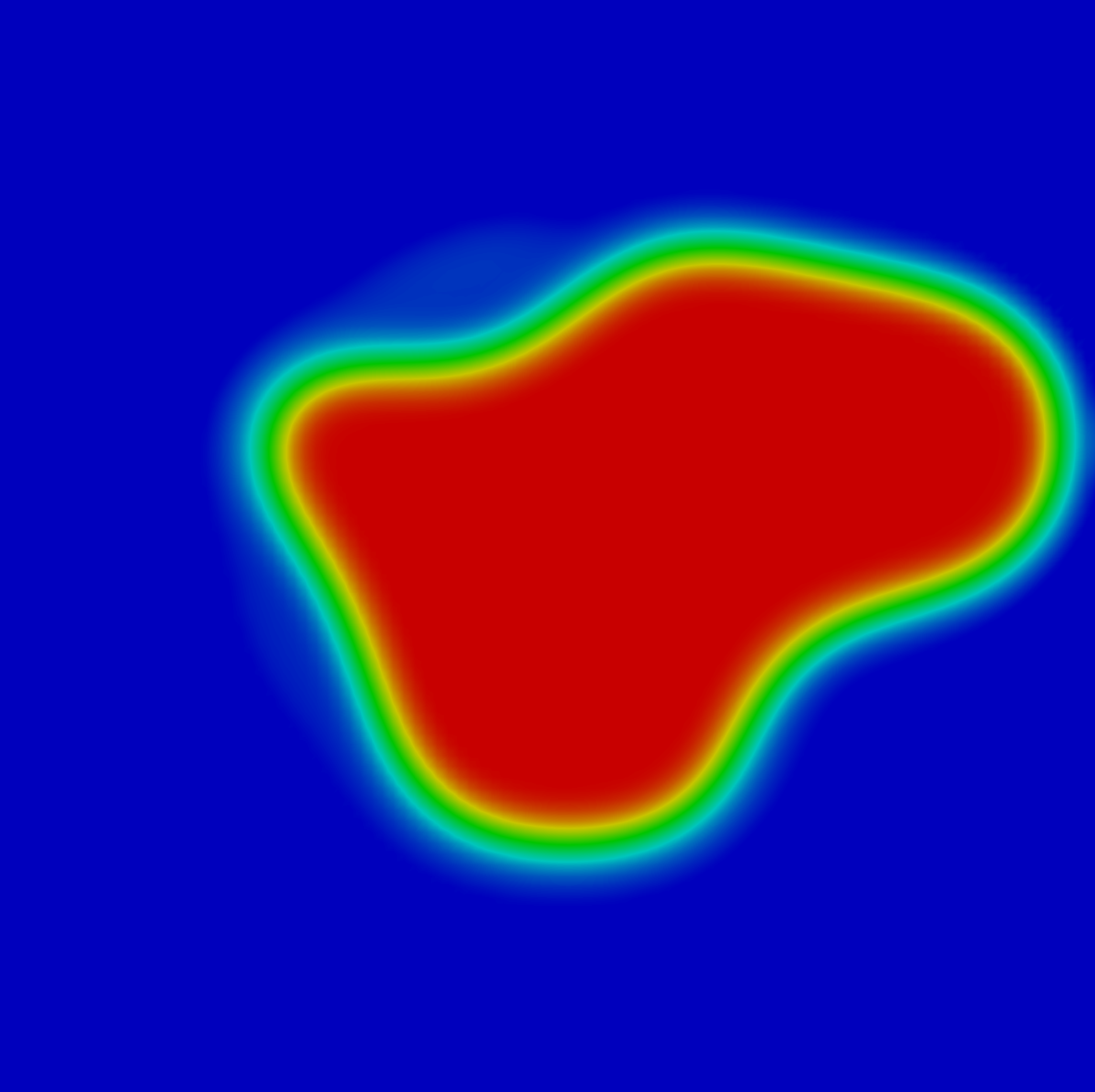}
		\end{tabular} \\
		\hspace{1.5cm}\includegraphics[width=.6\textwidth]{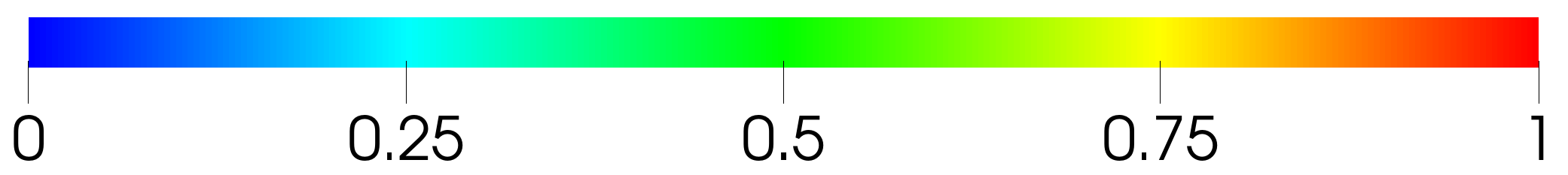}
		\caption{\label{Fig:TumorSameSeed}Evolution of the tumor volume fraction $\phi(t,x)$ over time in the domain $D$ for increasing noise intensities but for a fixed seed.}
	\end{center} 
\end{figure}

In \cref{Fig:TumorContourSame}, we provide a visual representation of the distinct tumor shapes at four different time steps, with a particular focus on the interface of the tumor volume fraction. The contour plots in this figure delineate the contour line where the tumor volume fraction equals 50\%, i.e. $\phi(t,x)=0.5$, effectively marking the boundary between tumor and non-tumor regions. At each of the four highlighted time steps, the contour plot vividly illustrates the evolving shape of the tumor. The most significant observation in \cref{Fig:TumorContourSame} is the variability in tumor shape over time. As the tumor evolves, the contour lines reveal the changes in the tumor's spatial distribution. The variability is particularly pronounced as we transition from the initial state to later time steps, reflecting the influence of noise and stochasticity on the tumor's growth patterns.

\begin{figure}[H] \begin{center}
		\!\!\!\!\!\!\!\!\!\!\!\!\!\! \begin{tabular}{M{.21\textwidth}M{.21\textwidth}M{.21\textwidth}M{.21\textwidth}}
			\quad $t=0.1$ & \quad $t=0.3$& \quad $t=0.5$& \quad $t=0.7$ \\[-.3cm]	\includegraphics[width=.26\textwidth]{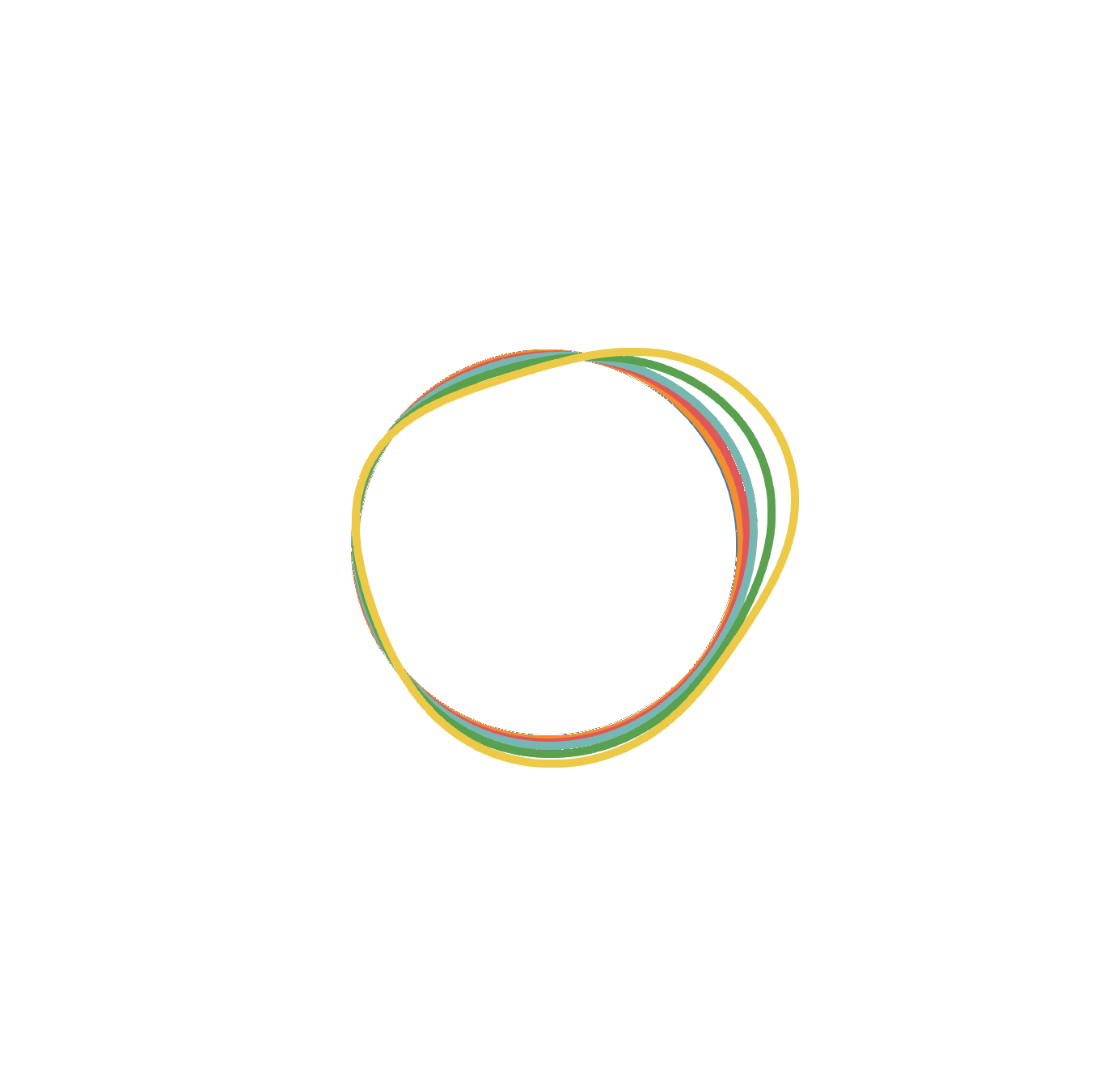}&
			\includegraphics[width=.26\textwidth]{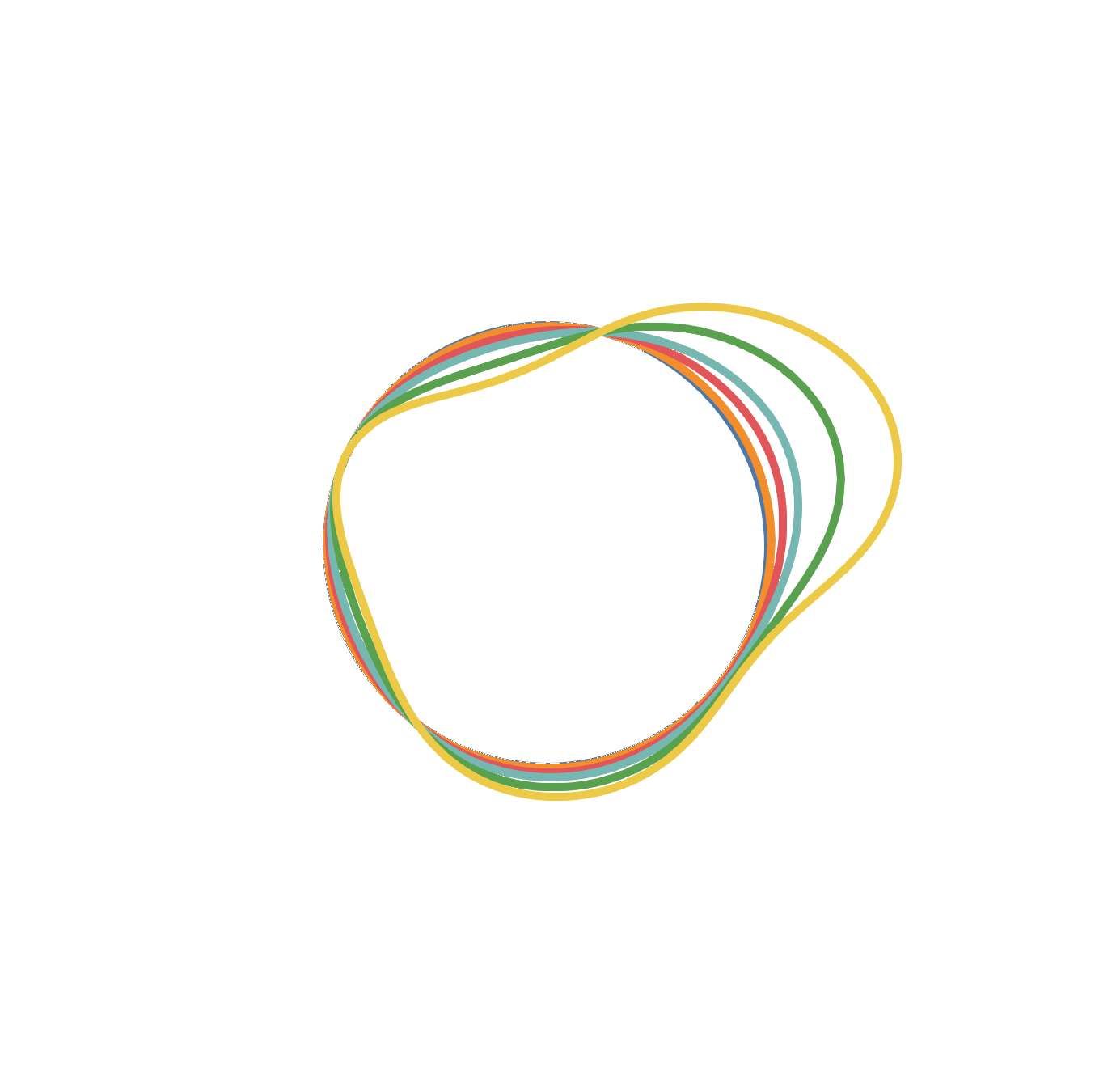}&
			\includegraphics[width=.26\textwidth]{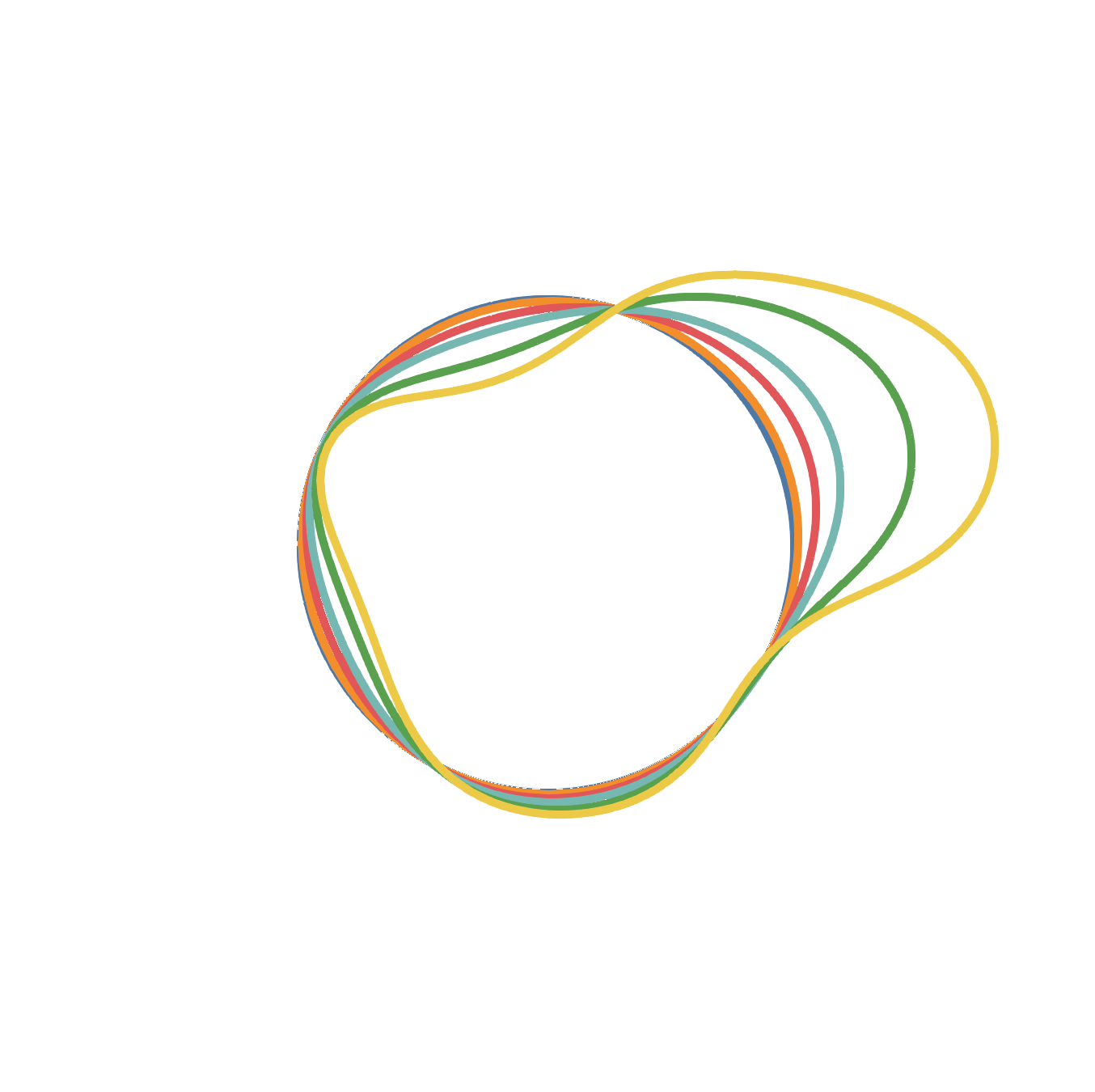}&
			\includegraphics[width=.26\textwidth]{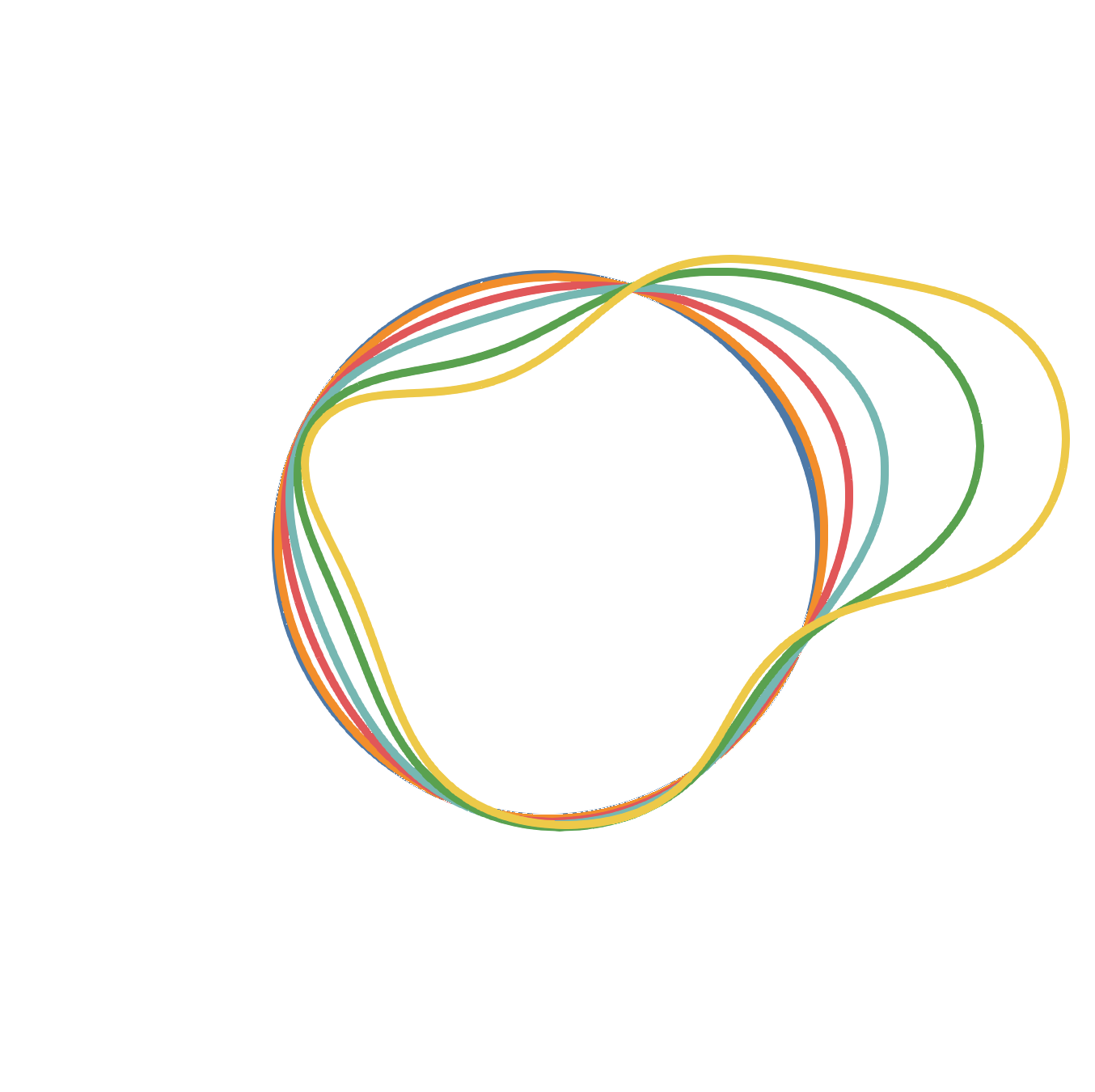} \\[-.5cm]
		\end{tabular} 
		\caption{\label{Fig:TumorContourSame}Contour plots at four distinct time steps visualize the evolving shape of the tumor -- contours represent the boundary where 50\% of the domain contains tumor cells, showcasing the dynamic progression of the tumor's shape for different noise levels (with fixed seed).}
	\end{center} 
\end{figure}

\cref{Fig:Nutrients} provides a series of visual representations showcasing the temporal evolution of nutrient distribution at the four different noise levels $\nu \in \{0,0.5,1,2.5\}$. The nutrient distribution is shown at two time steps, offering insights into the spatial changes in nutrient availability within the domain. We observe that the nutrient concentration is highest at the boundary of the domain. This concentration gradient is a result of the chosen boundary condition, which provides a continuous source of nutrients at the domain's boundary. As time progresses, we witness the diffusion of nutrients from the boundary toward the center of the domain. Nutrient levels decrease as we move closer to the domain's center, reflecting the expected behavior of diffusion. Notably, the presence of the tumor has a significant impact on the nutrient distribution. As the tumor grows and absorbs nutrients, its interface becomes evident within the nutrient distribution. The nutrient levels are notably reduced in the vicinity of the tumor, indicating the tumor's influence on nutrient availability.

\begin{figure}[H] \begin{center}
		\begin{tabular}{cM{.19\textwidth}M{.19\textwidth}M{.19\textwidth}M{.19\textwidth}}
			&
			\quad $\nu=0.0$&\quad $\nu=0.5$&\quad $\nu=1.0$&\quad $\nu=2.5$ \\
			\!\!\!\!\!\!$t=0.4$\!\!\!\!\!&
			\includegraphics[width=.22\textwidth]{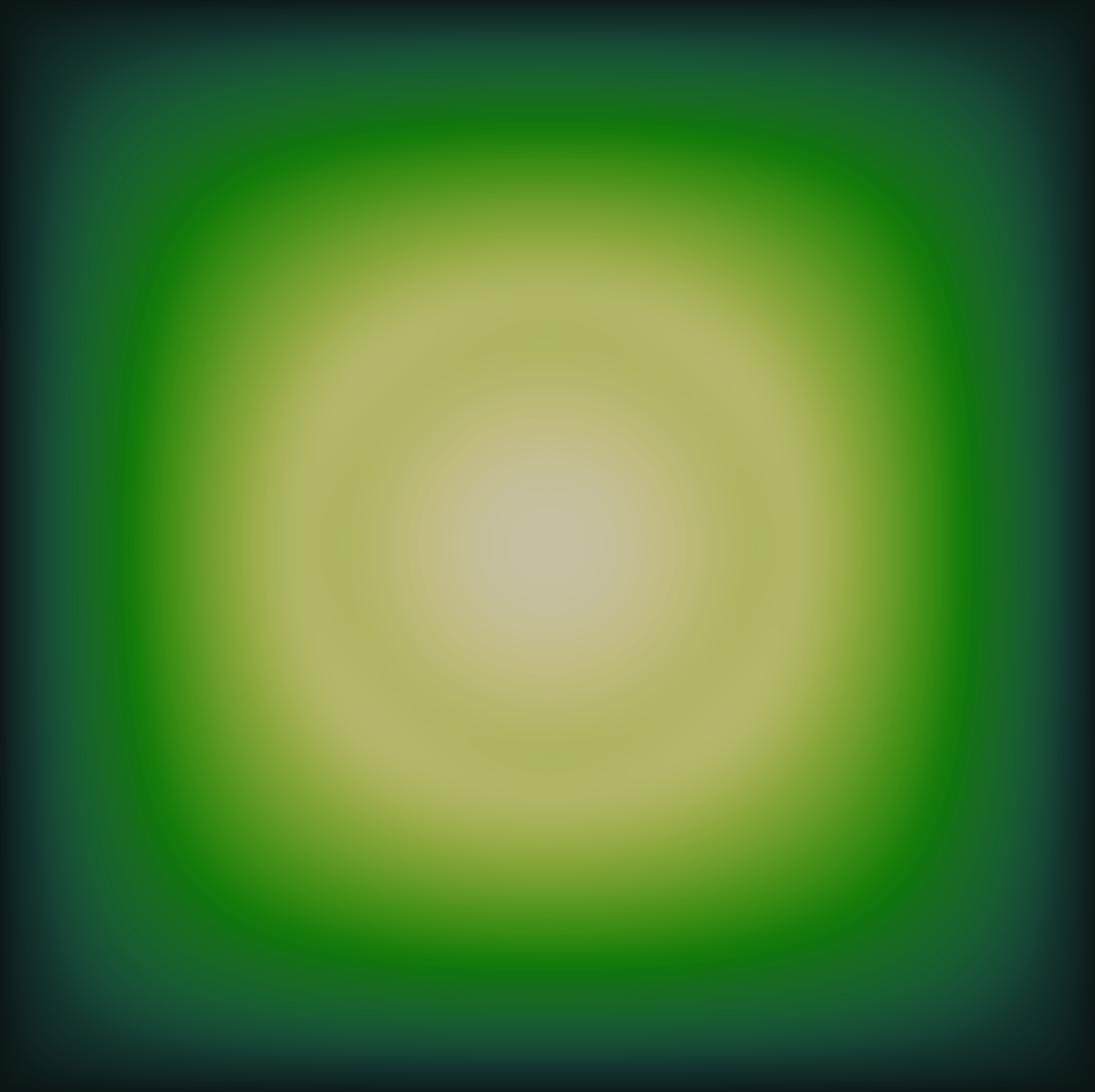}&
			\includegraphics[width=.22\textwidth]{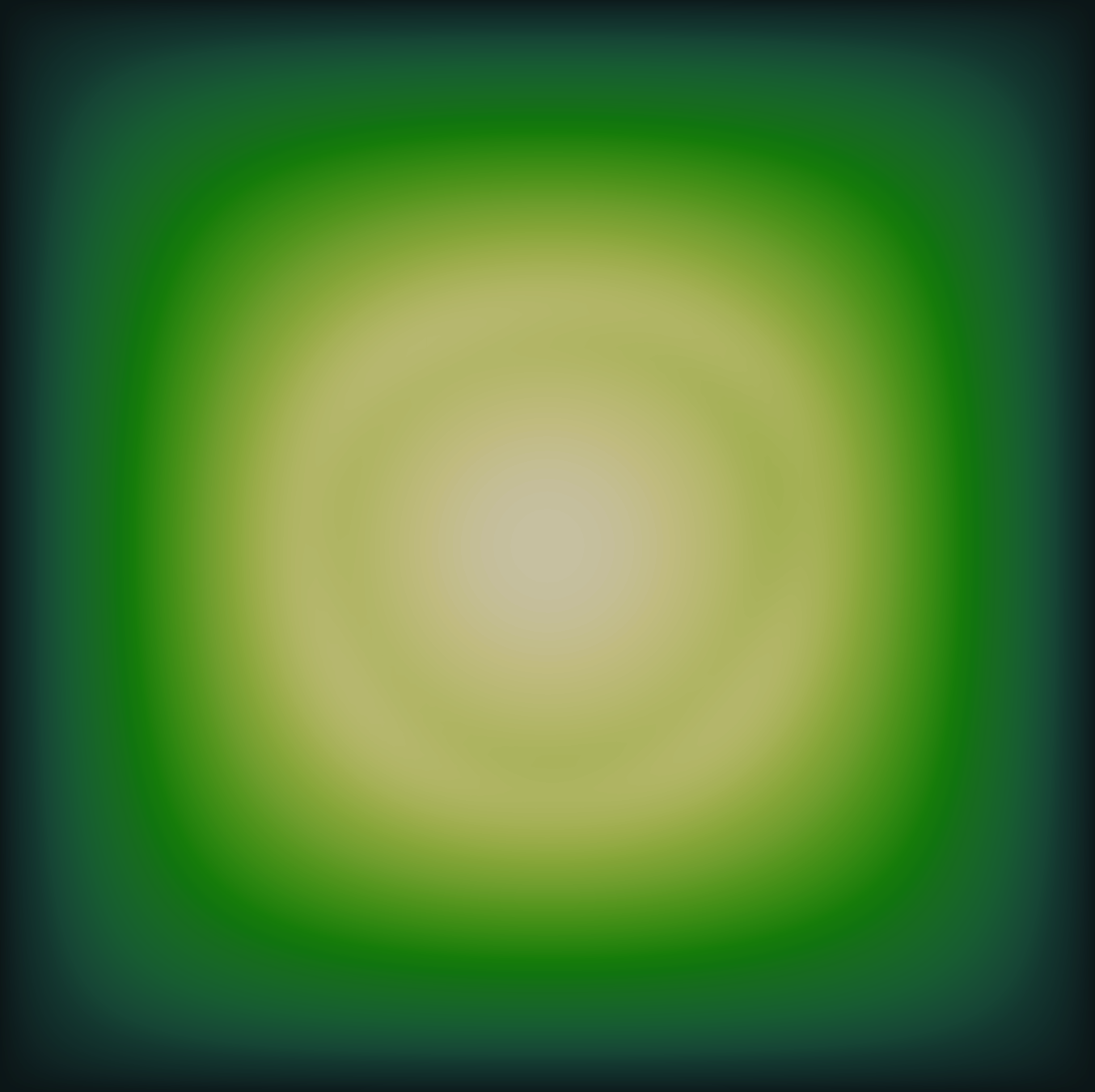}&
			\includegraphics[width=.22\textwidth]{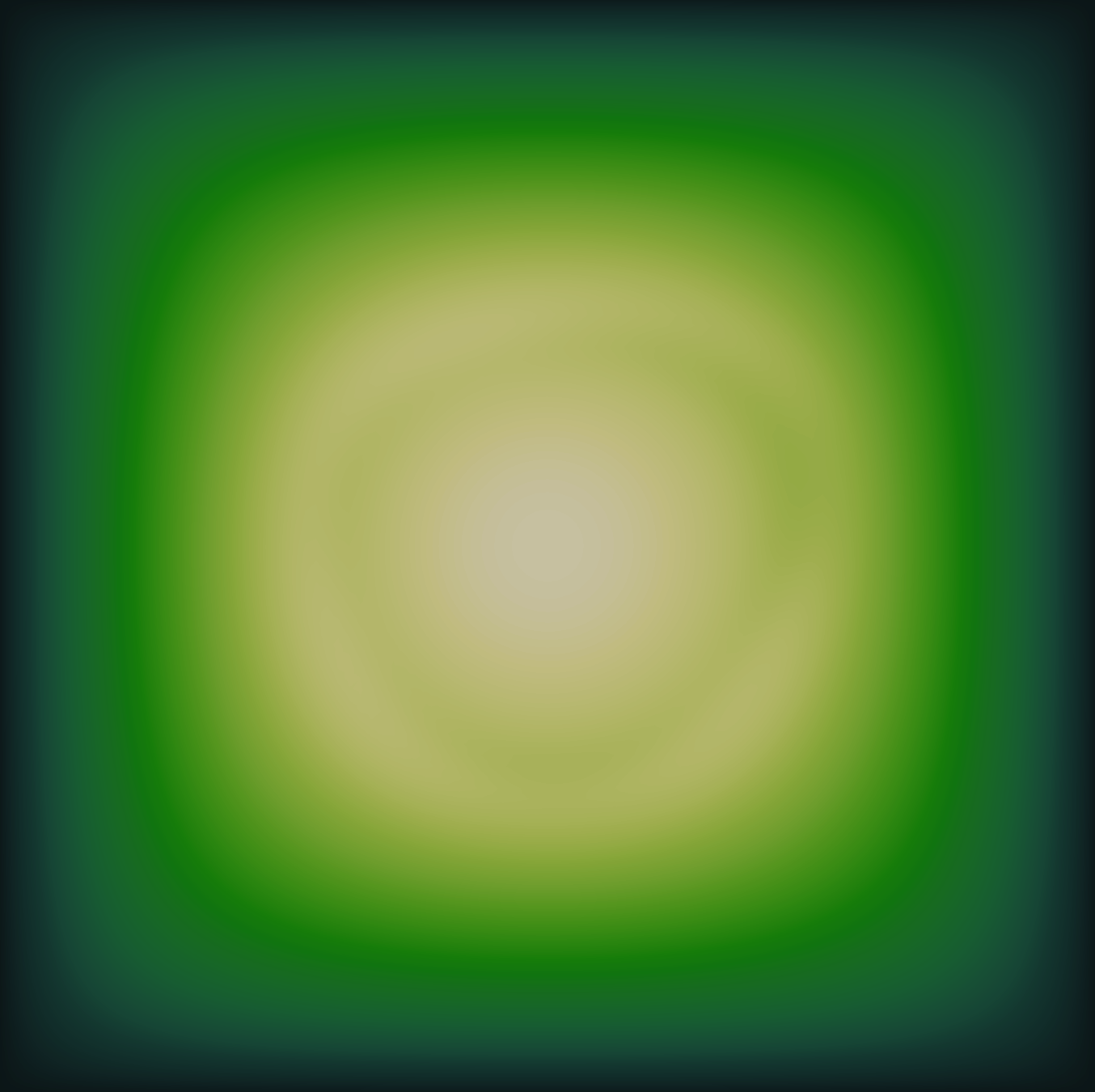}&
			\includegraphics[width=.22\textwidth]{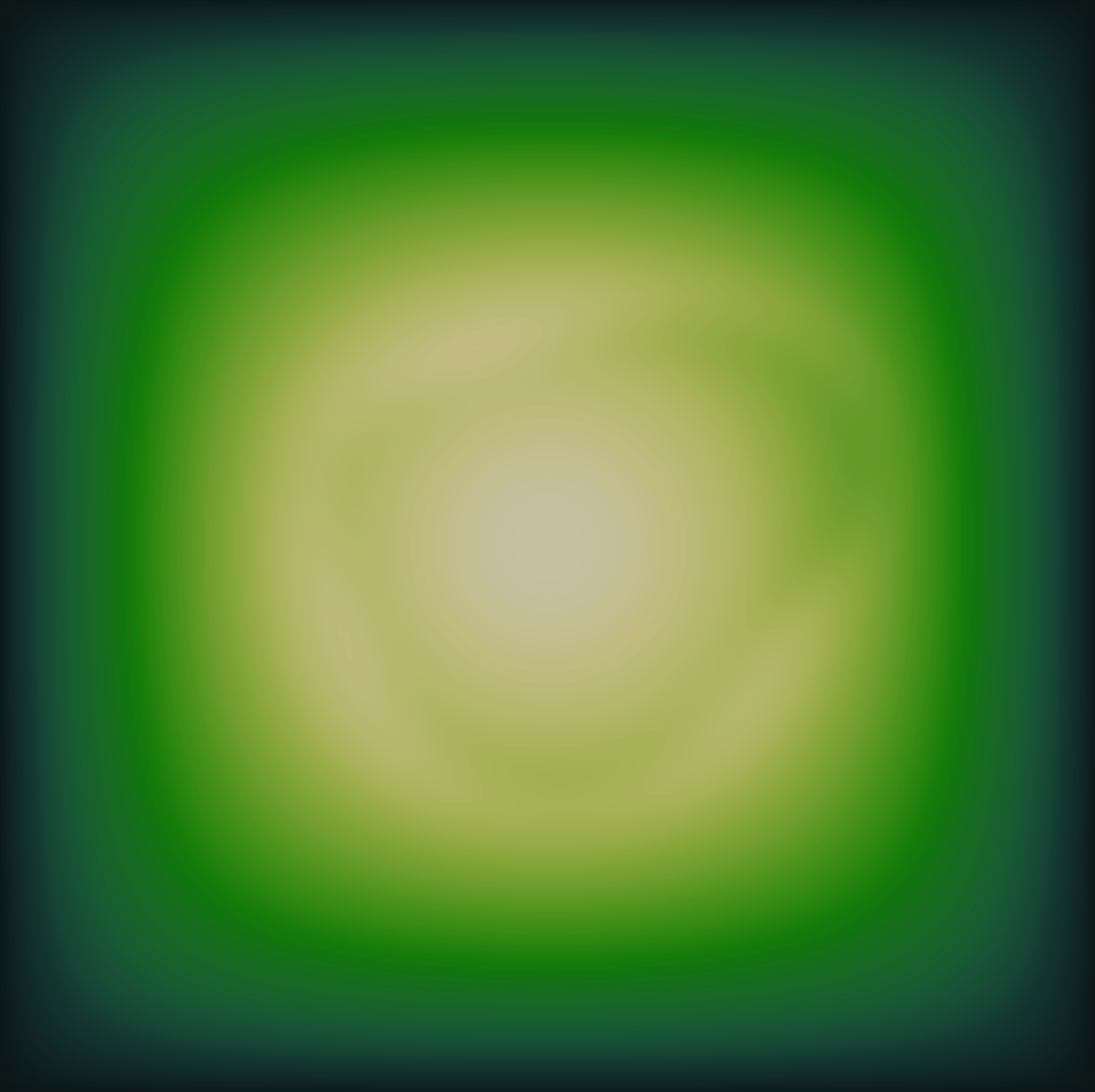}\\[-.1cm]
			\!\!\!\!\!\!$t=1.0$\!\!\!\!\!&
			\includegraphics[width=.22\textwidth]{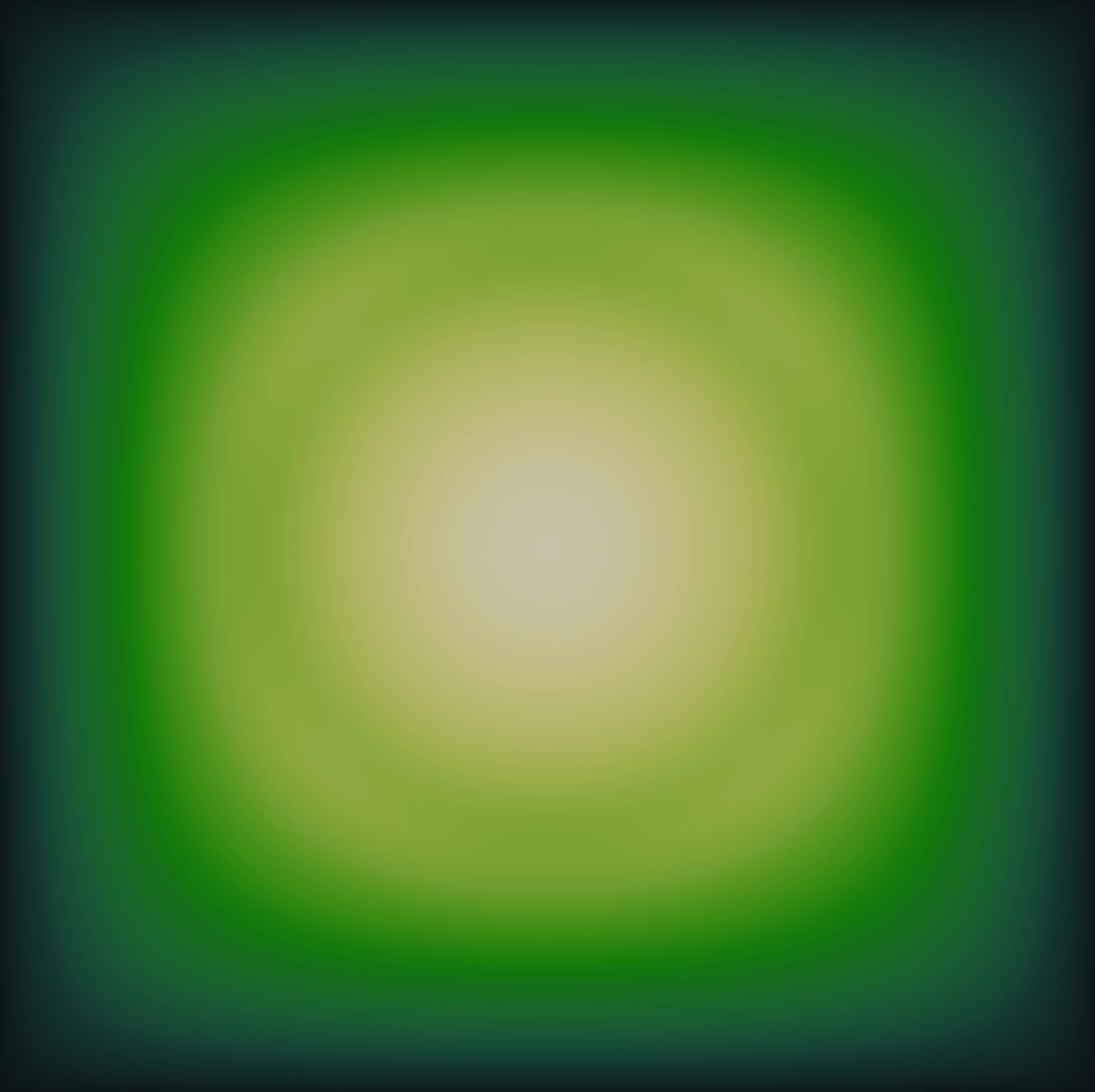}&
			\includegraphics[width=.22\textwidth]{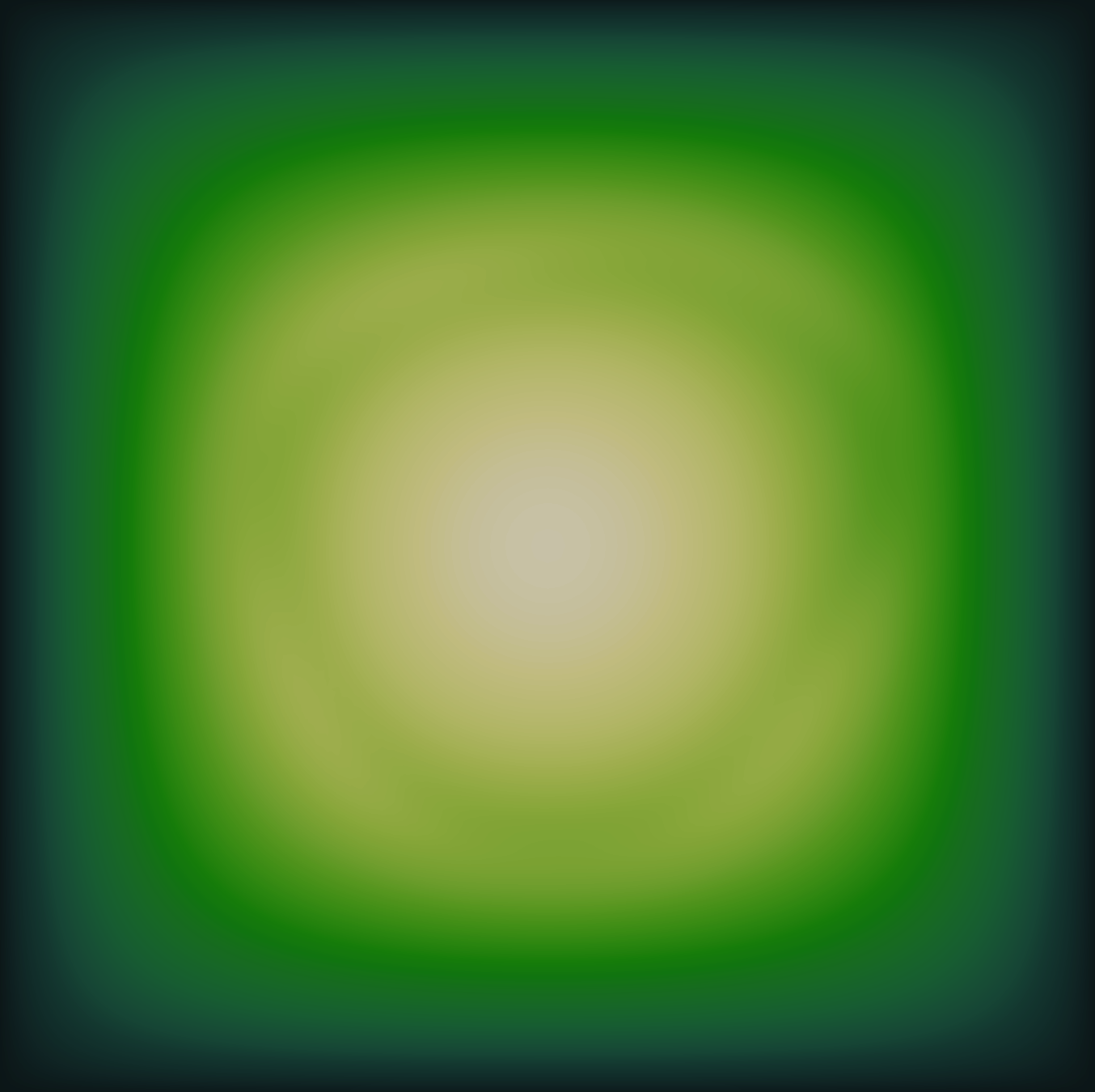}&
			\includegraphics[width=.22\textwidth]{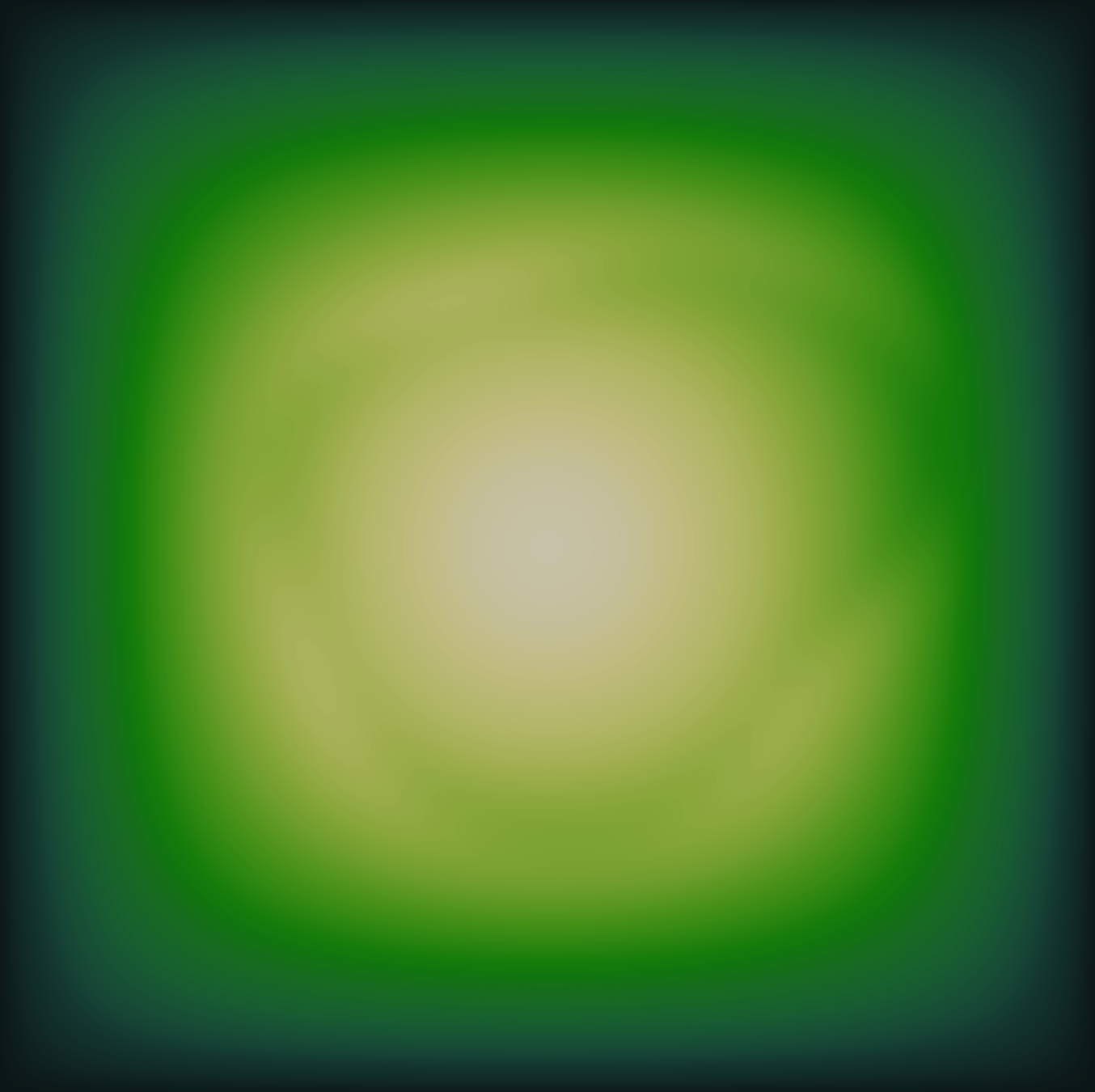}&
			\includegraphics[width=.22\textwidth]{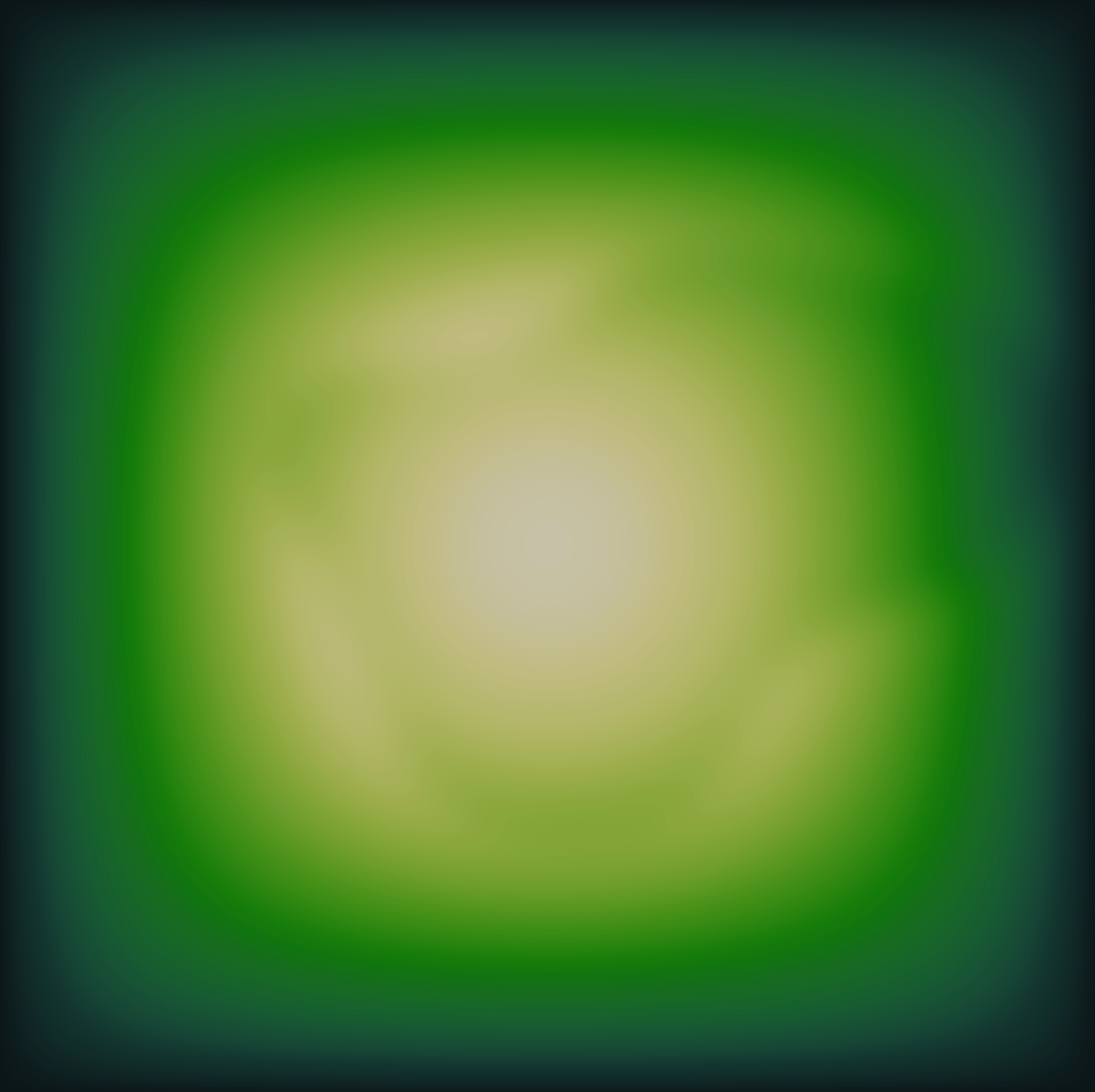}
		\end{tabular} \\
		\hspace{1.5cm}\includegraphics[width=.6\textwidth]{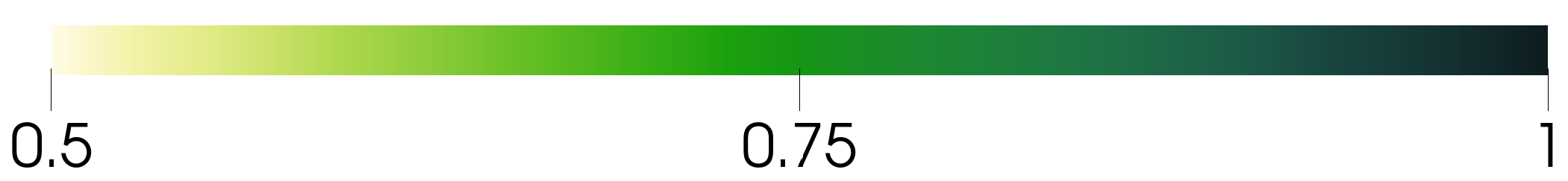}
		\caption{\label{Fig:Nutrients}Temporal evolution of nutrient distribution $\sigma(t,x)$  at the noise levels $\nu \in \{0,0.5,1,2.5\}$ (using a fixed seed).}
	\end{center} 
\end{figure}

\noindent{\textbf{Test 3 (Tumor evolution for different seeds).}}
In \cref{Fig:TumorDifferentSeedHighNoise}, we explore the tumor's evolution under conditions of high noise levels while considering different random seeds. Each simulation is truly random, and the figure showcases the results of four distinct samples at various time steps. At first glance, the figure highlights the significant variability in tumor shape. The choice of different random seeds leads to entirely distinct tumor shapes, even at earlier time steps. This highlights the inherent randomness and unpredictability introduced by high noise levels, shaping the tumor's growth patterns in unique ways for each simulation. Notably, the diversity in tumor shapes is evident even at earlier times, suggesting that the impact of randomness becomes apparent from the onset of the simulations. This underscores the role of high noise levels in influencing tumor behavior and shape diversity.

\begin{figure}[H] \begin{center}
		\begin{tabular}{cM{.19\textwidth}M{.19\textwidth}M{.19\textwidth}M{.19\textwidth}}
			& \quad seed 1 & \quad seed 2& \quad seed 3& \quad seed 4 \\
			\!\!\!\!\!\!$t=0.4$\!\!\!\!\!&
			\includegraphics[width=.22\textwidth]{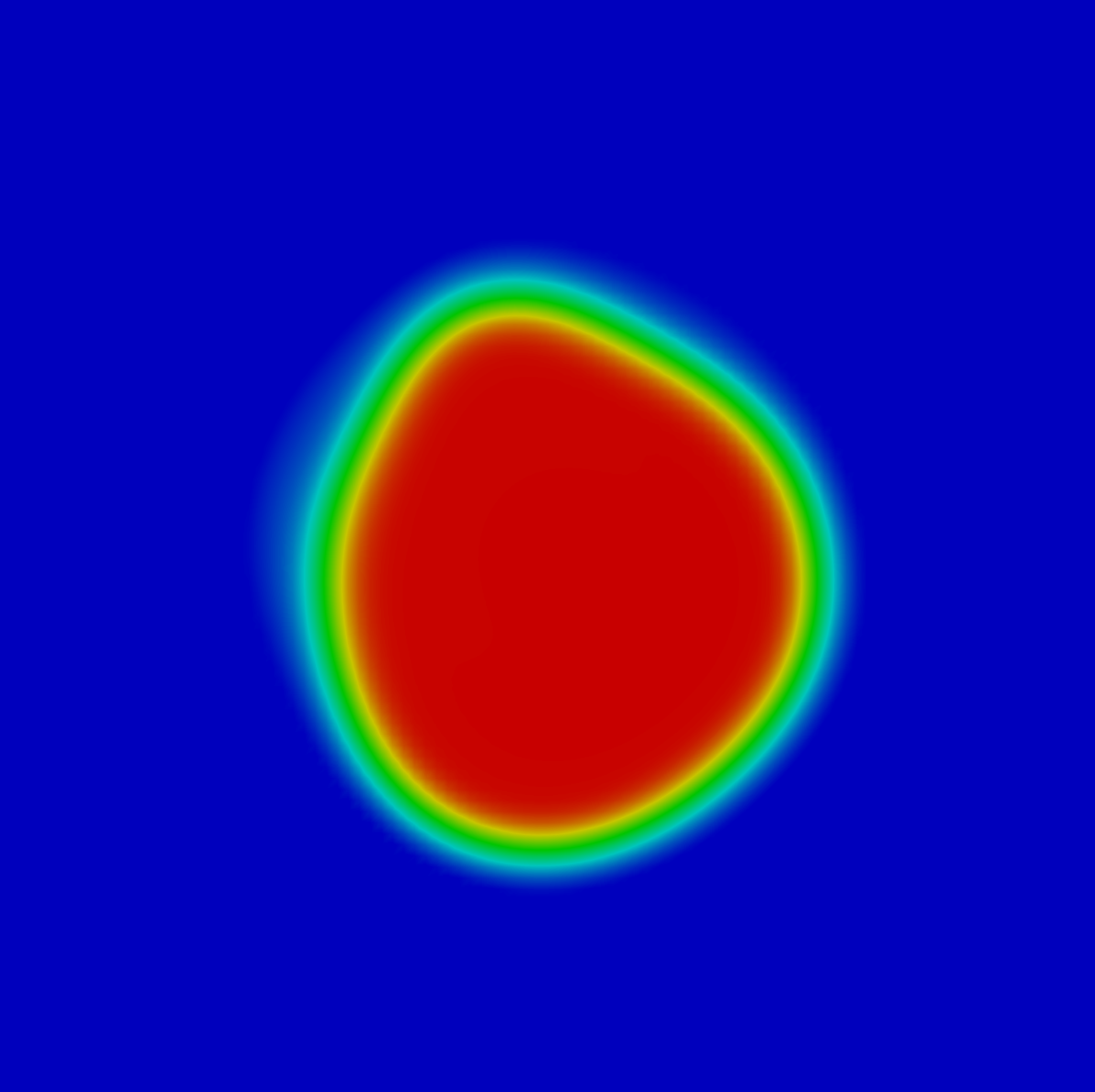}&
			\includegraphics[width=.22\textwidth]{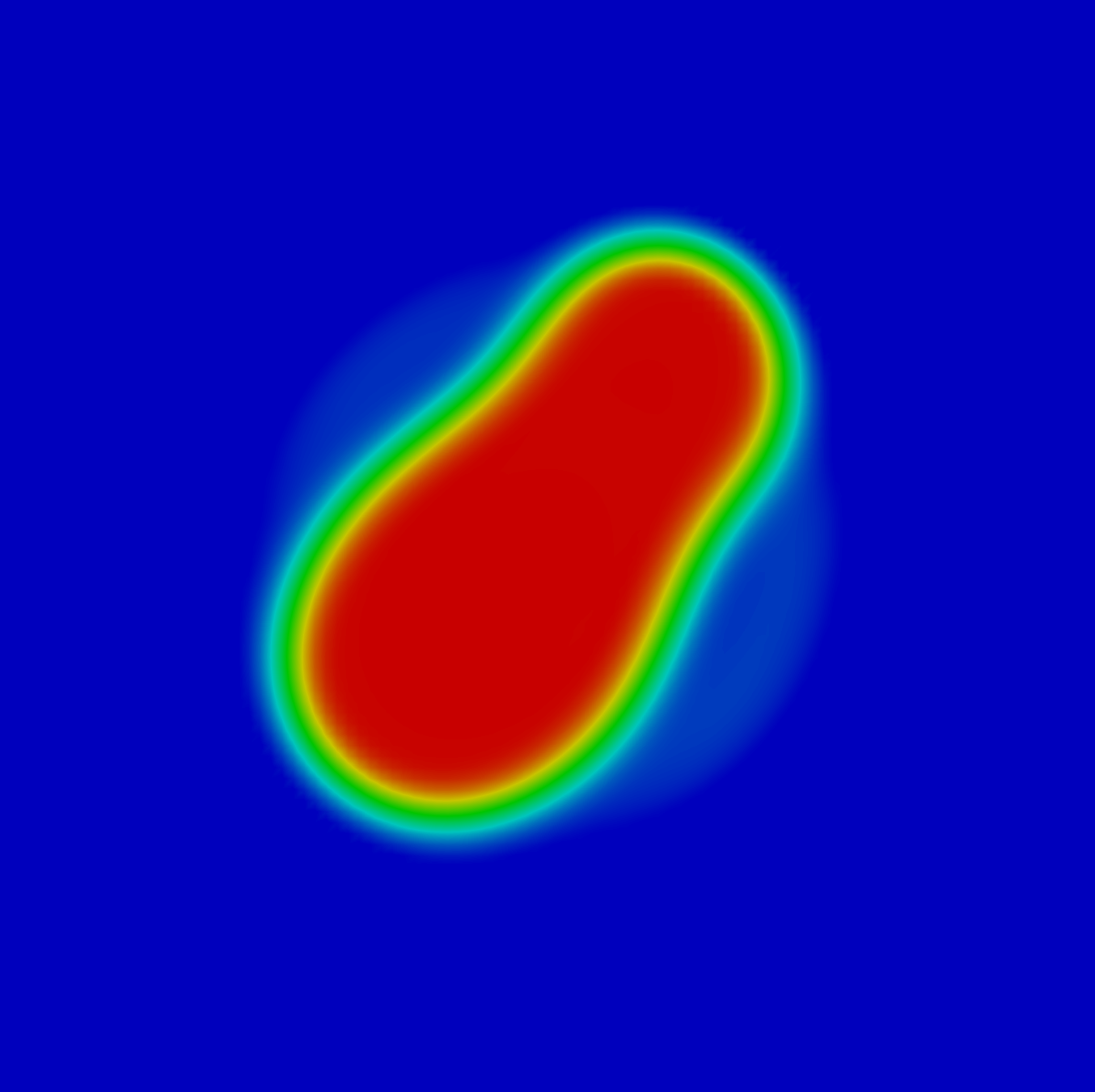}&
			\includegraphics[width=.22\textwidth]{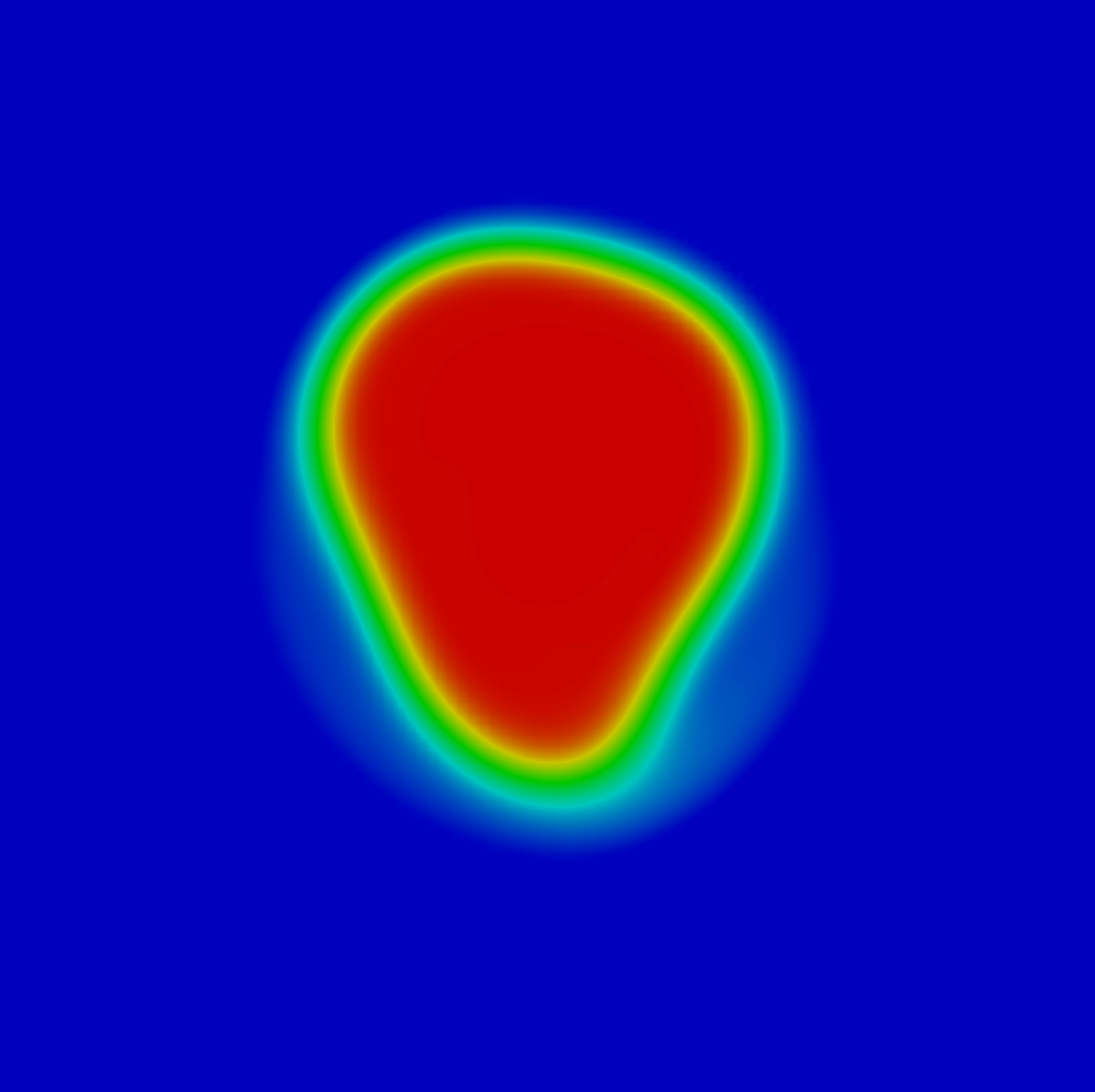}&
			\includegraphics[width=.22\textwidth]{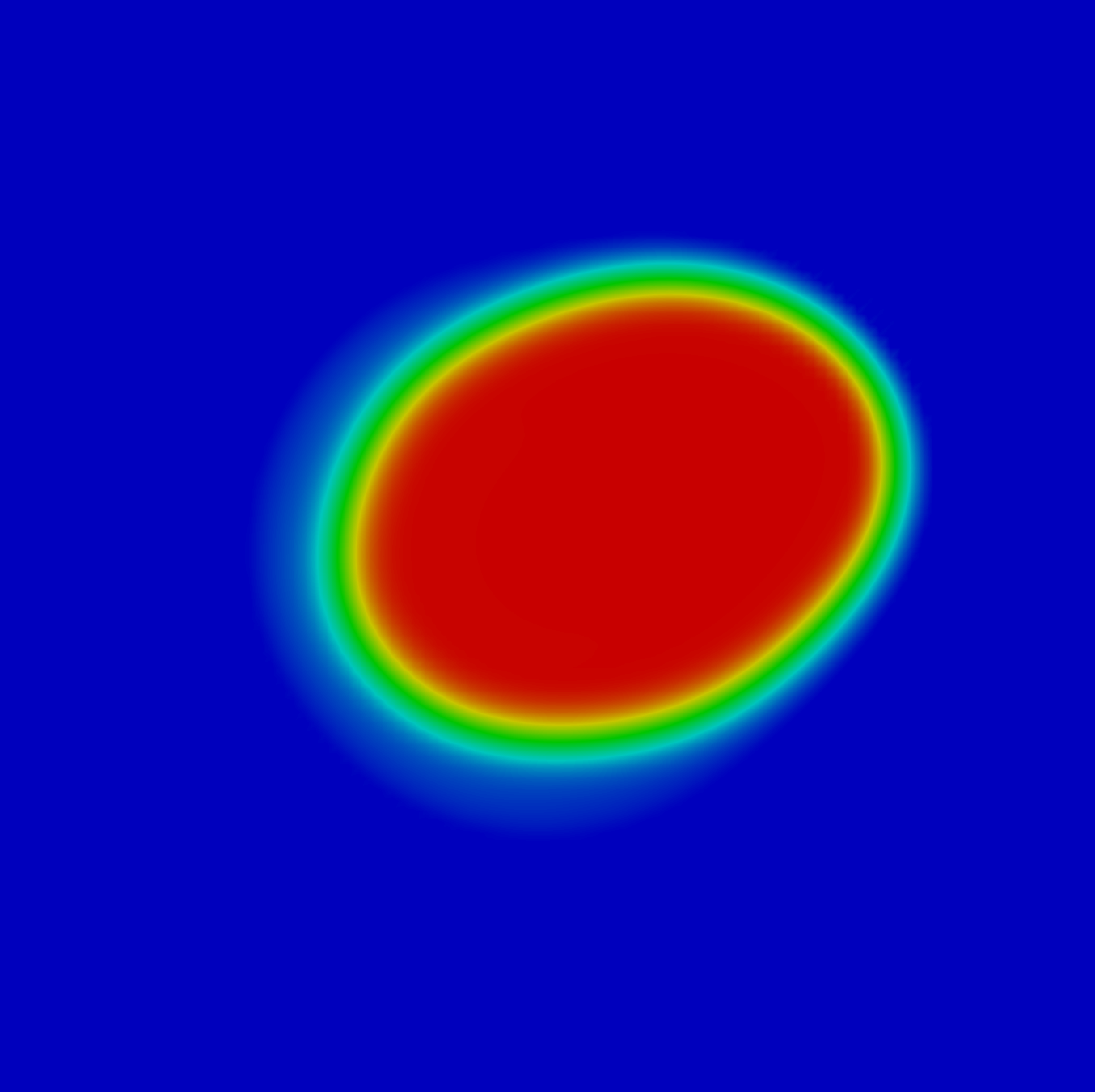}\\[-.1cm]
			\!\!\!\!\!\!$t=1.0$\!\!\!\!\!&
			\includegraphics[width=.22\textwidth]{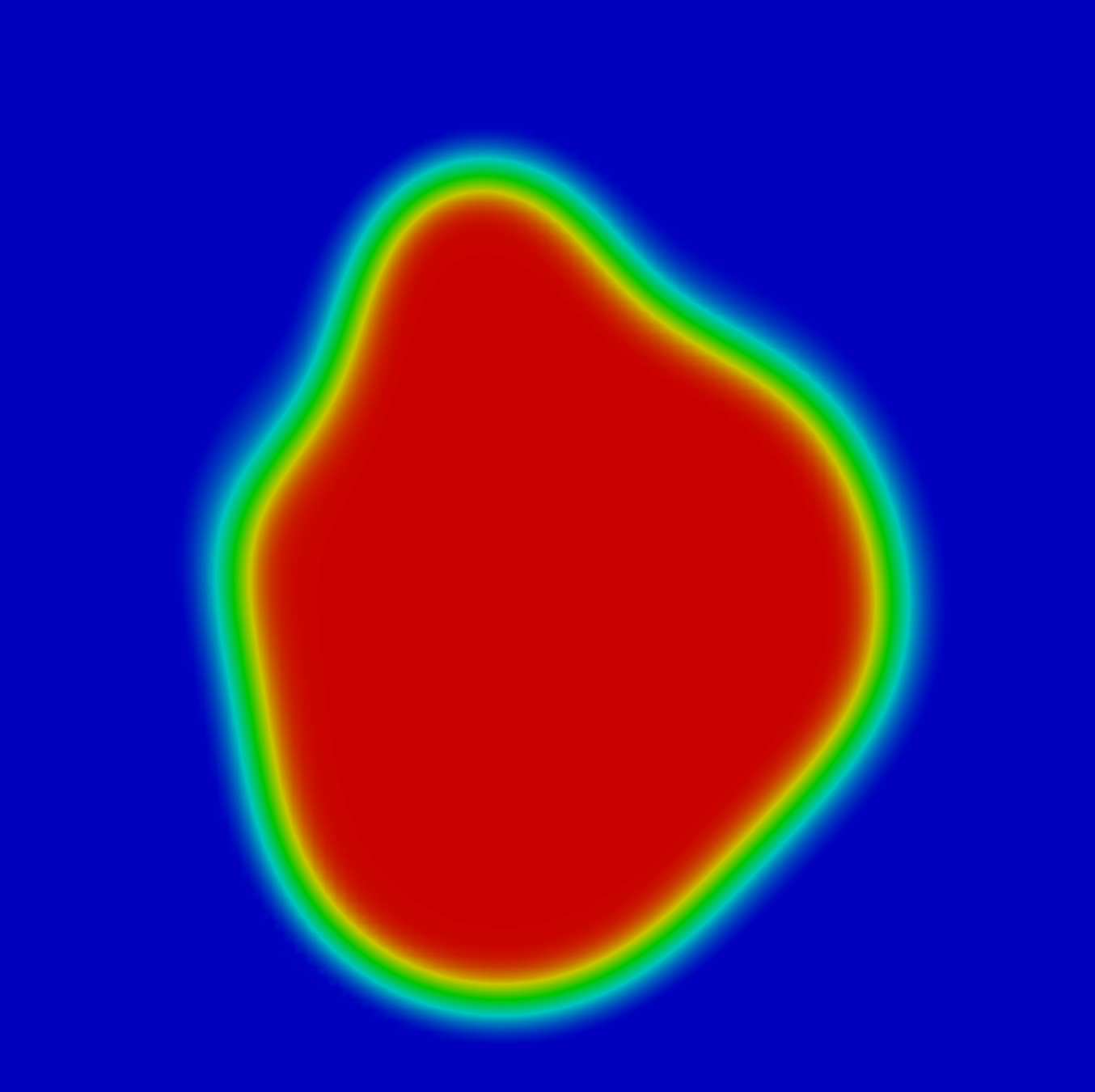}&
			\includegraphics[width=.22\textwidth]{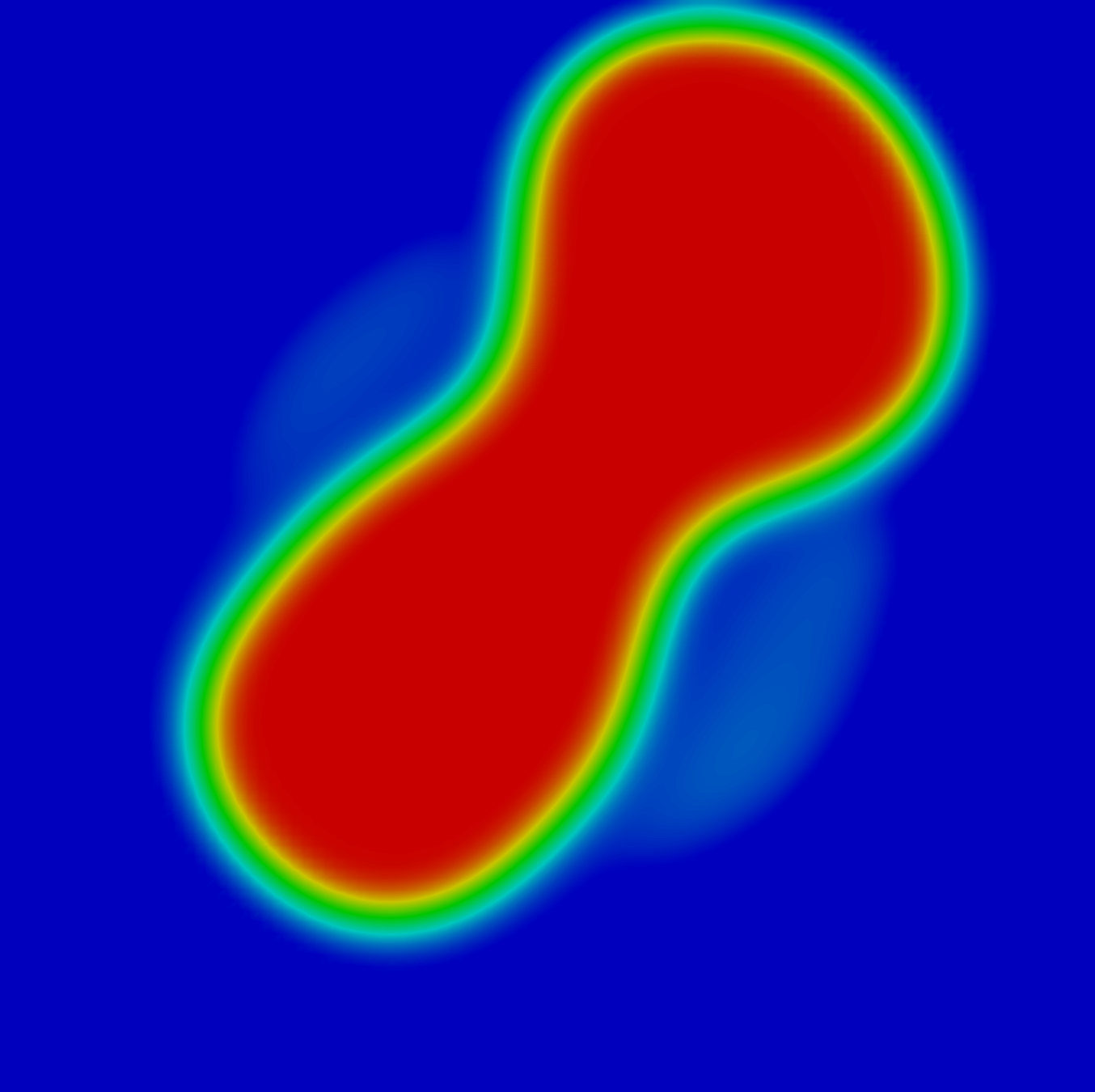}&
			\includegraphics[width=.22\textwidth]{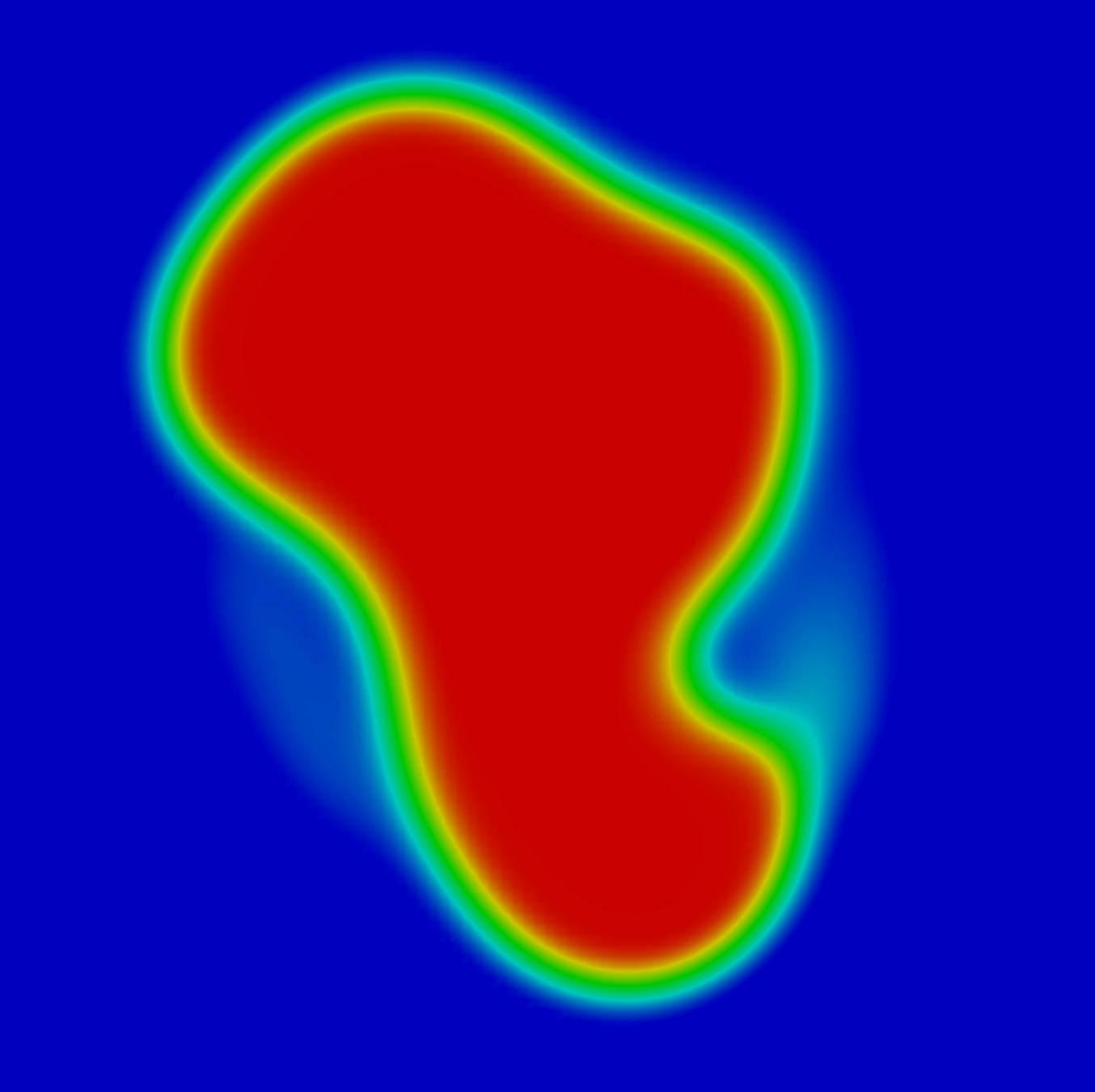}&
			\includegraphics[width=.22\textwidth]{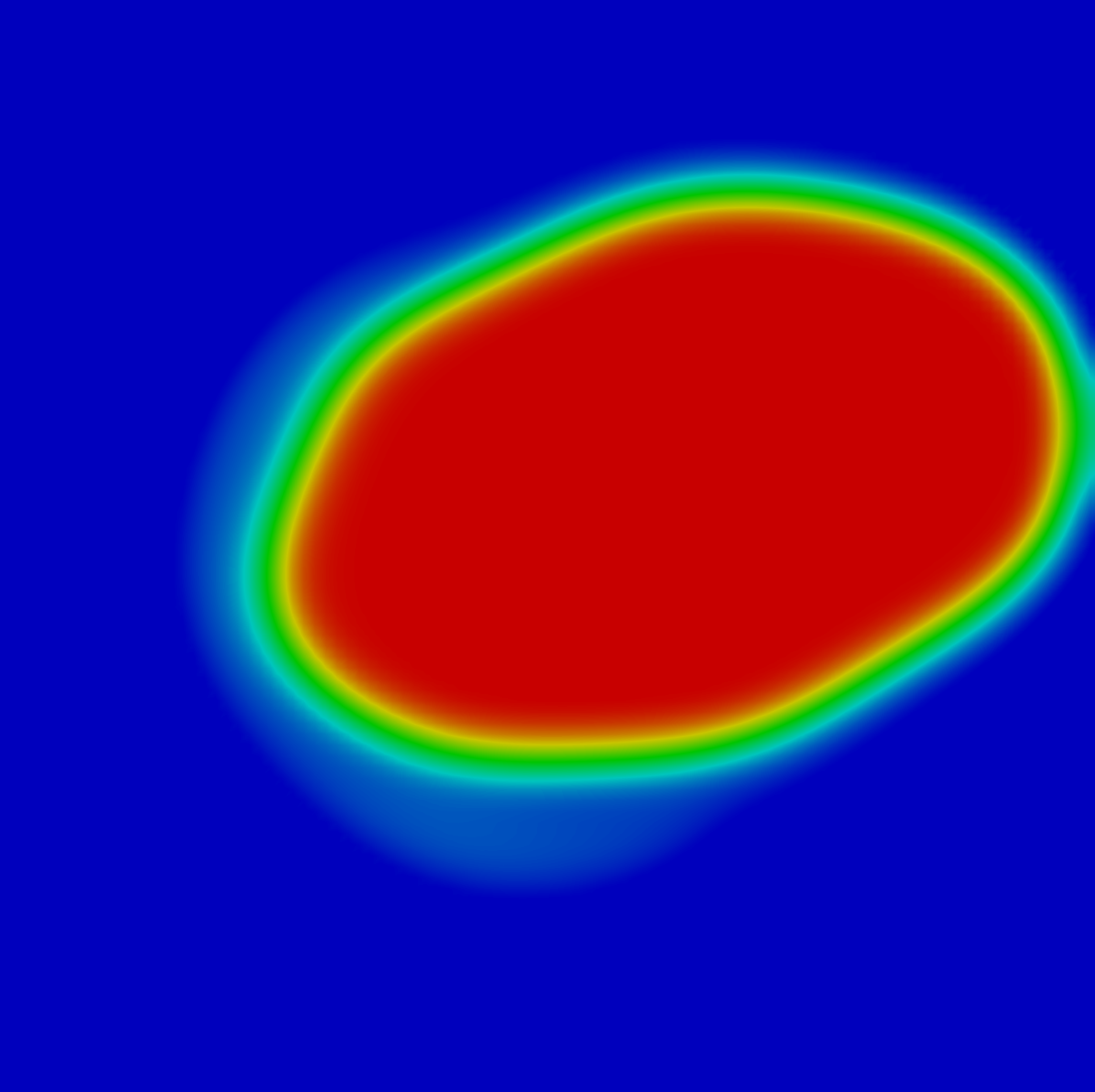}
		\end{tabular} \\
		\hspace{1.5cm}\includegraphics[width=.6\textwidth]{Figures/color1.png}
		\caption{\label{Fig:TumorDifferentSeedHighNoise}Evolution of the tumor volume fraction $\phi(t,x)$ over time in the domain $D$ at a fixed noise intensity $\nu=2.5$ for four random seeds.}
	\end{center} 
\end{figure}

\begin{figure}[H] \begin{center}
		\begin{tabular}{cM{.19\textwidth}M{.19\textwidth}M{.19\textwidth}M{.19\textwidth}}
			&
			\quad seed 1 & \quad seed 2& \quad seed 3& \quad seed 4 \\
			\!\!\!\!\!\!$t=0.4$\!\!\!\!\!&
			\includegraphics[width=.22\textwidth]{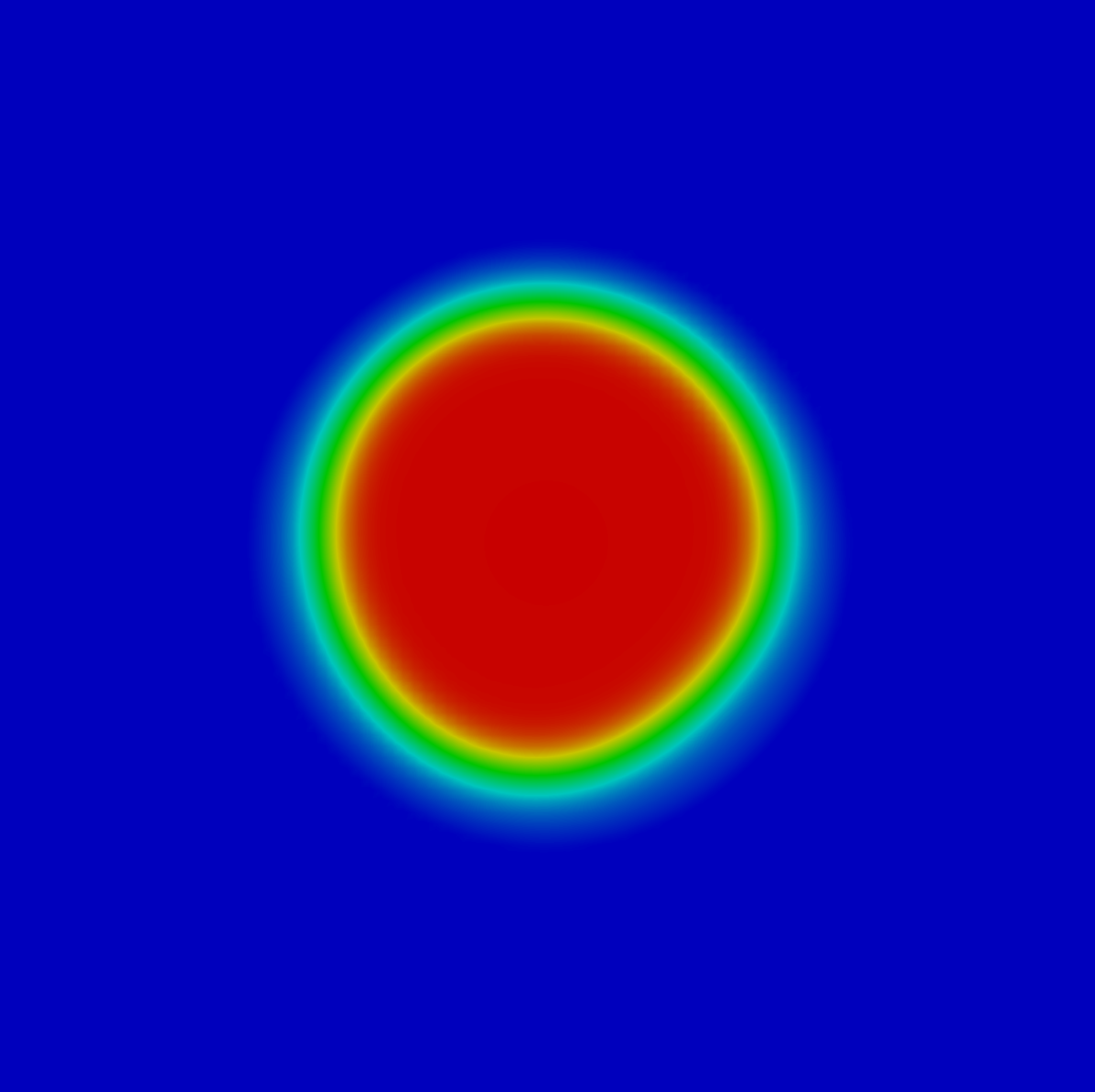}&
			\includegraphics[width=.22\textwidth]{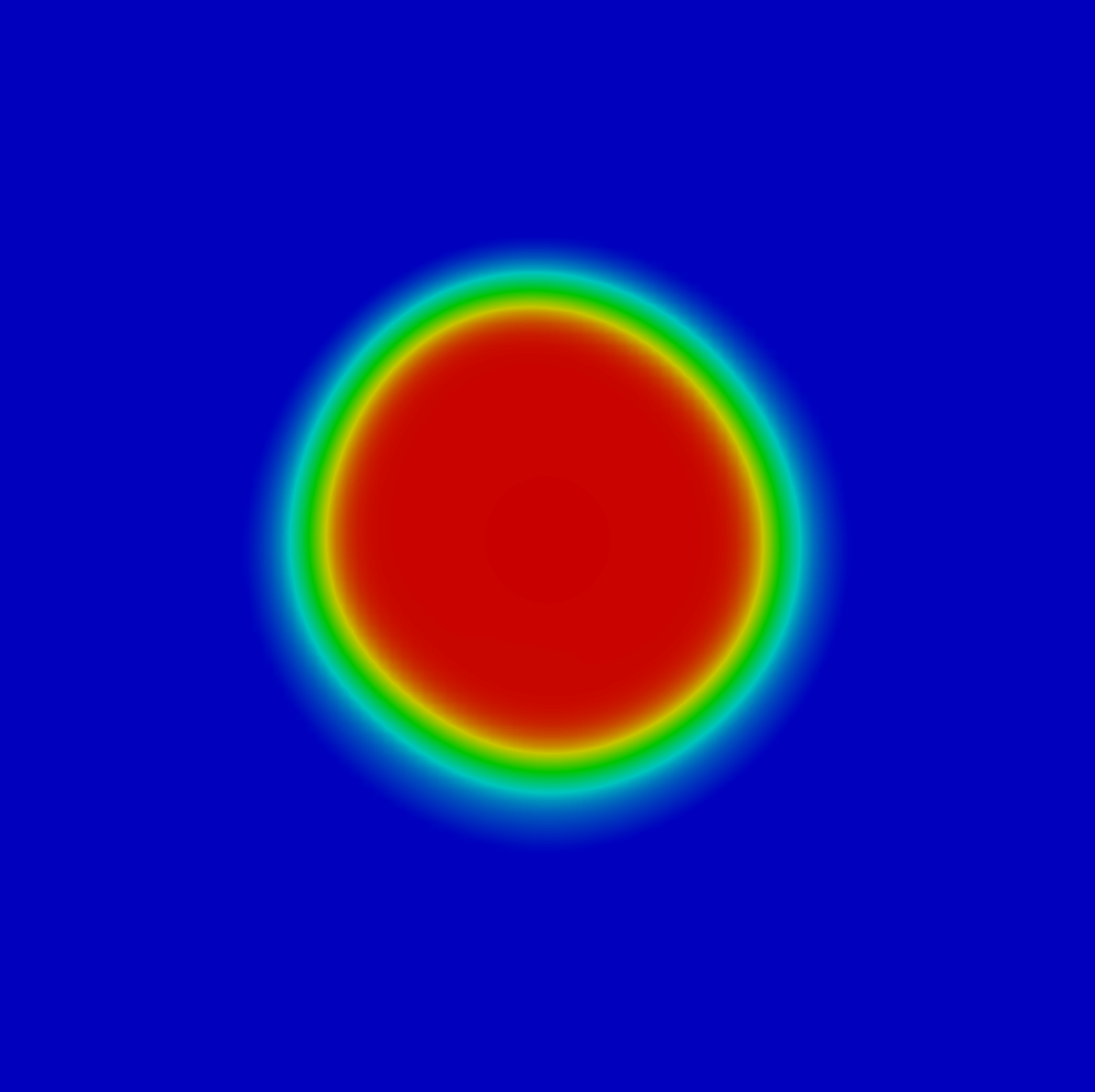}&
			\includegraphics[width=.22\textwidth]{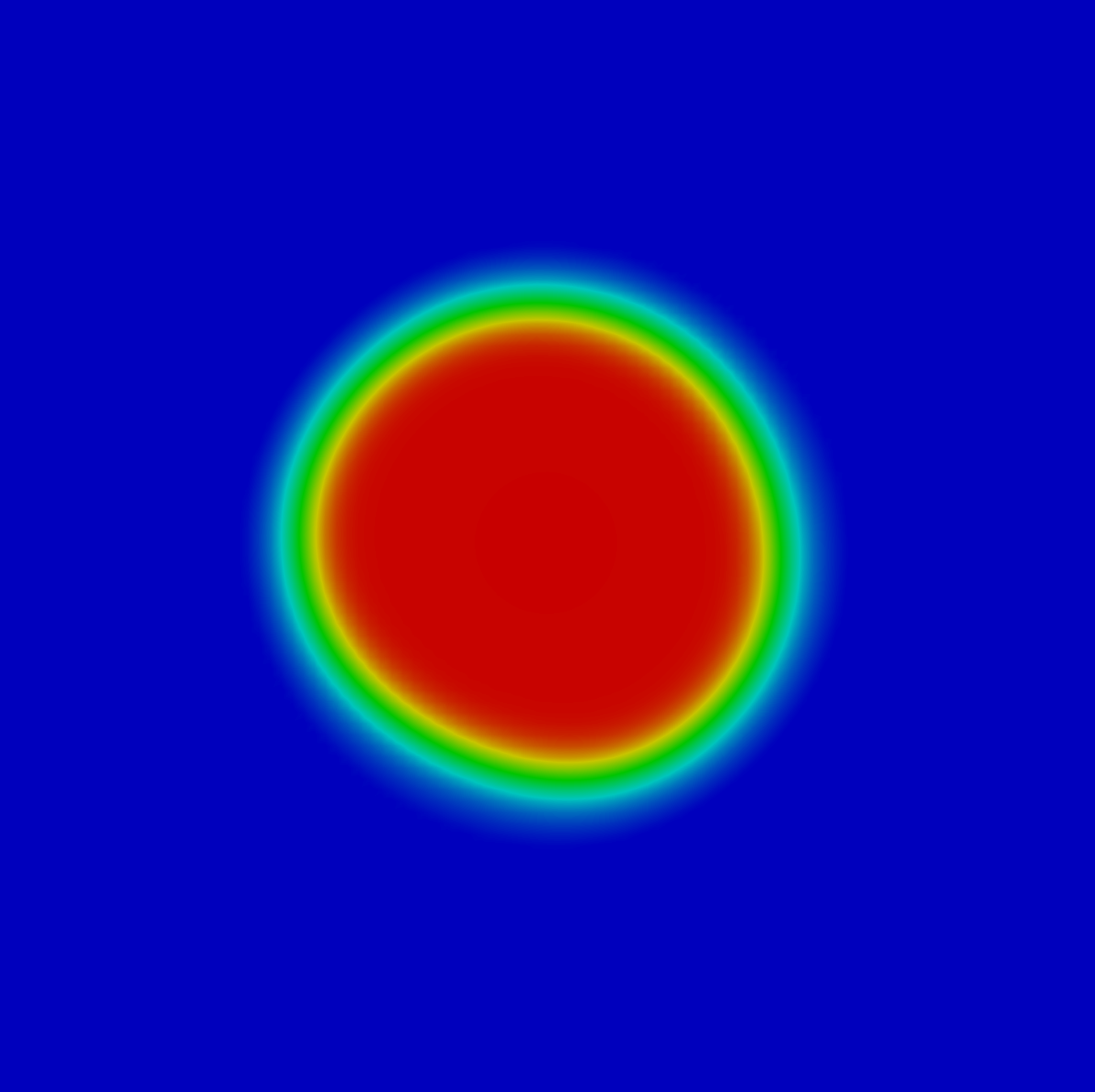}&
			\includegraphics[width=.22\textwidth]{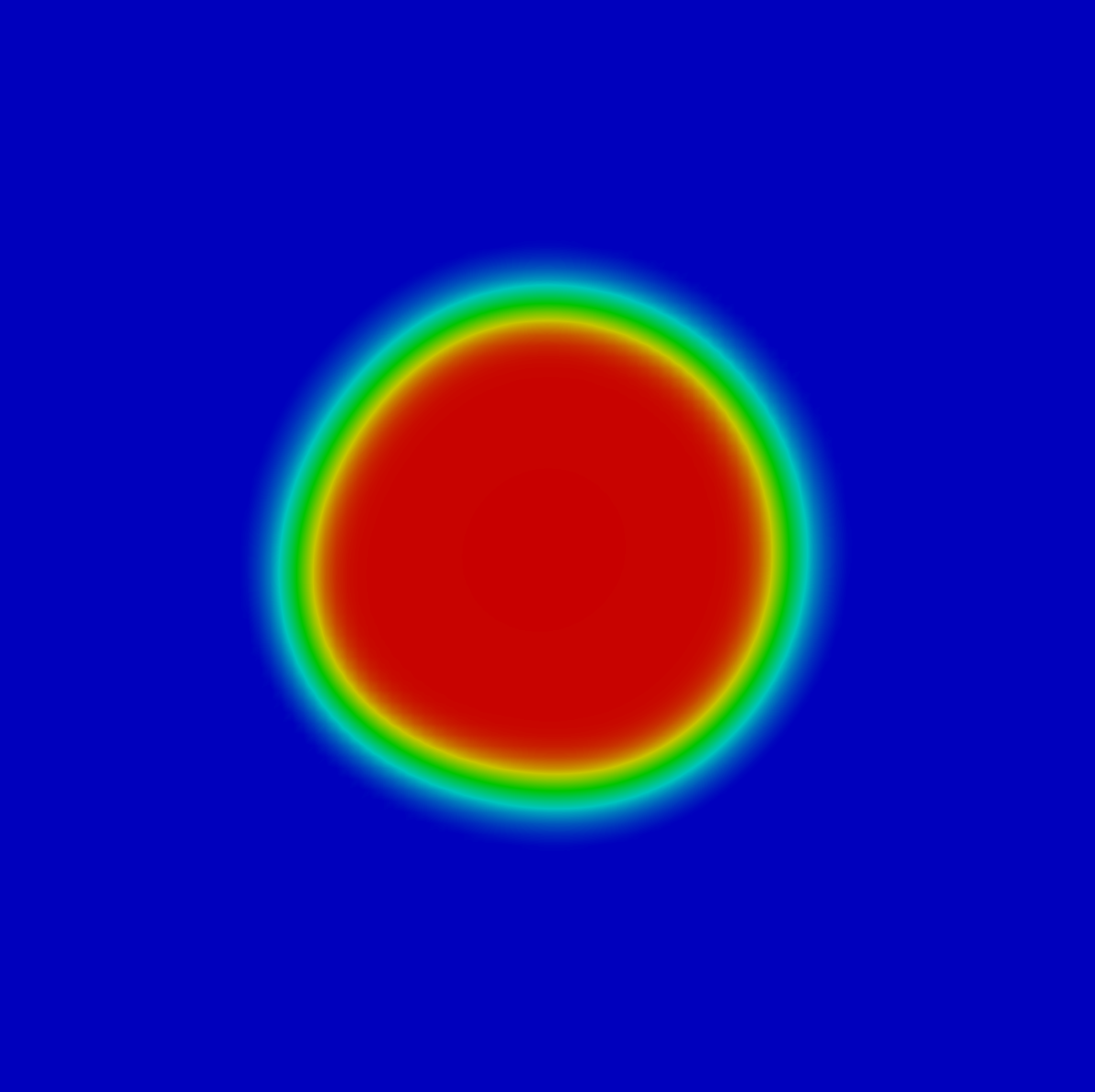}\\[-.1cm]
			\!\!\!\!\!\!$t=1.0$\!\!\!\!\!&
			\includegraphics[width=.22\textwidth]{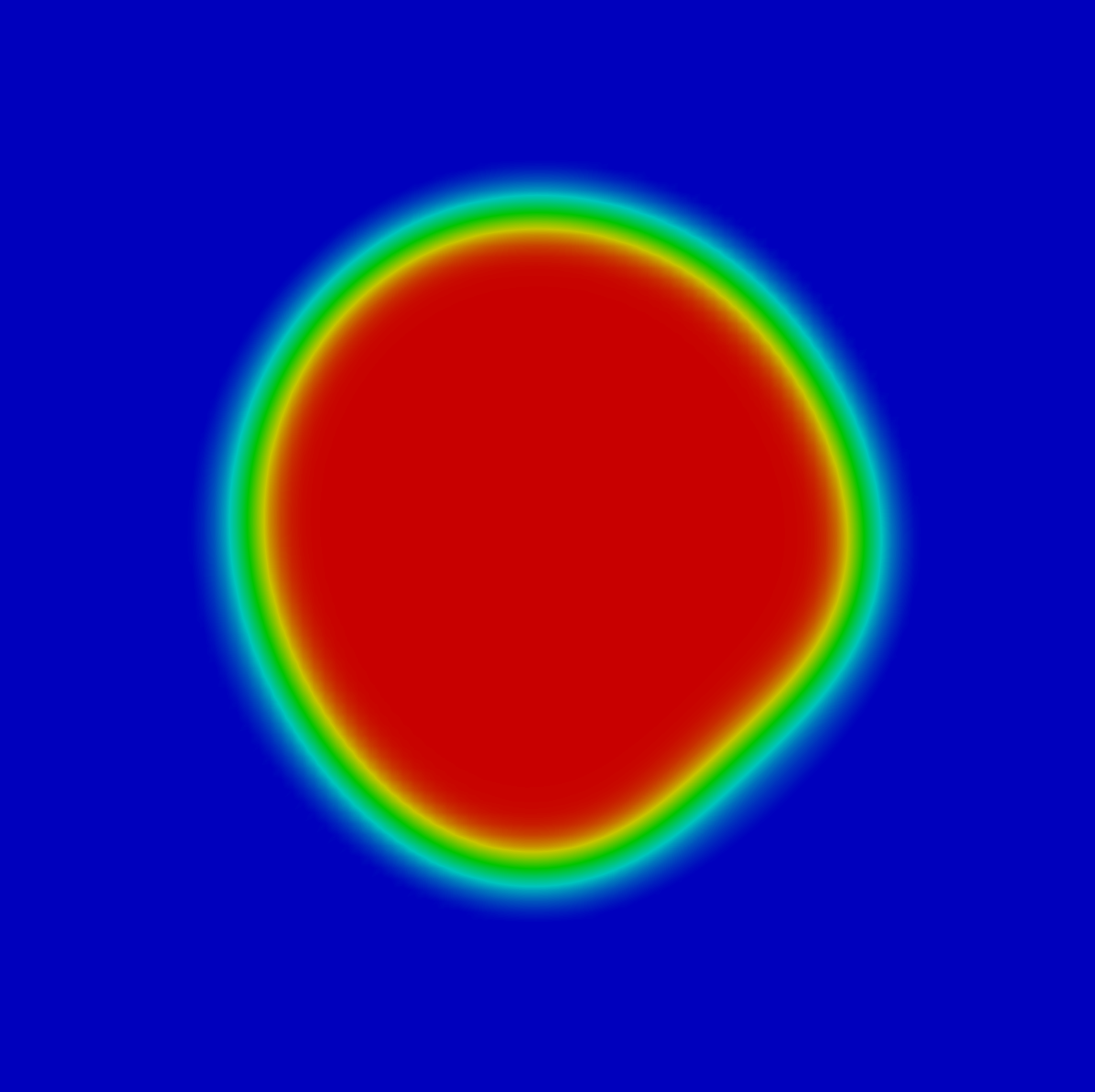}&
			\includegraphics[width=.22\textwidth]{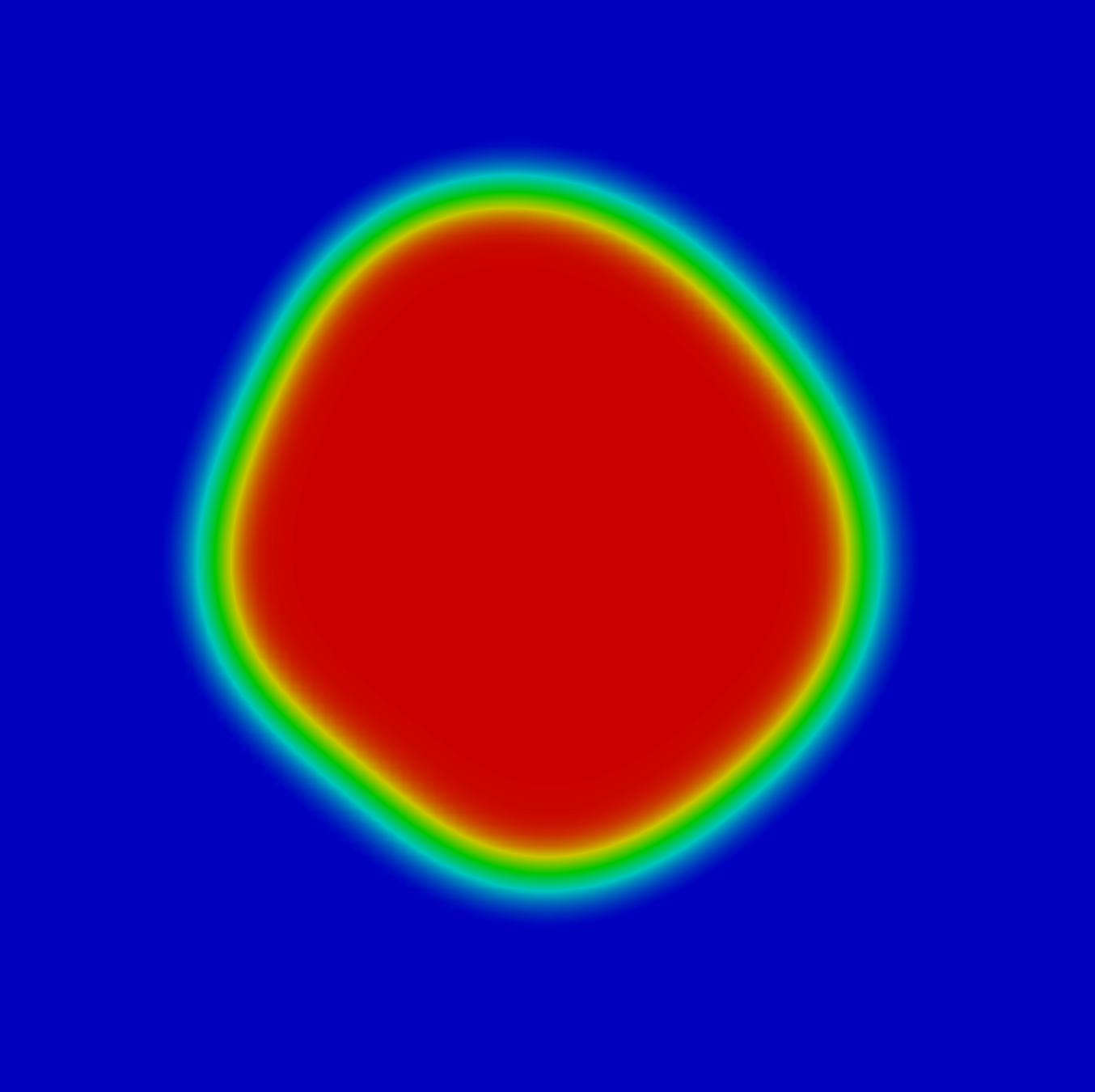}&
			\includegraphics[width=.22\textwidth]{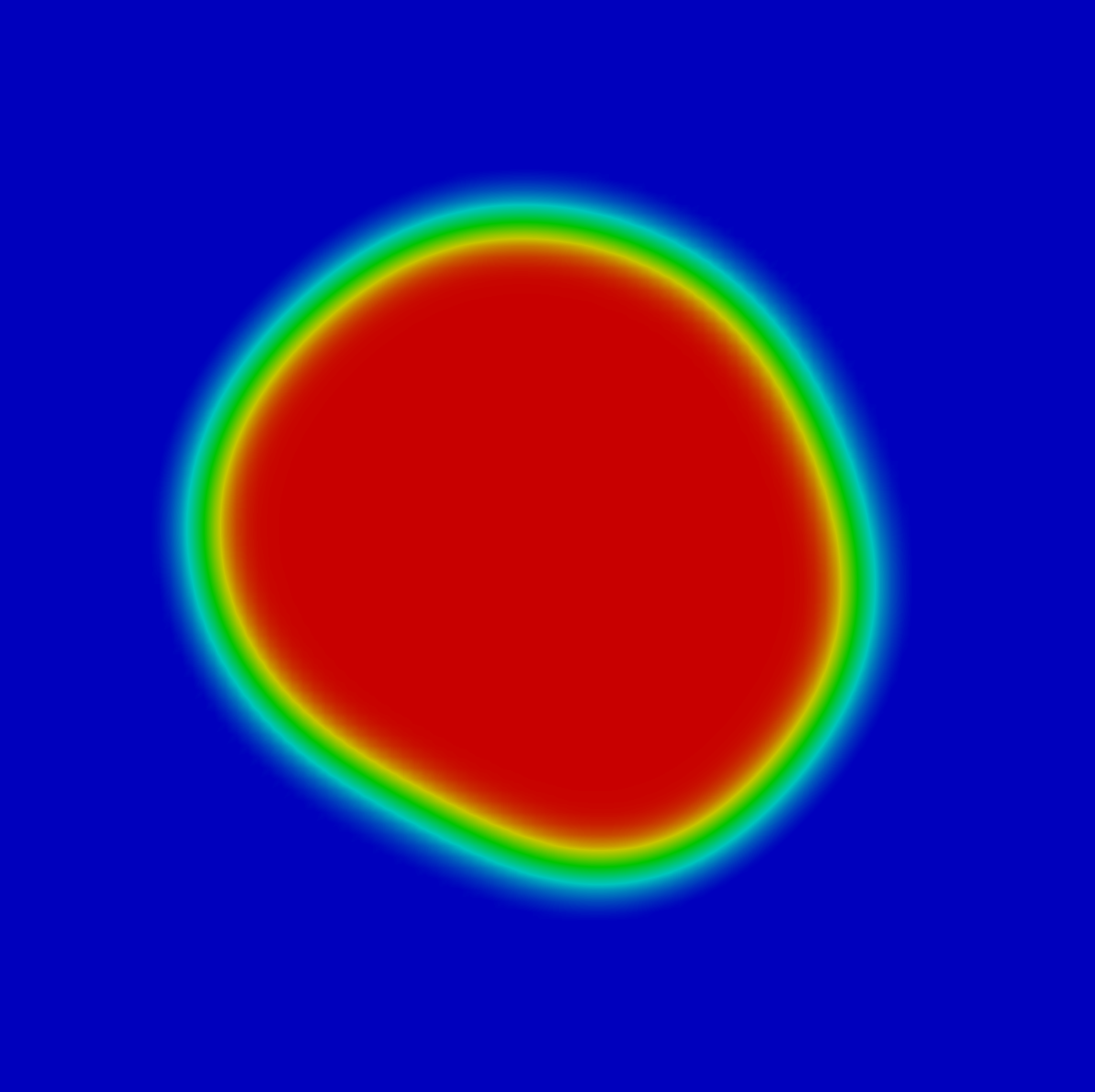}&
			\includegraphics[width=.22\textwidth]{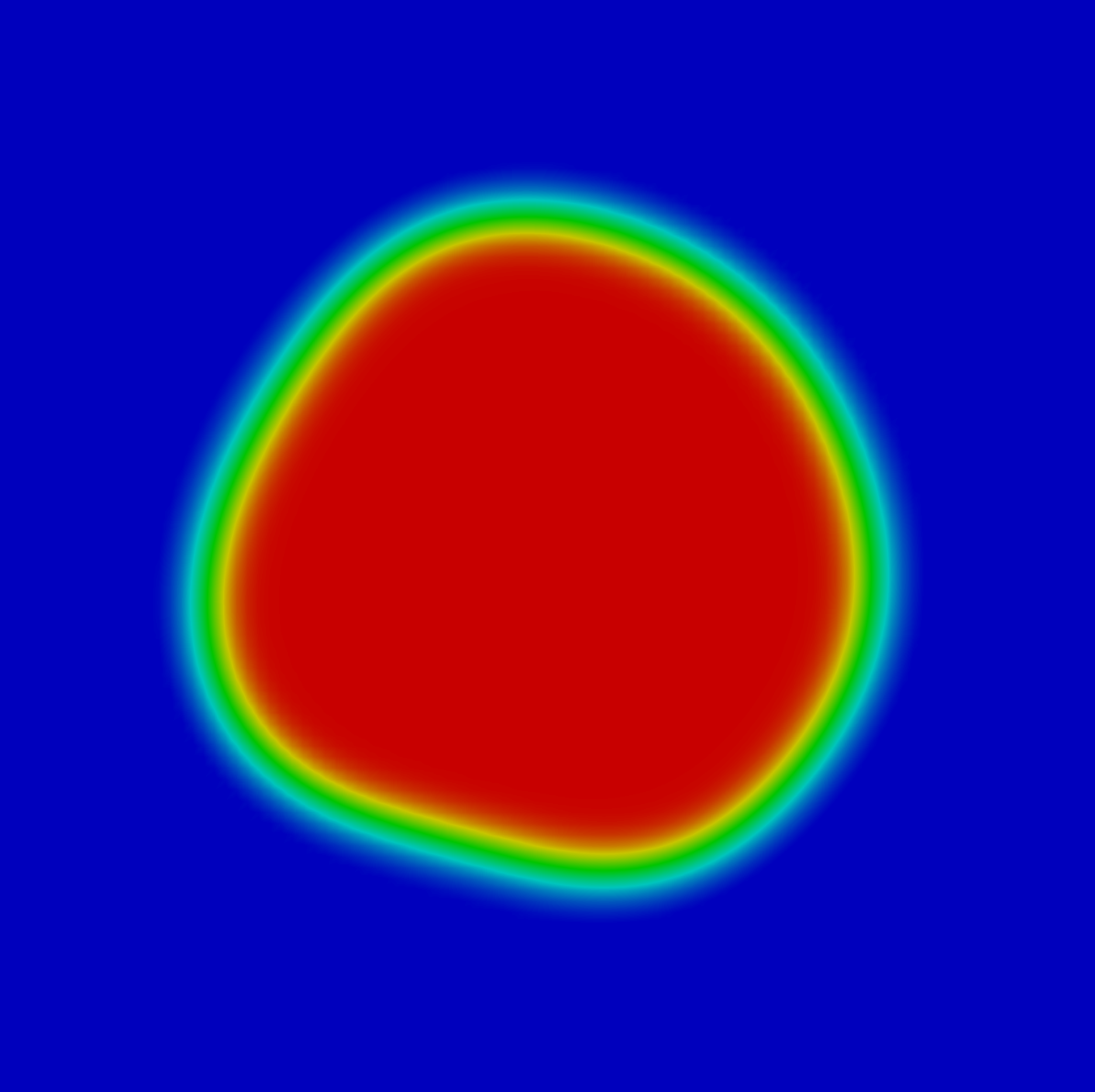}
		\end{tabular} \\
		\hspace{1.5cm}\includegraphics[width=.6\textwidth]{Figures/color1.png}
		\caption{\label{Fig:TumorDifferentSeedLowerNoise}Evolution of the tumor volume fraction $\phi(t,x)$ over time in the domain $D$ at a fixed noise intensity $\nu=0.5$ for four random seeds.}
	\end{center} 
\end{figure}

\cref{Fig:TumorDifferentSeedLowerNoise} explores the evolution of the tumor under conditions of low noise levels while considering different random seeds. Similar to \cref{Fig:TumorDifferentSeedHighNoise}, the simulations are influenced by randomness, but in this case, noise levels are lower. In contrast to the high noise level case, \cref{Fig:TumorDifferentSeedLowerNoise} reveals that, at $t=0.4$, the tumor shapes exhibit some resemblance among different random seeds. While the tumor shapes may slowly drift apart over time, the variability at earlier time steps is not as pronounced as in the higher noise level case. These observations underscore the effect of noise levels on tumor evolution. In low noise scenarios, the tumor shapes at early times exhibit a degree of consistency among different simulations, while high noise levels lead to a wide range of random and distinct shapes from the outset.

In \cref{Fig:TumorContourDifferentSeed}, we present a contour plot that visualizes the impact of different random seeds on tumor evolution. The contour plot shows the results of five distinct samples, with each sample representing a unique simulation outcome at different time steps. This contour plot uses the contour line of 50\% tumor cells to demarcate the evolving boundary of the tumor. What's striking in \cref{Fig:TumorContourDifferentSeed} is the observable divergence in tumor contours over time. Each of the four distinct samples exhibits its own path of tumor growth, highlighting the variability and unpredictability introduced by high noise levels. The contour plot provides a clear visual representation of how the tumor shapes drift apart over time in this high noise level scenario. The initial similarity in contour shapes slowly gives way to distinct and random tumor contours, underscoring the influence of randomness on tumor growth patterns. Unlike the high noise scenario, where tumor contours drift apart, the low noise case reveals the presence of wobbly circular shapes. The contours exhibit a certain consistency among different random seeds, emphasizing that the presence of noise, although present, does not lead to as significant a divergence in tumor shapes as in the high noise case.

\begin{figure}[H] \begin{center}
		\begin{tabular}{cM{.19\textwidth}M{.19\textwidth}M{.19\textwidth}M{.19\textwidth}}
			&\quad $t=0.1$ & \quad $t=0.4$& \quad $t=0.7$&\quad  $t=1.0$ \\
			\!\!\!\!\!\!$\nu=2.5$\!\!\!\!\! &	\includegraphics[width=.22\textwidth]{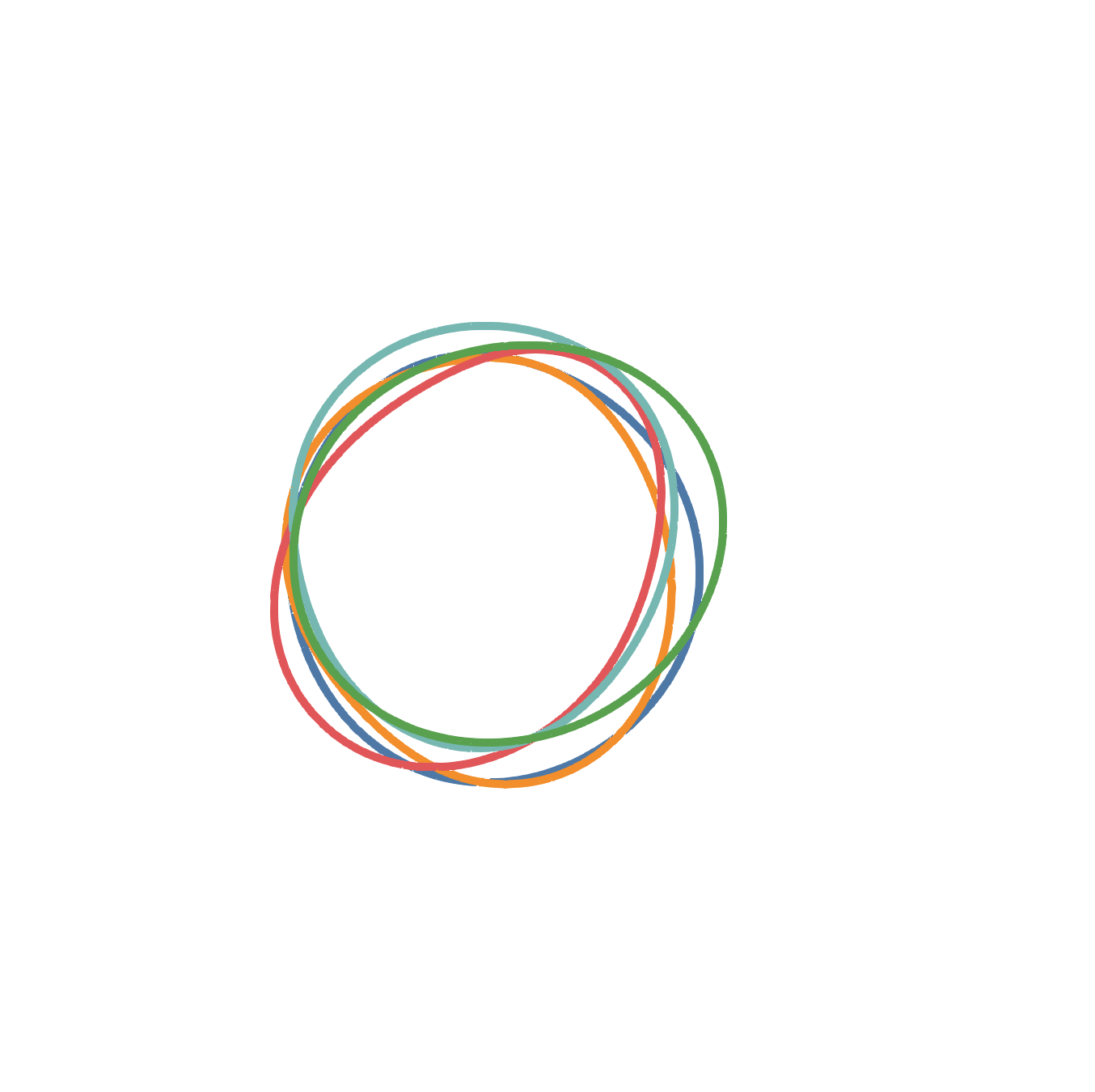}&
			\includegraphics[width=.22\textwidth]{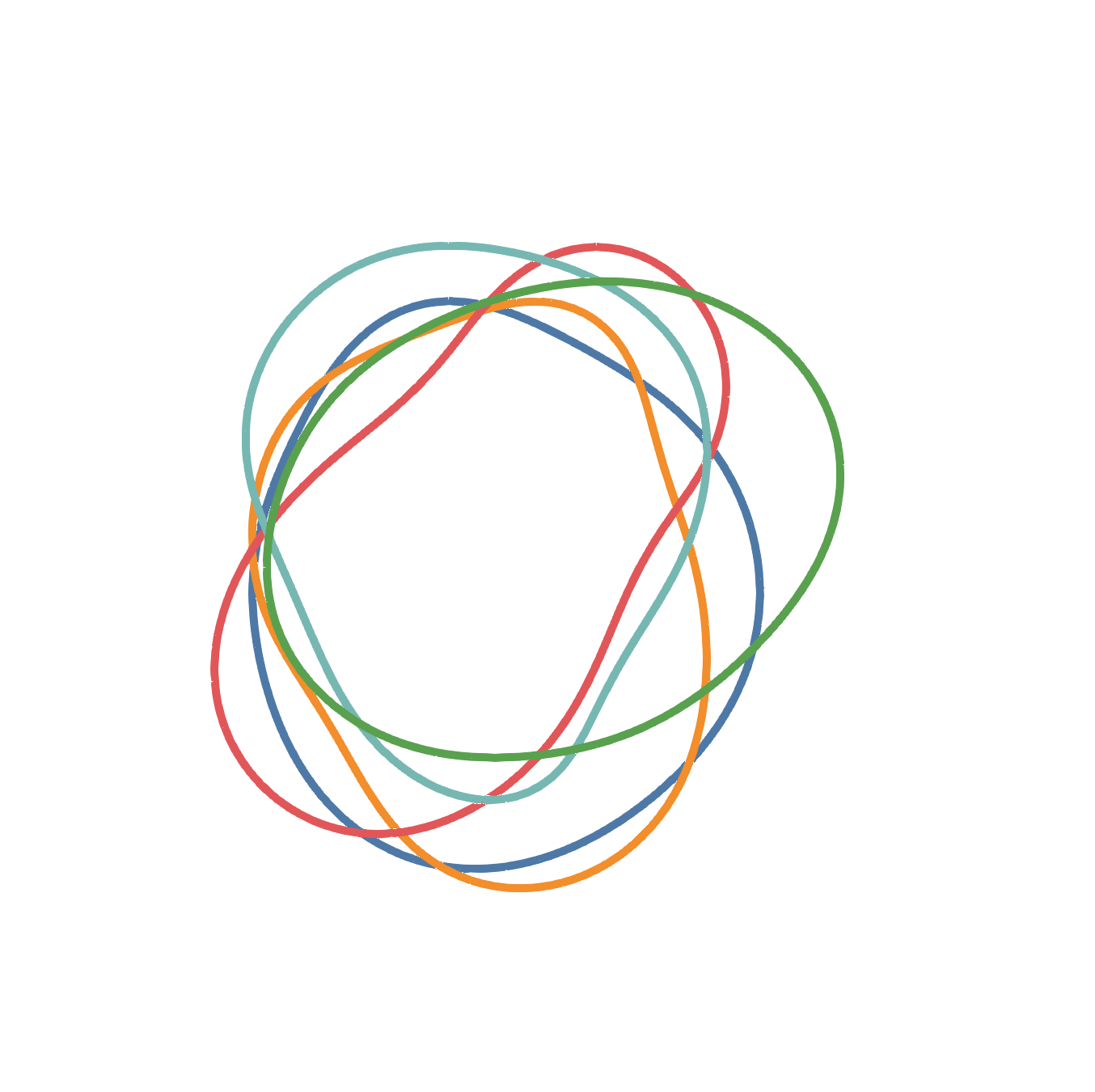}&
			\includegraphics[width=.22\textwidth]{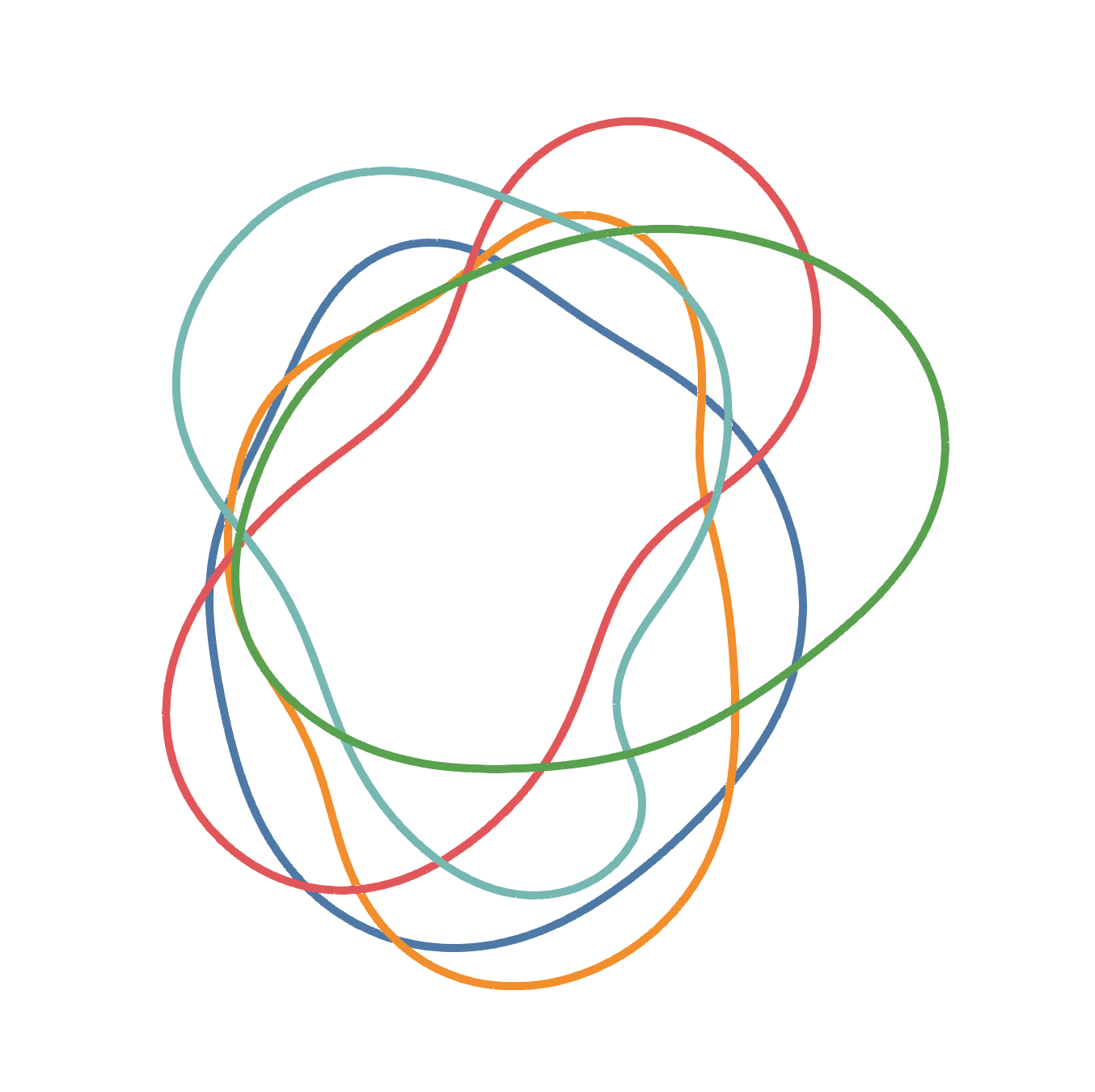}&
			\includegraphics[width=.22\textwidth]{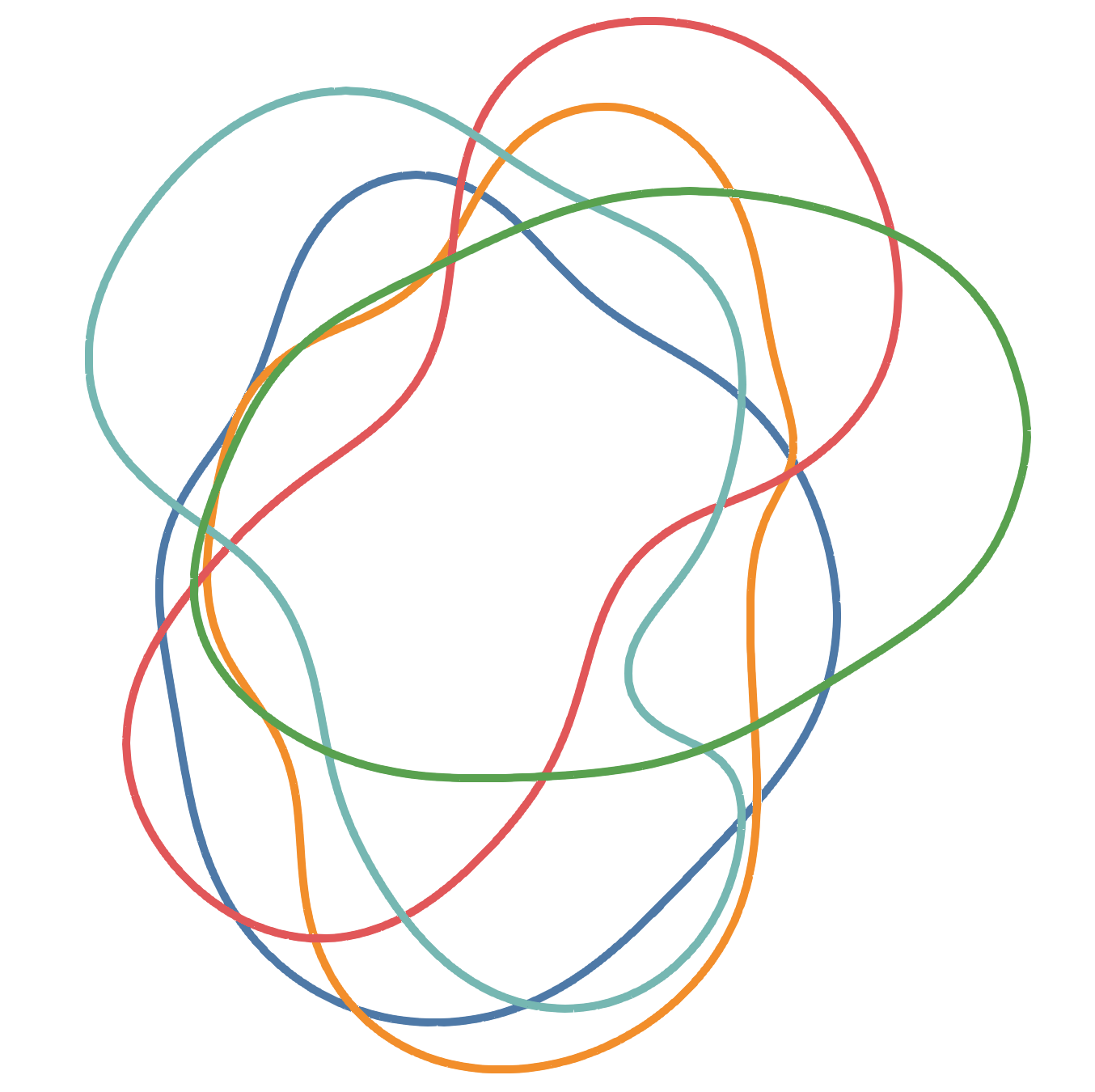} \\
			\!\!\!\!\!\!$\nu=0.5$\!\!\!\!\! & \includegraphics[width=.22\textwidth]{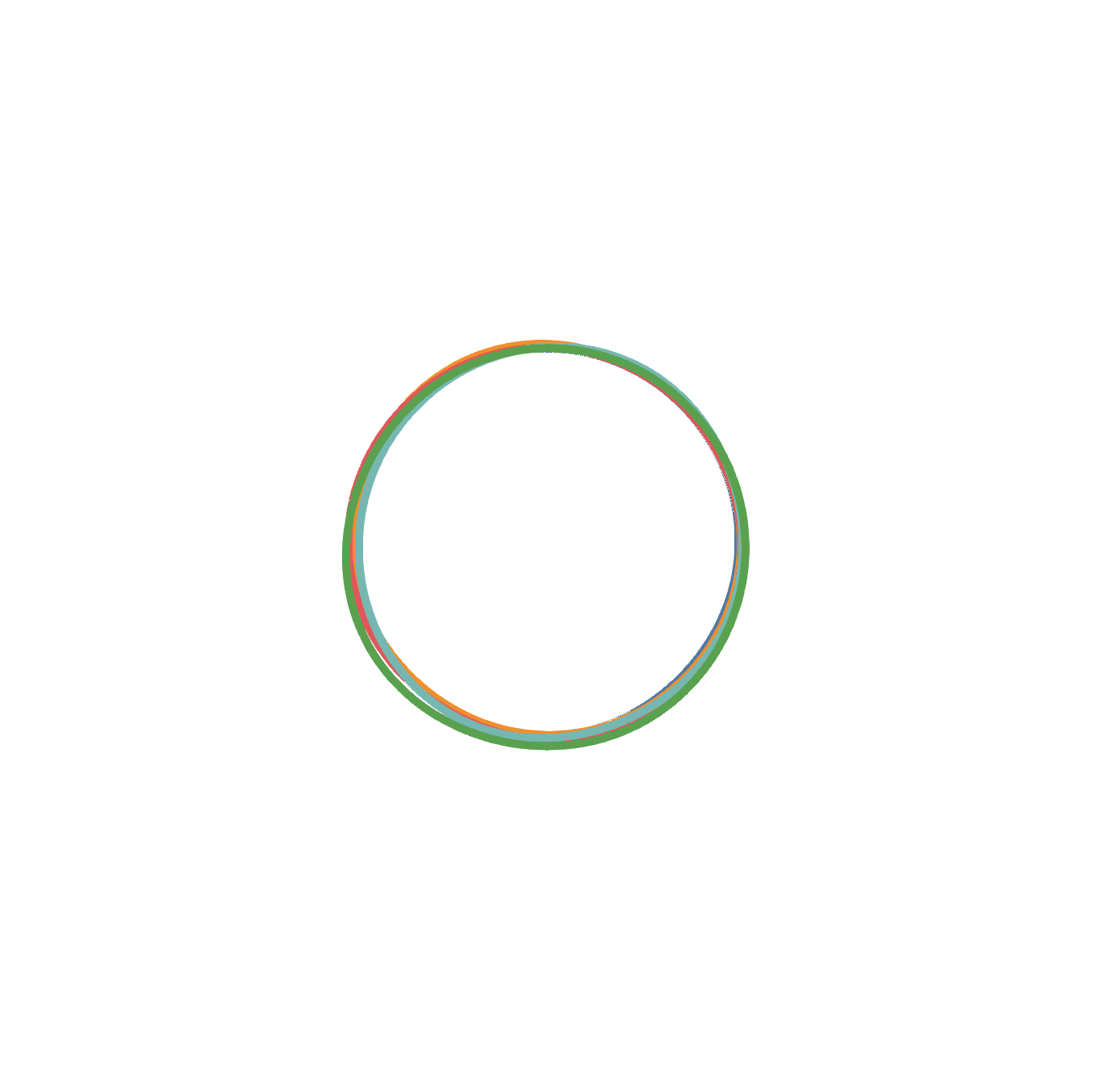}&
			\includegraphics[width=.22\textwidth]{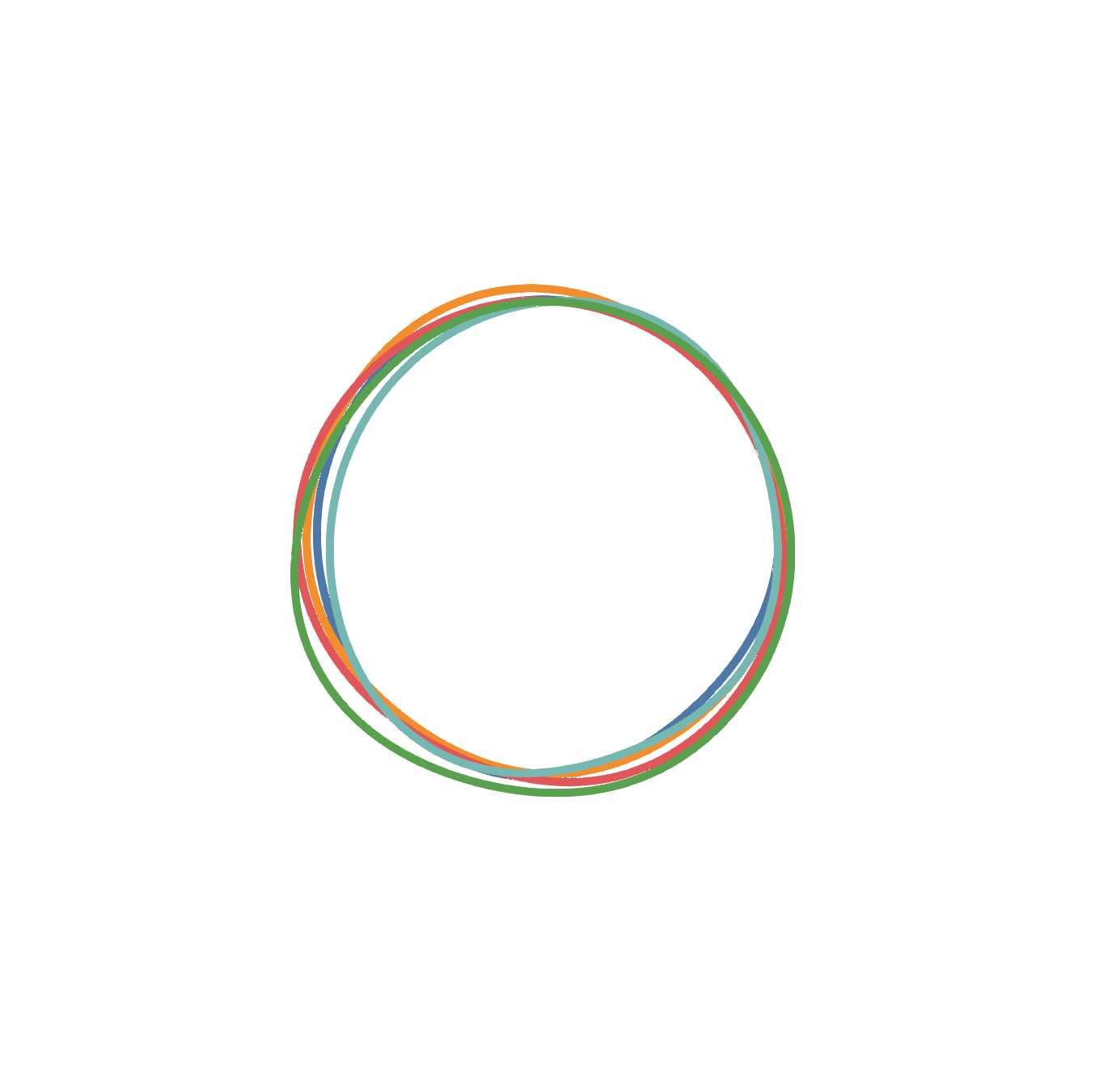}&
			\includegraphics[width=.22\textwidth]{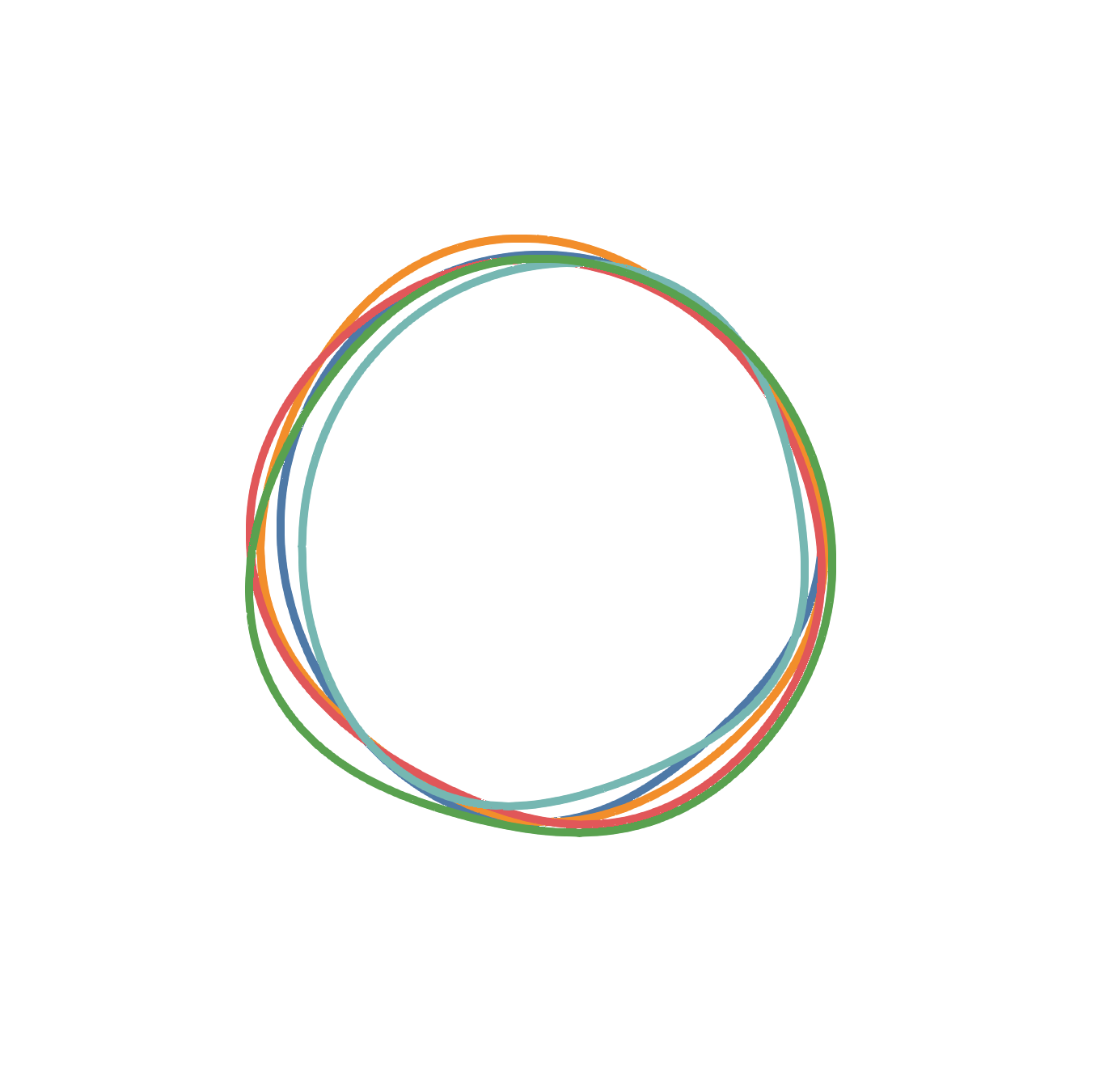}&
			\includegraphics[width=.22\textwidth]{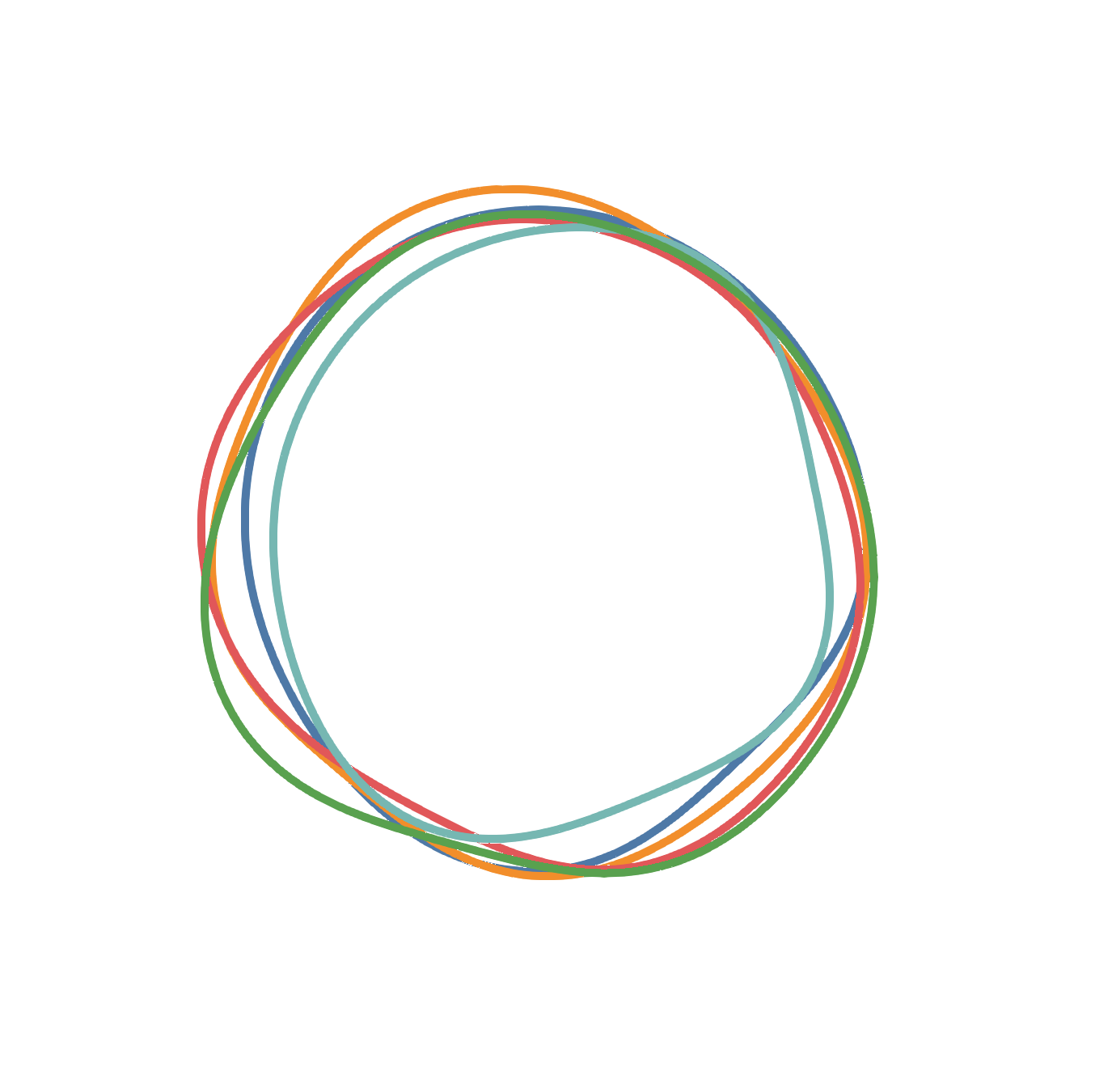}
		\end{tabular} 
		\caption{\label{Fig:TumorContourDifferentSeed}Evolution of the contour lines $\phi(t,x)=0.5$ over time at the fixed noise intensities $\nu\in\{0.5,2.5\}$ for random seeds}
	\end{center} 
\end{figure}

\backmatter

\renewcommand*{\bibfont}{\normalfont\footnotesize}
\setlength{\bibsep}{1pt}
\bibliography{literature.bib}
\end{document}